\documentclass[reqno,12pt]{amsart}

\usepackage{fullpage}
\usepackage[english]{babel}
\usepackage[T1]{fontenc}
\usepackage{graphicx}
\usepackage{amsmath}
\usepackage{amsfonts}
\usepackage{amssymb}
\usepackage{a4wide}

\newtheorem{theo}{Theorem}
\newtheorem*{crit}{Criterion for univariate mirror maps}

\newtheorem*{theo2}{Theorem}

\newtheorem*{propo2}{Landau's criterion}

\newtheorem{lemme}{Lemma}

\theoremstyle{remark}
\newtheorem*{Remarque}{Remark}
\newtheorem*{Remarques}{Remarks}

\numberwithin{equation}{section}

\author{E. DELAYGUE}
\title{Criterion for the integrality of the Taylor coefficients of mirror maps in several variables}
\date{}

\begin{document}
\maketitle

\begin{abstract}
We give a necessary and sufficient condition for the integrality of the Taylor coefficients at the origin of formal power series $q_i({\mathbf z})=z_i\exp(G_i({\mathbf z})/F({\mathbf z}))$, with ${\mathbf z}=(z_1,\dots,z_d)$ and where $F({\mathbf z})$ and $G_i({\mathbf z})+\log(z_i)F({\mathbf z})$, $i=1,\dots,d$ are particular solutions of certain $A$-systems of differential equations. This criterion is based on the analytical properties of Landau's function (which is classically associated with the sequences of factorial ratios) and it generalizes the criterion in the case of one variable presented in ``Critère pour l'intégralité des coefficients de Taylor des applications miroir'' [J.~Reine Angew. Math.]. One of the techniques used to prove this criterion is a generalization of a version of a theorem of Dwork on the formal congruences between formal series, proved by Krattenthaler and Rivoal in ``Multivariate $p$-adic formal congruences and integrality of Taylor coefficients of mirror maps'' [arXiv:0804.3049v3, math.NT]. This criterion involves the integrality of the Taylor coefficients of new univariate mirror maps listed in ``Tables of Calabi--Yau equations'' [arXiv:math/0507430v2, math.AG] by Almkvist, van Enckevort, van Straten and Zudilin. 
\end{abstract}

\section{Introduction}

The mirror maps considered in this article are formal series of $d$ variables $z_i(x_1,\dots,x_d)$, $i=1,\dots,d$, such that the map 
$$
(x_1,\dots,x_d)\mapsto(z_1(x_1,\dots,x_d),\dots,z_d(x_1,\dots,x_d))
$$
is the compositional inverse of the map
$$
(y_1,\dots,y_d)\mapsto(q_1(y_1,\dots,y_d),\dots,q_d(y_1,\dots,y_d)),
$$ 
with, writing ${\mathbf y}=(y_1,\dots,y_d)$, $q_i({\mathbf y})=y_i\exp(G_i({\mathbf y})/F({\mathbf y}))$ for $i=1,\dots,d$ and where $F({\mathbf y})$ and $G_i({\mathbf y})+\log(y_i)F({\mathbf y})$ are particular solutions of a certain $A$-system of linear differential equations. These objects are geometric in nature because the series $F({\mathbf y})$ are $A$-hypergeometric functions (\footnote{The $A$-hypergeometric series are also called GKZ hypergeometric series. See \cite{Stienstra} for an introduction to these series, which generalize the classic hypergeometric series in the multivariate case.}) which can be viewed as the period of certain multi-parameter families of algebraic varieties in a product of weighted projective spaces (see \cite{Hosono} for details). 
\medskip

A classic example of multivariate mirror maps, studied in \cite{Batyrev}, \cite{Stienstra} and \cite{Tanguy 2} is related to the series
\begin{equation}\label{exchange1}
F(z_1,z_2)=\sum_{m,n\geq 0}\frac{(3m+3n)!}{m!^3n!^3}z_1^mz_2^n
\end{equation}
which is solution of the system of differential equations
\[\left\{ 
\begin{array}{clcr}
D_1^3y-z_1\left(3D_1+3D_2+1\right)\left(3D_1+3D_2+2\right)\left(3D_1+3D_2+3\right)y=0,\\
D_2^3y-z_2\left(3D_1+3D_2+1\right)\left(3D_1+3D_2+2\right)\left(3D_1+3D_2+3\right)y=0,
\end{array}
\right.
\]
where $D_1=z_1\frac{d}{dz_1}$ and $D_2=z_2\frac{d}{dz_2}$. We find two other solutions of this system $G_1(z_1,z_2)+\log(z_1)F(z_1,z_2)$ and $G_2(z_1,z_2)+\log(z_2)F(z_1,z_2)$ where
$$
G_1(z_1,z_2)=\sum_{m,n\geq 0}\frac{(3m+3n)!}{m!^3n!^3}(3H_{3m+3n}-3H_m)z_1^mz_2^n
$$
and 
$$
G_2(z_1,z_2)=\sum_{m,n\geq 0}\frac{(3m+3n)!}{m!^3n!^3}(3H_{3m+3n}-3H_n)z_1^mz_2^n.
$$
This set of solutions enables us to define two \textit{canonical coordinates} 
$$
q_1(z_1,z_2)=z_1\exp(G_1(z_1,z_2)/F(z_1,z_2))\quad\textup{and}\quad q_2(z_1,z_2)=z_2\exp(G(z_1,z_2)/F(z_1,z_2)).
$$
The associated \textit{mirror maps} are defined by the formal series $z_1(q_1,q_2)$ and $z_2(q_1,q_2)$ such that the map $(q_1,q_2)\mapsto(z_1(q_1,q_2),z_2(q_1,q_2))$ is the compositional inverse of the map $(z_1,z_2)\mapsto(q_1(z_1,z_2),q_2(z_1,z_2))$.

According to the Corollary $1$ from \cite{Tanguy 2}, the series $q_1(z_1,z_2)$, $q_2(z_1,z_2)$, $z_1(q_1,q_2)$ and $z_2(q_1,q_2)$ have integral Taylor coefficients. 
\medskip

Mirror maps are of interest in Mathematical Physics and Algebraic Geometry. Particularly, within Mirror Symmetry Theory, it has been observed that the Taylor coefficients of mirror maps are integers. This surprising observation has led to the study of these objects within Number Theory, which has led to its proof in many cases (see further down in the introduction). The aim of this article is to establish a necessary and sufficient condition for the integrality of all the Taylor coefficients of mirror maps defined by ratios of factorials of linear forms.

\subsection{Definition of mirror maps}

In order to define the mirror maps involved in this article, we introduce some standard multi-index notation, which we use throughout the article. Namely, given a positive integer $d$, $k\in\{1,\dots,d\}$ and vectors ${\mathbf m}:=(m_1,\dots,m_d)$ and ${\mathbf n}:=(n_1,\dots,n_d)$ in $\mathbb{R}^d$, we write ${\mathbf m}\cdot{\mathbf n}$ for the scalar product $m_1n_1+\cdots+m_dn_d$ and ${\mathbf m}^{(k)}$ for $m_k$. We write ${\mathbf m}\geq{\mathbf n}$ if and only if $m_i\geq n_i$ for all $i\in\{1,\dots,d\}$. In addition, if ${\mathbf z}:=(z_1,\dots,z_d)$ is a vector of variables and if ${\mathbf n}:=(n_1,\dots,n_d)\in\mathbb{Z}^d$, then we write ${\mathbf z}^{{\mathbf n}}$ for the product $z_1^{n_1}\cdots z_d^{n_d}$. Finally, we write ${\mathbf 0}$ for the vector $(0,\dots,0)\in\mathbb{Z}^d$.

Given two sequences of vectors in $\mathbb{N}^d$ $e:=({\mathbf e}_1,\dots,{\mathbf e}_{q_1})$ and $f:=({\mathbf f}_1,\dots,{\mathbf f}_{q_2})$ , we write $|e|:=\sum_{i=1}^{q_1}{\mathbf e}_i$ and $|f|:=\sum_{i=1}^{q_2}{\mathbf f}_i\in\mathbb{N}^d$ so that, for all $k\in\{1,\dots,d\}$, we have $|e|^{(k)}=\sum_{i=1}^{q_1}\mathbf{e}_i^{(k)}$ and $|f|^{(k)}=\sum_{i=1}^{q_2}\mathbf{f}_i^{(k)}$. For all $\mathbf{n}\in\mathbb{N}^d$, we write
$$
\mathcal{Q}_{e,f}({\mathbf n}):=\frac{({\mathbf e}_1\cdot{\mathbf n})!\cdots({\mathbf e}_{q_1}\cdot{\mathbf n})!}{({\mathbf f}_1\cdot{\mathbf n})!\cdots({\mathbf f}_{q_2}\cdot{\mathbf n})!}.
$$ 
We define the formal series
$$
F_{e,f}({\mathbf z}):=\sum_{{\mathbf n}\geq {\mathbf 0}}\frac{({\mathbf e}_1\cdot{\mathbf n})!\cdots({\mathbf e}_{q_1}\cdot{\mathbf n})!}{({\mathbf f}_1\cdot{\mathbf n})!\cdots({\mathbf f}_{q_2}\cdot{\mathbf n})!}{\mathbf z}^{{\mathbf n}}
$$
and
\begin{equation}\label{definitionG}
G_{e,f,k}({\mathbf z}):=\sum_{{\mathbf n}\geq{\mathbf 0}}\frac{({\mathbf e}_1\cdot{\mathbf n})!\cdots({\mathbf e}_{q_1}\cdot{\mathbf n})!}{({\mathbf f}_1\cdot{\mathbf n})!\cdots({\mathbf f}_{q_2}\cdot{\mathbf n})!}\left(\sum_{i=1}^{q_1}{\mathbf e}_i^{(k)}H_{{\mathbf e}_i\cdot{\mathbf n}}-\sum_{j=1}^{q_2}{\mathbf f}_j^{(k)}H_{{\mathbf f}_j\cdot{\mathbf n}}\right){\mathbf z}^{{\mathbf n}},
\end{equation}
where $k\in\{1,\dots,d\}$ and, for all $m\in\mathbb{N}$, $H_m:=\sum_{i=1}^{m}\frac{1}{i}$ is the $m$-th harmonic number. The series $F_{e,f}({\mathbf z})$ is a $A$-hypergeometric series and is therefore a solution of a $A$-system of linear differential equations. In some cases, we find $d$ additional solutions of this system together with at most logarithmic singularities at the origin, the $G_{e,f,k}({\mathbf z})+\log(z_k)F({\mathbf z})$ for $k\in\{1,\dots,d\}$. 

In the context of mirror symmetry, when $|e|=|f|$, the $d$ functions 
$$
q_{e,f,k}({\mathbf z}):=z_k\exp(G_{e,f,k}({\mathbf z})/F_{e,f}({\mathbf z})),\,k\in\{1,\dots,d\},
$$ 
are \textit{canonical coordinates}. The compositional inverse of the map 
$$
{\mathbf z}\mapsto(q_{e,f,1}({\mathbf z}),\dots,q_{e,f,d}({\mathbf z}))
$$
defines the vector $(z_{e,f,1}({\mathbf q}),\dots,z_{e,f,d}({\mathbf q}))$ of \textit{mirror maps}.

The aim of this article is to establish a necessary and sufficient condition for the integrality of the coefficients of the $d$ mirror maps $z_{e,f,k}({\mathbf q})$, that is, to determine under which conditions, for all $k\in\{1,\dots,d\}$, we have $z_{e,f,k}({\mathbf q})\in\mathbb{Z}[[{\mathbf q}]]$. In the context of Number Theory of this article, the mirror map $z_{e,f,k}({\mathbf q})$ and the corresponding canonical coordinate $q_{e,f,k}({\mathbf z})$ play strictly the same role because, for all $k\in\{1,\dots,d\}$, we have $q_{e,f,k}({\mathbf z})\in z_k\mathbb{Z}[[{\mathbf z}]]$ if and only if, for all $k\in\{1,\dots,d\}$, we have $z_{e,f,k}({\mathbf q})\in q_k\mathbb{Z}[[{\mathbf q}]]$ (see \cite[Partie 1.2]{Tanguy 2}). Therefore, we shall formulate the criterion exclusively for canonical coordinate but it also holds for the corresponding mirror maps.

\subsection{Statement of the criterion}\label{\'Enoncé du critère}

Before stating the criterion for the integrality of the Taylor coefficients of $q_{e,f,k}({\mathbf z})$, we recall the definition of \textit{Landau's function} associated with a ratio of factorials of linear forms. Given two sequences of vectors in $\mathbb{N}^d$ $e:=({\mathbf e}_1,\dots,{\mathbf e}_{q_1})$ and $f:=({\mathbf f}_1,\dots,{\mathbf f}_{q_2})$ , we write $\Delta_{e,f}$ the Landau's function associated with $\mathcal{Q}_{e,f}$, which is defined, for all ${\mathbf x}\in\mathbb{R}^d$, by
$$
\Delta_{e,f}({\mathbf x}):=\sum_{i=1}^{q_1}\lfloor{\mathbf e}_i\cdot{\mathbf x}\rfloor-\sum_{j=1}^{q_2}\lfloor{\mathbf f}_j\cdot{\mathbf x}\rfloor,
$$
where $\lfloor\cdot\rfloor$ denotes the floor function. We also write $\{\cdot\}$ for the fractional part function. We still write $\lfloor\cdot\rfloor$, respectively $\{\cdot\}$, for the function defined, for all ${\mathbf x}=(x_1,\cdots,x_d)\in\mathbb{R}^d$, by $\lfloor{\mathbf x}\rfloor:=(\lfloor x_1\rfloor,\cdots,\lfloor x_d\rfloor)$, respectively by $\{{\mathbf x}\}:=(\{x_1\},\cdots,\{x_d\})$. For all ${\mathbf c}\in\mathbb{N}^d$, we have $\lfloor{\mathbf c}\cdot{\mathbf x}\rfloor=\lfloor{\mathbf c}\cdot\{{\mathbf x}\}\rfloor+{\mathbf c}\cdot\lfloor{\mathbf x}\rfloor$ and therefore $\Delta_{e,f}({\mathbf x})=\Delta_{e,f}(\{{\mathbf x}\})+(|e|-|f|)\cdot\lfloor{\mathbf x}\rfloor$. So, we have $|e|=|f|$ if and only if $\Delta_{e,f}$ is $1$-periodic in each of its variables. We write $\mathcal{D}_{e,f}$ for the semi-algebraic set of all ${\mathbf x}\in [0,1[^d$ such that there exists ${\mathbf d}\in\{{\mathbf e}_1,\cdots,{\mathbf e}_{q_1},{\mathbf f}_1,\cdots,{\mathbf f}_{q_2}\}$ verifying ${\mathbf d}\cdot{\mathbf x}\geq 1$. The set $[0,1[^d\setminus\mathcal{D}_{e,f}$ is nonempty and the function $\Delta_{e,f}$ vanishes on $[0,1[^d\setminus\mathcal{D}_{e,f}$. The following proposition shows that the Landau's function provides a characterization of the sequences $e$ and $f$ such that, for all ${\mathbf n}\in\mathbb{N}^d$, $\mathcal{Q}_{e,f}({\mathbf n})$ is an integer.

\begin{propo2}\label{critère Landau}
Let $e$ and $f$ be two sequences of vectors in $\mathbb{N}^d$. We have the following dichotomy.
\begin{itemize}
\item[$(i)$] If, for all ${\mathbf x}\in [0,1]^d$, we have $\Delta_{e,f}({\mathbf x})\geq 0$, then, for all ${\mathbf n}\in\mathbb{N}^d$, we have $\mathcal{Q}_{e,f}({\mathbf n})\in\mathbb{N}$.
\item[$(ii)$] If there exists ${\mathbf x}\in [0,1]^d$ such that $\Delta_{e,f}({\mathbf x})\leq -1$, then there are only finitely many prime numbers $p$ such that all terms of the family $\mathcal{Q}_{e,f}$ are in $\mathbb{Z}_p$.
\end{itemize}
\end{propo2}

\begin{Remarque}
Assertion $(i)$ is a result of Landau from \cite{Landau}: he has proved that it is in fact a necessary and sufficient condition. We prove Landau's criterion assertion $(ii)$ in Section \ref{section demLandau}.
\end{Remarque} 

In literature, one can distinguish several results proving the integrality of the Taylor coefficients of univariate mirror maps (\textit{i.e.} $d=1$) when $|e|=|f|$. One can find them, in an increasing order of generality, in \cite{Lian 2}, \cite{Zudilin}, \cite{Tanguy} and \cite{Delaygue}. Refer to the introduction from \cite{Delaygue} for a detailed statement of all these results. In the univariate case, the most general result builds up a criterion for the integrality of the Taylor coefficients of mirror maps defined by sequences of ratios of factorials. According to the notations of this article, it reads as follows:

\begin{crit}[Theorem 1 from \cite{Delaygue}] 
Let $e$ and $f$ be two disjoint sequences of positive integers such that $\mathcal{Q}_{e,f}$ is a sequence of integers (which is equivalent to $\Delta_{e,f}\geq 0$ on $[0,1]$) and which satisfy $|e|=|f|$. Then, we have the following dichotomy.
\begin{itemize}
\item[$(i)$] If, for all $x\in\mathcal{D}_{e,f}$, we have $\Delta_{e,f}(x)\geq 1$, then $q_{e,f,1}(z)\in z\mathbb{Z}[[z]]$.
\item[$(ii)$] If there exists $x\in\mathcal{D}_{e,f}$ such that $\Delta_{e,f}(x)=0$, then there are only finitely many prime numbers $p$ such that $q_{e,f,1}(z)\in z\mathbb{Z}_p[[z]]$.
\end{itemize}
\end{crit}

In the multivariate case, Krattenthaler and Rivoal proved in \cite{Tanguy 2} the integrality of the Taylor coefficients of mirror maps belonging to large infinite families. In order to state this result, for all $k\in\{1,\dots,d\}$, we write ${\mathbf 1}_k$ for the vector in $\mathbb{N}^d$, all coordinates of which equal to zero except the $k$-th which is equal to $1$.

\begin{theo2}[Corollary 1 from \cite{Tanguy 2}]
Let $e$ and $f$ be two sequences of vectors in $\mathbb{N}^d$ verifying $|e|=|f|$ and such that $f$ is only composed of vectors of the form ${\mathbf 1}_k$ with $k\in\{1,\dots,d\}$. Then, for all $k\in\{1,\dots,d\}$, we have $q_{e,f,k}({\mathbf z})\in z_k\mathbb{Z}[[{\mathbf z}]]$.
\end{theo2}

The purpose of this article is to prove the following theorems, which provide a characterization of the multivariate mirror maps, associated with integral ratios of factorials of linear forms and all the Taylor coefficients of which are integers. We prove in Section \ref{partie comp ant} that they contain the results of other authors who worked on this subject previously. First, we consider the case $|e|=|f|$ and then we state the results when there exists $k\in\{1,\dots,d\}$ such that $|e|^{(k)}>|f|^{(k)}$. When there exists $k\in\{1,\dots,d\}$ such that $|e|^{(k)}<|f|^{(k)}$, the family $\mathcal{Q}_{e,f}$ has a term that is not an integer and the question of the integrality of the Taylor coefficients of $q_{e,f,k}(\mathbf{z})$ is still open.

\begin{theo}\label{critère}
Let $e$ and $f$ be two disjoint sequences of nonzero vectors in $\mathbb{N}^d$ such that $\mathcal{Q}_{e,f}$ is a family of integers (equivalent to $\Delta_{e,f}\geq 0$ on $[0,1]^d$) and which satisfy $|e|=|f|$. Then we have the following dichotomy.
\begin{itemize}
\item[$(i)$] If, for all ${\mathbf x}\in\mathcal{D}_{e,f}$, we have $\Delta_{e,f}({\mathbf x})\geq 1$, then, for all $k\in\{1,\dots,d\}$, we have $q_{e,f,k}({\mathbf z})\in z_k\mathbb{Z}[[{\mathbf z}]]$.
\item[$(ii)$] If there exists ${\mathbf x}\in\mathcal{D}_{e,f}$ such that $\Delta_{e,f}({\mathbf x})=0$, then there exists $k\in\{1,\dots,d\}$ such that there are only finitely many prime numbers $p$ such that $q_{e,f,k}({\mathbf z})\in z_k\mathbb{Z}_p[[{\mathbf z}]]$.
\end{itemize}
\end{theo}

\begin{Remarques}
\begin{itemize}
\item Note the similarity between Landau's criterion and Theorem \ref{critère}.
\item We assume that the terms of the sequences $e$ and $f$ are nonzero and that these sequences are disjoint in order to rule out the possibility that $\Delta_{e,f}$ vanish identically, which corresponds to the formal series $F_{e,f}({\mathbf z})=(1-z_1)^{-1}\cdots(1-z_d)^{-1}$, $G_{e,f,k}({\mathbf z})=0$ and $q_{e,f,k}({\mathbf z})=z_k$.
\item Assertion $(ii)$ of Theorem \ref{critère} is optimal as, if $\Delta_{e,f}$ vanishes on $\mathcal{D}_{e,f}$ and if $d\geq 2$, then there may exist $k\in\{1,\dots,d\}$ such that $q_{e,f,k}({\mathbf z})\in z_k\mathbb{Z}[[{\mathbf z}]]$. Indeed, if one chooses $d=2$, $e=((3,0))$ and $f=((2,0),(1,0))$. Then we have $\mathcal{D}_{e,f}=\{(x_1,x_2)\in[0,1[^2\,:\,x_1\geq 1/3\}$, $\Delta_{e,f}((1/2,0))=0$ and $q_{e,f,2}({\mathbf z})=z_2$.  
\item Theorem \ref{critère} generalizes the criterion for univariate mirror maps and Corollary 1 from \cite{Tanguy 2} (see Section \ref{partie comp ant}). 
\end{itemize}
\end{Remarques}

We will now state a criterion for the integrality of the Taylor coefficients of \textit{mirror-type} maps $q_{{\mathbf L},e,f}$ defined, for all ${\mathbf L}\in\mathbb{N}^d$, by $q_{{\mathbf L},e,f}({\mathbf z}):=\exp(G_{{\mathbf L},e,f}({\mathbf z})/F_{e,f}({\mathbf z}))$, where $G_{{\mathbf L},e,f}$ is the formal power series 
\begin{equation}\label{definitionG L}
G_{{\mathbf L},e,f}({\mathbf z}):=\sum_{{\mathbf n}\geq{\mathbf 0}}\frac{({\mathbf e}_1\cdot{\mathbf n})!\cdots({\mathbf e}_{q_1}\cdot{\mathbf n})!}{({\mathbf f}_1\cdot{\mathbf n})!\cdots({\mathbf f}_{q_2}\cdot{\mathbf n})!}H_{{\mathbf L}\cdot{\mathbf n}}\,{\mathbf z}^{{\mathbf n}}.
\end{equation}

We write $\mathcal{E}_{e,f}$ for the set of all ${\mathbf L}\in\mathbb{N}^d\setminus\{{\mathbf 0}\}$ such that there is a ${\mathbf d}\in\{{\mathbf e}_1,\dots,{\mathbf e}_{q_1},{\mathbf f}_1,\dots,{\mathbf f}_{q_2}\}$ satisfying ${\mathbf L}\leq{\mathbf d}$. We have $q_{{\mathbf L},e,f}({\mathbf z})\in 1+\sum_{j=1}^dz_j\mathbb{Q}[[{\mathbf z}]]$ and 
\begin{equation}\label{retouche}
z_k^{-1}q_{e,f,k}({\mathbf z})=\left(\prod_{i=1}^{q_1}\big(q_{{\mathbf e}_i,e,f}({\mathbf z})\big)^{{\mathbf e}_i^{(k)}}\right)/\left(\prod_{j=1}^{q_2}\big(q_{{\mathbf f}_j,e,f}({\mathbf z})\big)^{{\mathbf f}_j^{(k)}}\right),
\end{equation}
so that if, for all ${\mathbf L}\in\mathcal{E}_{e,f}$, we have $q_{{\mathbf L},e,f}({\mathbf z})\in\mathbb{Z}[[\mathbf{z}]]$, then, for all $k\in\{1,\dots,d\}$, we have $q_{e,f,k}(\mathbf{z})\in z_k\mathbb{Z}[[{\mathbf z}]]$. Thus, assertion $(i)$ of Theorem \ref{critère2} implies assertion $(i)$ of Theorem~\ref{critère}. Assertion $(ii)$ of Theorem \ref{critère2} adds details to assertion $(ii)$ of Theorem \ref{critère}. To be more precise, it proves that there exists $k\in\{1,\dots,d\}$ such that $q_{e,f,k}(\mathbf{z})\notin z_k\mathbb{Z}[[\mathbf{z}]]$ and that all the mirror-type maps indeed involved in \eqref{retouche} have at least one Taylor coefficient which is not an integer. Thus Theorem \ref{critère} can be seen as a corollary of Theorem \ref{critère2}.

\begin{theo}\label{critère2}
Let $e$ and $f$ be two disjoint sequences of nonzero vectors in $\mathbb{N}^d$ such that $\mathcal{Q}_{e,f}$ is a family of integers (which is equivalent to $\Delta_{e,f}\geq 0$ on $[0,1]^d$) and which satisfy $|e|=|f|$. Then we have the following dichotomy.
\begin{itemize}
\item[$(i)$] If, for all $\mathbf{x}\in\mathcal{D}_{e,f}$, we have $\Delta_{e,f}({\mathbf x})\geq 1$, then, for all ${\mathbf L}\in\mathcal{E}_{e,f}$, we have $q_{{\mathbf L},e,f}({\mathbf z})\in\mathbb{Z}[[{\mathbf z}]]$.
\item[$(ii)$] If there exists ${\mathbf x}\in\mathcal{D}_{e,f}$ such that $\Delta_{e,f}({\mathbf x})=0$, then there exists $k\in\{1,\dots,d\}$ such that, if ${\mathbf L}\in\mathcal{E}_{e,f}$ verifies ${\mathbf L}^{(k)}\geq 1$, then there are only finitely many prime numbers $p$ such that $q_{{\mathbf L},e,f}({\mathbf z})\in\mathbb{Z}_p[[{\mathbf z}]]$. Furthermore, there are only finitely many prime numbers $p$ such that $q_{e,f,k}(\mathbf{z})\in z_k\mathbb{Z}_p[[\mathbf{z}]]$.
\end{itemize}
\end{theo}

Theorem \ref{critère2} generalizes Theorem 2 from \cite{Delaygue} and Theorem 2 from \cite{Tanguy 2} (see Section \ref{partie comp ant}). If there exists $k\in\{1,\dots,d\}$ such that $|e|^{(k)}>|f|^{(k)}$, we have the following theorem which generalizes Theorem 3 from \cite{Delaygue}.

\begin{theo}\label{cas e>f}
Let $e$ and $f$ be two disjoint sequences of nonzero vectors in $\mathbb{N}^d$ such that $\mathcal{Q}_{e,f}$ is a family of integers (which is equivalent to $\Delta_{e,f}\geq 0$ on $[0,1]^d$) and such that there exists $k\in\{1,\dots,d\}$ verifying $|e|^{(k)}>|f|^{(k)}$. Then,
\begin{itemize}
\item[$(a)$] there are only finitely many prime numbers $p$ such that $q_{e,f,k}({\mathbf z})\in z_k\mathbb{Z}_p[[{\mathbf z}]]$;
\item[$(b)$] for all ${\mathbf L}\in\mathcal{E}_{e,f}$ verifying ${\mathbf L}^{(k)}\geq 1$, there are only finitely many prime numbers $p$ such that $q_{{\mathbf L},e,f}({\mathbf z})\in\mathbb{Z}_p[[{\mathbf z}]]$.
\end{itemize}
\end{theo}

\subsection{Comparison of Theorems \ref{critère}, \ref{critère2} and \ref{cas e>f} with previous results}\label{partie comp ant}

First, we prove that Theorems \ref{critère} and \ref{critère2} generalize Corollary 1 and Theorem 2 from \cite{Tanguy 2}. We only have to prove that, if $e$ and $f$ are two disjoint sequences of nonzero vectors in $\mathbb{N}^d$, verifying $|e|=|f|$ and such that $f$ is only constituted by vectors $\mathbf{1}_k$ with $k\in\{1,\dots,d\}$, then, for all $\mathbf{x}\in\mathcal{D}_{e,f}$, we have $\Delta_{e,f}(\mathbf{x})\geq 1$. Indeed, if $\mathbf{x}\in\mathcal{D}_{e,f}$, then $\mathbf{x}\in[0,1[^d$ and, for all $k\in\{1,\dots,d\}$, we have $\mathbf{1}_k\cdot\mathbf{x}=0$. Thus, there exists an element $\mathbf{d}$ in $e$ such that $\mathbf{d}\cdot\mathbf{x}\geq 1$ and we have 
$$
\Delta_{e,f}(\mathbf{x})=\sum_{i=1}^{q_1}\lfloor\mathbf{e}_i\cdot\mathbf{x}\rfloor-\sum_{j=1}^{q_2}\lfloor\mathbf{f}_j\cdot\mathbf{x}\rfloor=\sum_{i=1}^{q_1}\lfloor\mathbf{e}_i\cdot\mathbf{x}\rfloor\geq 1.
$$

Let us now prove that Theorems \ref{critère2} and \ref{cas e>f} generalize Theorems 2 and 3 from \cite{Delaygue}. It is sufficient to note that if $d=1$, then $e$ and $f$ are two sequences of positive integers and, writing $M_{e,f}$ for the greatest element in the sequences $e$ and $f$, we obtain $\mathcal{E}_{e,f}=\{1,\dots,M_{e,f}\}$ and $\mathcal{D}_{e,f}=[1/M_{e,f},1[$.

\subsection{Structure of proofs}

First, we prove assertion $(ii)$ of Landau's criterion in Section~\ref{section demLandau}. 

Section \ref{sectioncongruform} is dedicated to the statement and the proof of Theorem \ref{theo généralisé}, which generalizes criteria of formal congruences proved by Dwork and by Krattenthaler and Rivoal. These criteria were crucial for the previous results about the integrality of the Taylor coefficients of mirror maps. Theorem \ref{theo généralisé} is central to the proofs of Theorems \ref{critère} and \ref{critère2}. 

In Section \ref{equiv crit}, we reduce the proofs of Theorems \ref{critère}, \ref{critère2} and \ref{cas e>f} to the proofs of $p$-adic relations. 

Section \ref{section technical lemma} is dedicated to the statement and the proof of a technical lemma which we will use to prove both assertions of Theorems \ref{critère} and \ref{critère2}.

We prove assertions $(i)$ of Theorems \ref{critère} and \ref{critère2} in Section \ref{rte65}, this is by far the longest and the most technical part of this article. Particularly, we have to prove certain number of delicate $p$-adic estimations in order to be able to apply Theorem \ref{theo généralisé}.

In Sections \ref{Proof(ii)768} and \ref{démo e>f}, we prove assertions $(ii)$ of Theorems \ref{critère} and \ref{critère2} and the Theorem  \ref{cas e>f}, which ensue rather fast from reformulations of these theorems established in Section \ref{equiv crit}.

Finally in Section \ref{consequenceZudi}, we prove that Theorems \ref{critère} and \ref{critère2} enable us to obtain the integrality of the Taylor coefficients of new univariate mirror maps listed in \cite{Tables} by Almkvist, van Enckevort, van Straten and Zudilin.

\section{Proof of assertion $(ii)$ of Landau's criterion}\label{section demLandau}

First, let us introduce some additional notations which we will use throughout this article. Given $d\in\mathbb{N}$, $d\geq 1$, $\lambda\in\mathbb{R}$, $k\in\{1,\dots,d\}$ and vectors ${\mathbf m}:=(m_1,\dots,m_d)$ and ${\mathbf n}:=(n_1,\dots,n_d)$ in $\mathbb{R}^d$, we write ${\mathbf m}+{\mathbf n}$ for $(m_1+n_1,\dots,m_d+n_d)$, $\lambda{\mathbf m}$ or ${\mathbf m}\lambda$ for $(\lambda m_1,\dots,\lambda m_d)$, and ${\mathbf m}/\lambda$ for $(m_1/\lambda,\dots,m_d/\lambda)$ when $\lambda$ is nonzero. 

To prove assertion $(ii)$ of Landau's criterion, we will use the fact that, for all prime $p$ and all ${\mathbf n}\in\mathbb{N}^d$, we have $v_p(\mathcal{Q}_{e,f}({\mathbf n}))=\sum_{\ell=1}^{\infty}\Delta({\mathbf n}/p^{\ell})$. Indeed, we recall that, for all $m\in\mathbb{N}$, we have the formula $v_p(m!)=\sum_{\ell=1}^{\infty}\lfloor m/p^{\ell}\rfloor$. Thereby, we get
\begin{align*}
v_p(\mathcal{Q}_{e,f}({\mathbf n}))
&=v_p\left(\frac{({\mathbf e}_1\cdot{\mathbf n})!\cdots({\mathbf e}_{q_1}\cdot{\mathbf n})!}{({\mathbf f}_1\cdot{\mathbf n})!\cdots({\mathbf f}_{q_2}\cdot{\mathbf n})!}\right)\\
&=\sum_{\ell=1}^{\infty}\left(\sum_{i=1}^{q_1}\lfloor{\mathbf e}_i\cdot{\mathbf n}/p^{\ell}\rfloor-\sum_{j=1}^{q_2}\lfloor{\mathbf f}_j\cdot{\mathbf n}/p^{\ell}\rfloor\right)=\sum_{\ell=1}^{\infty}\Delta\left(\frac{{\mathbf n}}{p^{\ell}}\right).
\end{align*}

We will need the following lemma, which we will also use for the proofs of assertions $(ii)$ of Theorems \ref{critère} and \ref{critère2}. In the rest of the article, we write $\mathbf{1}$ for the vector $(1,\dots,1)\in\mathbb{N}^d$.

\begin{lemme}\label{localconst}
Let $u:=({\mathbf u}_1,\dots,{\mathbf u}_n)$ be a sequence of vectors in $\mathbb{N}^d$ and ${\mathbf x}_0\in\mathbb{R}^d$. Then, there exists $\mu>0$ such that, for all ${\mathbf x}\in\mathbb{R}^d$ satisfying ${\mathbf 0}\leq{\mathbf x}\leq\mu{\mathbf 1}$ and all $i\in\{1,\dots,n\}$, we have $\lfloor{\mathbf u}_i\cdot({\mathbf x}_0+{\mathbf x})\rfloor=\lfloor{\mathbf u}_i\cdot{\mathbf x}_0\rfloor$.
\end{lemme}

\begin{proof}
For all $y>0$, there exists $\nu_y>0$ such that $\lfloor y+\nu_y\rfloor=\lfloor y\rfloor$. Thus, writing $\nu:=\min\{\nu_{{\mathbf u}_i\cdot{\mathbf x}_0}\,:\,1\leq i\leq n\}>0$, we obtain that, for all $i\in\{1,\dots,n\}$, we have $\lfloor{\mathbf u}_i.{\mathbf x}_0+\nu\rfloor=\lfloor{\mathbf u}_i.{\mathbf x}_0\rfloor$. Therefore, writing $\mu:=\min\{\nu/|{\mathbf u}_i|\,:\,1\leq i\leq n,\,{\mathbf u}_i\neq{\mathbf 0}\}>0$, we get that, for all ${\mathbf 0}\leq{\mathbf x}\leq\mu{\mathbf 1}$ and all $i\in\{1,\dots,n\}$, we have ${\mathbf u}_i\cdot{\mathbf x}\leq\mu |{\mathbf u}_i|\leq\nu$ so $\lfloor{\mathbf u}_i\cdot({\mathbf x}_0+{\mathbf x})\rfloor=\lfloor{\mathbf u}_i\cdot{\mathbf x}_0\rfloor$. This completes the proof of the lemma.
\end{proof}

\begin{proof}[Proof of assertion $(ii)$ of Landau's criterion]
Given ${\mathbf x}_0\in [0,1]^d$ satisfying $\Delta_{e,f}({\mathbf x}_0)\leq -1$ and applying Lemma \ref{localconst} with, instead of $u$, the sequence constituted by the elements of $e$ and $f$, we obtain that there exists $\mu>0$ such that, for all ${\mathbf x}\in\mathbb{R}^d$ verifying ${\mathbf 0}\leq{\mathbf x}\leq\mu{\mathbf 1}$, we have $\Delta_{e,f}({\mathbf x}_0+{\mathbf x})=\Delta_{e,f}({\mathbf x}_0)\leq -1$. We write $\mathcal{U}:=\{{\mathbf x}_0+{\mathbf x}\,:\,{\mathbf 0}\leq{\mathbf x}\leq\mu{\mathbf 1}\}$ during the proof.

There exists a constant $\mathcal{N}_1$ such that, for all prime $p\geq\mathcal{N}_1$, there is ${\mathbf n}_p\in\mathbb{N}^d$ such that ${\mathbf n}_p/p\in\mathcal{U}$. There exists a constant $\mathcal{N}_2$ such that, for all prime $p\geq\mathcal{N}_2$ and all ${\mathbf d}\in\{{\mathbf e}_1,\dots,{\mathbf e}_{q_1},{\mathbf f}_1,\dots,{\mathbf f}_{q_2}\}$, we have $|{\mathbf d}|(\mu+1)/p<1$.

Thus, for all prime number $p\geq\mathcal{N}:=\max(\mathcal{N}_1,\mathcal{N}_2)$ and all integer $\ell\geq 2$, we have $\Delta_{e,f}({\mathbf n}_p/p)\leq -1$ and, as ${\mathbf n}_p/p\in\mathcal{U}$, we have ${\mathbf n}_p/p\leq (1+\mu){\mathbf 1}$ and ${\mathbf n}_p/p^{\ell}\leq{\mathbf n}_p/p^2\leq(\mu+1)/p{\mathbf 1}$. As a result, for all ${\mathbf d}\in\{{\mathbf e}_1,\dots,{\mathbf e}_{q_1},{\mathbf f}_1,\dots,{\mathbf f}_{q_2}\}$, we obtain ${\mathbf d}\cdot{\mathbf n}_p/p^{\ell}\leq |{\mathbf d}|(\mu+1)/p<1$, which leads to ${\mathbf n}_p/p^{\ell}\in [0,1[^d\setminus\mathcal{D}_{e,f}$ and so $\Delta_{e,f}({\mathbf n}_p/p^{\ell})=0$.

Thus, for all prime $p\geq\mathcal{N}$, we have $v_p(\mathcal{Q}_{e,f}({\mathbf n}_p))=\sum_{\ell=1}^{\infty}\Delta_{e,f}({\mathbf n}_p/p^{\ell})\leq -1$, which finishes the proof of Landau's criterion.
\end{proof}

\section{Formal congruences}\label{sectioncongruform}

The proof of assertion $(i)$ of Theorem \ref{critère2} is essentially based on the generalization (Theorem \ref{theo généralisé} below) of a theorem of Krattenthaler and Rivoal \cite[Theorem 1, p. 3]{Tanguy 2} which is a multivariate adaptation of a Dwork's theorem \cite[Theorem 1, p. 296]{Dwork}. 

Before stating the Theorem \ref{theo généralisé}, we introduce some notations. Let $p$ be a prime number and $d\in\mathbb{N}$, $d\geq 1$. We write $\Omega$ for the completion of the algebraic closure of $\mathbb{Q}_p$ and $\mathcal{O}$ for the ring of integers of $\Omega$. 

If $\mathcal{N}$ is a subset of $\bigcup_{t\geq 1}\left(\{0,\dots,p^t-1\}^d\times\{t\}\right)$, then, for all $s\in\mathbb{N}$, we write $\Psi_s(\mathcal{N})$ for the set of all ${\mathbf u}\in\{0,\dots,p^s-1\}^d$ such that, for all $({\mathbf n},t)\in\mathcal{N}$, with $t\leq s$, and all ${\mathbf j}\in\{0,\dots,p^{s-t}-1\}^d$, we have ${\mathbf u}\neq{\mathbf j}+p^{s-t}{\mathbf n}$. 

Given $\mathbf{u}\in\{0,\dots,p^s-1\}^d$, $\mathbf{u}:=\sum_{k=0}^{s-1}\mathbf{u_k}p^k$ with $\mathbf{u}_k\in\{0,\dots,p-1\}^d$, we write $\mathcal{M}_s(\mathbf{u})$ for the word $\mathbf{u}_{0}\cdots\mathbf{u}_{s-1}$ of length $s$ on the alphabet $\{0,\dots,p-1\}^d$. According to this definition, we have $\mathbf{u}\in\Psi_s(\mathcal{N})$ if and only if none of the words $\mathcal{M}_t(\mathbf{n})$, $(\mathbf{n},t)\in\mathcal{N}$, is a suffix of $\mathcal{M}_s(\mathbf{u})$.

For example, let us take $\mathcal{N}:=\{(\mathbf{0},t)\,:\,t\geq 1\}$. In this case, $\Psi_s(\mathcal{N})$ is the set of all $\mathbf{u}=\sum_{k=0}^{s-1}\mathbf{u}_kp^k$ such that $\mathbf{u}_{s-1}\neq\mathbf{0}$. We observe that $\Psi_s(\mathcal{N})=\Psi_s(\mathcal{N}')$ with $\mathcal{N}'=\{(\mathbf{0},1)\}$.

\begin{theo}\label{theo généralisé}
Let us fix a prime number $p$. Let $({\mathbf A}_r)_{r\geq 0}$ be a sequence of maps from $\mathbb{N}^d$ to $\Omega\setminus\{0\}$ and $({\mathbf g}_r)_{r\geq 0}$ be a sequence of maps from $\mathbb{N}^d$ to $\mathcal{O}\setminus\{0\}$. We assume that there exists $\mathcal{N}\subset\bigcup_{t\geq 1}\left(\{0,\dots,p^t-1\}^d\times\{t\}\right)$ such that, for all $r\geq 0$, we have
\begin{itemize}
\item[$(i)$]{$|{\mathbf A}_r({\mathbf 0})|_p=1$;}
\item[$(ii)$]{for all ${\mathbf m}\in\mathbb{N}^d$, we have ${\mathbf A}_r({\mathbf m})\in {\mathbf g}_r({\mathbf m})\mathcal{O}$;}
\end{itemize}
\begin{itemize}
\item[$(iii)$]{for all $s\in\mathbb{N}$ and $\mathbf{m}\in\mathbb{N}^d$, we have: 
\begin{itemize}
\item[$(a)$] for all ${\mathbf u}\in\Psi_s(\mathcal{N})$ and ${\mathbf v}\in\{0,\dots,p-1\}^d$, we have  
$$
\frac{{\mathbf A}_r({\mathbf v}+{\mathbf u}p+{\mathbf m}p^{s+1})}{{\mathbf A}_r({\mathbf v}+{\mathbf u}p)}-\frac{{\mathbf A}_{r+1}({\mathbf u}+{\mathbf m}p^{s})}{{\mathbf A}_{r+1}({\mathbf u})}\in p^{s+1}\frac{{\mathbf g}_{r+s+1}({\mathbf m})}{{\mathbf A}_r({\mathbf v}+{\mathbf u}p)}\mathcal{O};
$$
\begin{itemize}
\item[$(a_1)$] furthermore, if ${\mathbf v}+p{\mathbf u}\in\Psi_{s+1}(\mathcal{N})$, then we have
$$
\frac{{\mathbf A}_r({\mathbf v}+{\mathbf u}p+{\mathbf m}p^{s+1})}{{\mathbf A}_r({\mathbf v}+{\mathbf u}p)}-\frac{{\mathbf A}_{r+1}({\mathbf u}+{\mathbf m}p^{s})}{{\mathbf A}_{r+1}({\mathbf u})}\in p^{s+1}\frac{{\mathbf g}_{r+s+1}({\mathbf m})}{{\mathbf g}_r({\mathbf v}+{\mathbf u}p)}\mathcal{O};
$$
\item[$(a_2)$] on the other hand, if ${\mathbf v}+p{\mathbf u}\notin\Psi_{s+1}(\mathcal{N})$, then we have 
$$
\frac{{\mathbf A}_{r+1}({\mathbf u}+p^{s}{\mathbf m})}{{\mathbf A}_{r+1}({\mathbf u})}\in p^{s+1}\frac{{\mathbf g}_{s+r+1}({\mathbf m})}{{\mathbf g}_{r}({\mathbf v}+p{\mathbf u})}\mathcal{O};
$$
\end{itemize}
\medskip

\item[$(b)$] For all $({\mathbf n},t)\in\mathcal{N}$, we have ${\mathbf g}_{r}\left({\mathbf n}+p^{t}{\mathbf m}\right)\in p^{t}{\mathbf g}_{r+t}({\mathbf m})\mathcal{O}$; 
\end{itemize}
}
\end{itemize}
Then, for all ${\mathbf a}\in\{0,\dots,p-1\}^d$, ${\mathbf m}\in\mathbb{N}^d$, $s,r\in\mathbb{N}$ and ${\mathbf K}\in\mathbb{Z}^d$, we have 
\begin{multline}\label{TheBut}
{\mathbf S}_r({\mathbf a},{\mathbf K},s,p,{\mathbf m}):=\\
\sum_{{\mathbf m}p^s\leq{\mathbf j}\leq({\mathbf m}+{\mathbf 1})p^s-{\mathbf 1}}\big({\mathbf A}_r({\mathbf a}+p({\mathbf K}-{\mathbf j})){\mathbf A}_{r+1}({\mathbf j})-{\mathbf A}_{r+1}({\mathbf K}-{\mathbf j}){\mathbf A}_r({\mathbf a}+{\mathbf j}p)\big)\in p^{s+1}{\mathbf g}_{s+r+1}({\mathbf m})\mathcal{O},
\end{multline}
where we extend $\mathbf{A}_r$ to $\mathbb{Z}^d$ by ${\mathbf A}_r({\mathbf n})=0$ if there is an $i\in\{1,\dots,d\}$ such that $n_i<0$.
\end{theo}

This theorem generalizes Theorem 1 from \cite{Tanguy 2}. Indeed, let $A:\mathbb{N}^d\mapsto\mathbb{Z}_p\setminus\{0\}$ and $g:\mathbb{N}^d\mapsto\mathbb{Z}_p\setminus\{0\}$ be two maps verifying conditions $(i)$, $(ii)$ and $(iii)$ of Theorem 1 from~\cite{Tanguy 2}. Let $(\mathbf{A}_r)_{r\geq 0}$ be the constant sequence of value $A$ and $(\mathbf{g}_r)_{r\geq 0}$ be the constant sequence of value $g$. These two sequences verify conditions $(i)$ and $(ii)$ of Theorem~\ref{theo généralisé}. Let us choose $\mathcal{N}:=\emptyset$ so that, for all $s\in\mathbb{N}$, we have $\Psi_s(\mathcal{N})=\{0,\dots,p^s-1\}^d$. In particular, conditions $(a_2)$ and $(b)$ of Theorem \ref{theo généralisé} are empty. Thus we only have to prove that $(\mathbf{A}_r)_{r\geq 0}$ and $(\mathbf{g}_r)_{r\geq 0}$ verify assertions $(a)$ and $(a_1)$ of Theorem \ref{theo généralisé}. The equality $\Psi_{s+1}(\mathcal{N})=\{0,\dots,p^{s+1}-1\}^d$, associated with assertion $(ii)$, proves that condition $(a_1)$ implies assertion $(a)$. But assertion $(a_1)$ corresponds to no other assertion than $(iii)$ of Theorem 1 from \cite{Tanguy 2}. Thus the conditions of Theorem \ref{theo généralisé} are valid and we have the conclusion of Theorem~1 from \cite{Tanguy 2}.

The aim of the end of this section is to prove Theorem \ref{theo généralisé}.

\subsection{Proof of Theorem \ref{theo généralisé}}

The structure of the proof is based on those of the theorems of Dwork and Krattenthaler and Rivoal, but it rather appreciably differs in details.

For all $s\in\mathbb{N}$, $s\geq 1$, we write $\alpha_s$ for the following assertion: for all ${\mathbf a}\in\{0,\dots,p-1\}^d$, $u\in\{0,\dots,s-1\}$, ${\mathbf m}\in\mathbb{N}^d$, $r\geq0$ and ${\mathbf K}\in\mathbb{Z}^d$, we have the congruence 
$$
{\mathbf S}_r({\mathbf a},{\mathbf K},u,p,{\mathbf m})\in p^{u+1}{\mathbf g}_{u+r+1}({\mathbf m})\mathcal{O}.
$$

For all $s\in\mathbb{N}$, $s\geq 1$ and $t\in\{0,\dots,s\}$, we write $\beta_{t,s}$ for the following assertion: for all ${\mathbf a}\in\{0,\dots,p-1\}^d$, ${\mathbf m}\in\mathbb{N}^d$, $r\geq0$ and ${\mathbf K}\in\mathbb{Z}^d$, we have the congruence 
\begin{multline*}
{\mathbf S}_r({\mathbf a},{\mathbf K}+{\mathbf m}p^{s},s,p,{\mathbf m})\equiv\\ \sum_{{\mathbf j}\in\Psi_{s-t}(\mathcal{N})}\frac{{\mathbf A}_{t+r+1}({\mathbf j}+{\mathbf m}p^{s-t})}{{\mathbf A}_{t+r+1}({\mathbf j})}{\mathbf S}_r({\mathbf a},{\mathbf K},t,p,{\mathbf j})\mod p^{s+1}{\mathbf g}_{s+r+1}({\mathbf m})\mathcal{O}.
\end{multline*}
For all ${\mathbf a}\in\{0,\dots,p-1\}^d$, ${\mathbf K}\in\mathbb{Z}^d$, $r\in\mathbb{N}$ and ${\mathbf j}\in\mathbb{N}^d$, we set 
$$
{\mathbf U}_r({\mathbf a},{\mathbf K},p,{\mathbf j}):={\mathbf A}_r({\mathbf a}+p({\mathbf K}-{\mathbf j})){\mathbf A}_{r+1}({\mathbf j})-{\mathbf A}_{r+1}({\mathbf K}-{\mathbf j}){\mathbf A}_r({\mathbf a}+{\mathbf j}p).
$$
Then we have 
$$
{\mathbf S}_r({\mathbf a},{\mathbf K},s,p,{\mathbf m})=\sum_{{\mathbf 0}\leq{\mathbf j}\leq(p^s-1){\mathbf 1}}{\mathbf U}_r({\mathbf a},{\mathbf K},p,{\mathbf j}+{\mathbf m}p^s).
$$ 
We state now four lemmas enabling us to prove \eqref{TheBut}.

\begin{lemme}\label{Assertion 1}
Assertion $\alpha_1$ is true.
\end{lemme}

\begin{lemme}\label{Assertion 2}
For all $s,r\in\mathbb{N}$, ${\mathbf m}\in\mathbb{N}^d$, $a\in\{0,\dots,p-1\}^d$, ${\mathbf j}\in\Psi_s(\mathcal{N})$ and ${\mathbf K}\in\mathbb{Z}^d$, we have
$$
{\mathbf U}_r({\mathbf a},{\mathbf K}+{\mathbf m}p^{s},p,{\mathbf j}+{\mathbf m}p^{s})\equiv\frac{{\mathbf A}_{r+1}({\mathbf j}+{\mathbf m}p^{s})}{{\mathbf A}_{r+1}({\mathbf j})}{\mathbf U}_r({\mathbf a},{\mathbf K},p,{\mathbf j})\mod p^{s+1}{\mathbf g}_{s+r+1}({\mathbf m})\mathcal{O}.
$$
\end{lemme}

\begin{lemme}\label{iii 0}
For all $s\in\mathbb{N}$, $s\geq 1$, if $\alpha_s$ is true, then, for all ${\mathbf a}\in\{0,\dots,p-1\}^d$, ${\mathbf K}\in\mathbb{Z}^d$, $r\geq 0$ and ${\mathbf m}\in\mathbb{N}^d$, we have
$$
{\mathbf S}_r({\mathbf a},{\mathbf K},s,p,{\mathbf m})\equiv\sum_{{\mathbf j}\in\Psi_s(\mathcal{N})}{\mathbf U}_r({\mathbf a},{\mathbf K},p,{\mathbf j}+{\mathbf m}p^s)\mod p^{s+1}{\mathbf g}_{s+r+1}({\mathbf m})\mathcal{O}.
$$
\end{lemme}

\begin{lemme}\label{Assertion 3}
For all $s\in\mathbb{N}$, $s\geq 1$, and all $t\in\{0,\dots,s-1\}$, assertions $\alpha_s$ and $\beta_{t,s}$ imply assertion $\beta_{t+1,s}$.
\end{lemme}

Before proving these lemmas, we check that their validity implies \eqref{TheBut}. We prove by induction on $s$ that $\alpha_s$ is true for all $s\geq 1$, which leads to the conclusion of Theorem \ref{theo généralisé}. According to Lemma \ref{Assertion 1}, $\alpha_1$ is true. Let us assume that $\alpha_s$ is true for a fixed $s\geq 1$. We note that $\beta_{0,s}$ is the assertion 
\begin{multline*}
\beta_{0,s}:{\mathbf S}_r({\mathbf a},{\mathbf K}+{\mathbf m}p^{s},s,p,{\mathbf m})\equiv\\
\sum_{{\mathbf j}\in\Psi_s(\mathcal{N})}\frac{{\mathbf A}_{r+1}({\mathbf j}+{\mathbf m}p^{s})}{{\mathbf A}_{r+1}({\mathbf j})}{\mathbf S}_r({\mathbf a},{\mathbf K},0,p,{\mathbf j})\mod p^{s+1}{\mathbf g}_{s+r+1}({\mathbf m})\mathcal{O}.
\end{multline*}
As ${\mathbf S}_r({\mathbf a},{\mathbf K},0,p,{\mathbf j})={\mathbf U}_r({\mathbf a},{\mathbf K},p,{\mathbf j})$, we have
$$
\sum_{{\mathbf j}\in\Psi_s(\mathcal{N})}\frac{{\mathbf A}_{r+1}({\mathbf j}+{\mathbf m}p^{s})}{{\mathbf A}_{r+1}({\mathbf j})}{\mathbf S}_r({\mathbf a},{\mathbf K},0,p,{\mathbf j})=\sum_{{\mathbf j}\in\Psi_s(\mathcal{N})}\frac{{\mathbf A}_{r+1}({\mathbf j}+{\mathbf m}p^{s})}{{\mathbf A}_{r+1}({\mathbf j})}{\mathbf U}_r({\mathbf a},{\mathbf K},p,{\mathbf j})
$$
and, according to Lemma \ref{Assertion 2}, we get
\begin{align}
\sum_{{\mathbf j}\in\Psi_s(\mathcal{N})}\frac{{\mathbf A}_{r+1}({\mathbf j}+{\mathbf m}p^{s})}{{\mathbf A}_{r+1}({\mathbf j})}{\mathbf U}_r&({\mathbf a},{\mathbf K},p,{\mathbf j})\notag\\
&\equiv\sum_{{\mathbf j}\in\Psi_s(\mathcal{N})}{\mathbf U}_r({\mathbf a},{\mathbf K}+{\mathbf m}p^{s},p,{\mathbf j}+{\mathbf m}p^{s})\mod p^{s+1}{\mathbf g}_{s+r+1}({\mathbf m})\mathcal{O}\notag\\
&\equiv {\mathbf S}_r({\mathbf a},{\mathbf K}+{\mathbf m}p^{s},s,p,{\mathbf m})\mod p^{s+1}{\mathbf g}_{s+r+1}({\mathbf m})\mathcal{O},\label{expli cond 0}
\end{align}
where \eqref{expli cond 0} is obtained \textit{via} Lemma \ref{iii 0}.

Hence, assertion $\beta_{0,s}$ is true. Then we get, according to Lemma \ref{Assertion 3}, the validity of $\beta_{1,s}$. By iteration of Lemma \ref{Assertion 3}, we finally obtain $\beta_{s,s}$ which is 
\begin{align}
{\mathbf S}_r({\mathbf a},{\mathbf K}+{\mathbf m}p^{s},s,p,{\mathbf m})
&\equiv\sum_{{\mathbf j}\in\Psi_{0}(\mathcal{N})}\frac{{\mathbf A}_{s+r+1}({\mathbf j}+{\mathbf m})}{{\mathbf A}_{s+r+1}({\mathbf j})}{\mathbf S}_r({\mathbf a},{\mathbf K},s,p,{\mathbf j})\mod p^{s+1}{\mathbf g}_{s+r+1}({\mathbf m})\mathcal{O}\notag\\
&\equiv\frac{{\mathbf A}_{s+r+1}({\mathbf m})}{{\mathbf A}_{s+r+1}({\mathbf 0})}{\mathbf S}_r({\mathbf a},{\mathbf K},s,p,{\mathbf 0})\mod p^{s+1}{\mathbf g}_{s+r+1}({\mathbf m})\mathcal{O}\label{assert beta s},
\end{align}
where we used the fact that $\Psi_{0}(\mathcal{N})=\{{\mathbf 0}\}$ for \eqref{assert beta s}. 

We will now prove that, for all ${\mathbf a}\in\{0,\dots,p-1\}^d$, $r\in\mathbb{N}$ and ${\mathbf K}\in\mathbb{Z}^d$, we have ${\mathbf S}_r({\mathbf a},{\mathbf K},s,p,{\mathbf 0})\in p^{s+1}\mathcal{O}$. For all ${\mathbf N}\in\mathbb{Z}^d$, we write $\textsl{P}_{{\mathbf N}}$ for the assertion: ``for all ${\mathbf a}\in\{0,\dots,p-1\}^d$ and $r\in\mathbb{N}$, we have $S_r({\mathbf a},{\mathbf N},s,p,{\mathbf 0})\in p^{s+1}\mathcal{O}$''. If there exists $i\in\{1,\dots,d\}$ such that $N_i<0$, then, for all ${\mathbf j}\in\{0,\dots,p^s-1\}^d$, we have ${\mathbf A}_r({\mathbf a}+p({\mathbf N}-{\mathbf j}))=0$ and ${\mathbf A}_{r+1}({\mathbf N}-{\mathbf j})=0$ so that $S_r({\mathbf a},{\mathbf N},s,p,{\mathbf 0})=0\in p^{s+1}\mathcal{O}$. First, we prove by contradiction that, for all $\mathbf{N}\in\mathbb{Z}^d$, $\textsl{P}_{\mathbf{N}}$ is true. Let us assume that there is a minimal element ${\mathbf N}\in\mathbb{N}^d$ such that $\textsl{P}_{{\mathbf N}}$ is false. Given ${\mathbf m}\in\mathbb{N}^d\setminus\{{\mathbf 0}\}$ and ${\mathbf N}':={\mathbf N}-{\mathbf m}p^s$ and applying \eqref{assert beta s} with ${\mathbf N}'$ instead of ${\mathbf K}$, we get
$$
{\mathbf S}_r({\mathbf a},{\mathbf N},s,p,{\mathbf m})\equiv\frac{{\mathbf A}_{s+r+1}({\mathbf m})}{{\mathbf A}_{s+r+1}({\mathbf 0})}{\mathbf S}_r({\mathbf a},{\mathbf N}',s,p,{\mathbf 0})\mod p^{s+1}{\mathbf g}_{s+r+1}({\mathbf m})\mathcal{O}.
$$
As ${\mathbf m}\in \mathbb{N}^d\setminus\{{\mathbf 0}\}$, we have ${\mathbf N}'<{\mathbf N}$, which, according to the definition of ${\mathbf N}$, leads to ${\mathbf S}_r({\mathbf a},{\mathbf N}',s,p,{\mathbf 0})\in p^{s+1}\mathcal{O}$. According to conditions $(i)$ and $(ii)$, we have $|{\mathbf A}_{s+r+1}({\mathbf 0})|_p=1$ and ${\mathbf A}_{s+r+1}({\mathbf m})\in{\mathbf g}_{s+r+1}({\mathbf m})\mathcal{O}$, so we get ${\mathbf S}_r({\mathbf a},{\mathbf N},s,p,{\mathbf m})\in p^{s+1}{\mathbf g}_{s+r+1}({\mathbf m})\mathcal{O}\subset p^{s+1}\mathcal{O}$. Thereby, for all ${\mathbf m}\in\mathbb{N}^d\setminus\{{\mathbf 0}\}$, we have ${\mathbf S}_r({\mathbf a},{\mathbf N},s,p,{\mathbf m})\in p^{s+1}\mathcal{O}$. Given ${\mathbf T}\in\mathbb{N}^d$ such that, for all $i\in\{1,\dots,d\}$ we have $(T_i+1)p^s>N_i$, we get
\begin{align}
\sum_{{\mathbf 0}\leq{\mathbf m}\leq{\mathbf T}} {\mathbf S}_r({\mathbf a}&,{\mathbf N},s,p,{\mathbf m})\notag\\
&=\sum_{{\mathbf 0}\leq{\mathbf m}\leq{\mathbf T}}\sum_{{\mathbf m}p^s\leq{\mathbf j}\leq({\mathbf m}+{\mathbf 1})p^s-{\mathbf 1}}\left({\mathbf A}_r({\mathbf a}+p({\mathbf N}-{\mathbf j})){\mathbf A}_{r+1}({\mathbf j})-{\mathbf A}_{r+1}({\mathbf N}-{\mathbf j}){\mathbf A}_r({\mathbf a}+{\mathbf j}p)\right)\notag\\
&=\sum_{{\mathbf 0}\leq{\mathbf j}\leq{\mathbf N}}\left({\mathbf A}_r({\mathbf a}+p({\mathbf N}-{\mathbf j})){\mathbf A}_{r+1}({\mathbf j})-{\mathbf A}_{r+1}({\mathbf N}-{\mathbf j}){\mathbf A}_r({\mathbf a}+{\mathbf j}p)\right)\label{rappel conv}\\
&=0\label{expli somme 0},
\end{align}
where we used the fact that ${\mathbf A}_r({\mathbf n})=0$ when there is an $i\in\{1,\dots,d\}$ such that $n_i<0$ for \eqref{rappel conv}, and \eqref{expli somme 0} occurs because the term of sum \eqref{rappel conv} is changed into its opposite when changing the index ${\mathbf j}$ in ${\mathbf N}-{\mathbf j}$. So we obtain ${\mathbf S}_r({\mathbf a},{\mathbf N},s,p,{\mathbf 0})=-\sum_{{\mathbf 0}<{\mathbf m}\leq{\mathbf T}}{\mathbf S}_r({\mathbf a},{\mathbf N},s,p,{\mathbf m})\in p^{s+1}\mathcal{O}$,
which is contradictory to the status of ${\mathbf N}$. Thus, for all ${\mathbf N}\in\mathbb{Z}^d$, $\textsl{P}_{{\mathbf N}}$ is true.

Furthermore, conditions $(i)$ and $(ii)$ respectively lead to 
$$
|{\mathbf A}_{s+r+1}({\mathbf 0})|_p=1\quad\textup{and}\quad{\mathbf A}_{s+r+1}({\mathbf m})\in{\mathbf g}_{s+r+1}({\mathbf m})\mathcal{O}.
$$
Then we obtain, according to \eqref{assert beta s}, that ${\mathbf S}_r({\mathbf a},{\mathbf K}+{\mathbf m}p^s,s,p,{\mathbf m})\in p^{s+1}{\mathbf g}_{s+r+1}({\mathbf m})\mathcal{O}$. This latest congruence is valid for all ${\mathbf a}\in\{0,\dots,p-1\}^d$, ${\mathbf K}\in\mathbb{Z}^d$, ${\mathbf m}\in\mathbb{N}^d$ and $r\geq 0$, which proves that the assertion $\alpha_{s+1}$ is true and completes the induction on $s$. We now have to prove Lemmas \ref{Assertion 1}, \ref{Assertion 2}, \ref{iii 0} and \ref{Assertion 3}.

\subsubsection{Proof of Lemma \ref{Assertion 1}}
Given ${\mathbf a}\in\{0,\dots,p-1\}^d$, ${\mathbf K}\in\mathbb{Z}^d$, ${\mathbf m}\in\mathbb{N}^d$ and $r\geq 0$, we have 
\begin{equation}\label{refref}
{\mathbf S}_r({\mathbf a},{\mathbf K},0,p,{\mathbf m})={\mathbf A}_r({\mathbf a}+p({\mathbf K}-{\mathbf m})){\mathbf A}_{r+1}({\mathbf m})-{\mathbf A}_{r+1}({\mathbf K}-{\mathbf m}){\mathbf A}_r({\mathbf a}+p{\mathbf m}).
\end{equation}
If ${\mathbf K}-{\mathbf m}\notin\mathbb{N}^d$, then we have ${\mathbf A}_r({\mathbf a}+p({\mathbf K}-{\mathbf m}))=0$ and ${\mathbf A}_{r+1}({\mathbf K}-{\mathbf m})=0$ so that ${\mathbf S}_r({\mathbf a},{\mathbf K},0,p,{\mathbf m})=0\in p{\mathbf g}_{r+1}({\mathbf m})\mathcal{O}$, as expected. Thus we can assume that ${\mathbf K}-{\mathbf m}\in\mathbb{N}^d$. We rewrite \eqref{refref} as follows. 
\begin{multline}
{\mathbf S}_r({\mathbf a},{\mathbf K},0,p,{\mathbf m})={\mathbf A}_{r}({\mathbf a})\Bigg({\mathbf A}_{r+1}({\mathbf m})\left(\frac{{\mathbf A}_{r}({\mathbf a}+p({\mathbf K}-{\mathbf m}))}{{\mathbf A}_{r}({\mathbf a})}-\frac{{\mathbf A}_{r+1}({\mathbf K}-{\mathbf m})}{{\mathbf A}_{r+1}({\mathbf 0})}\right)\label{der0}\\-{\mathbf A}_{r+1}({\mathbf K}-{\mathbf m})\left(\frac{{\mathbf A}_{r}({\mathbf a}+{\mathbf m}p)}{{\mathbf A}_{r}({\mathbf a})}-\frac{{\mathbf A}_{r+1}({\mathbf m})}{{\mathbf A}_{r+1}({\mathbf 0})}\right)\Bigg).
\end{multline}
As $\Psi_{0}(\mathcal{N})=\{{\mathbf 0}\}$, we can use $(a)$, with ${\mathbf 0}$ instead of ${\mathbf u}$ and ${\mathbf a}$ instead of ${\mathbf v}$, to obtain
$$
\frac{{\mathbf A}_{r}({\mathbf a}+p({\mathbf K}-{\mathbf m}))}{{\mathbf A}_{r}({\mathbf a})}-\frac{{\mathbf A}_{r+1}({\mathbf K}-{\mathbf m})}{{\mathbf A}_{r+1}({\mathbf 0})}\in p\frac{{\mathbf g}_{r+1}({\mathbf K}-{\mathbf m})}{{\mathbf A}_r({\mathbf a})}\mathcal{O}
$$
and
$$
\frac{{\mathbf A}_{r}({\mathbf a}+{\mathbf m}p)}{{\mathbf A}_{r}({\mathbf a})}-\frac{{\mathbf A}_{r+1}({\mathbf m})}{{\mathbf A}_{r+1}({\mathbf 0})}\in p\frac{{\mathbf g}_{r+1}({\mathbf m})}{{\mathbf A}_r({\mathbf a})}\mathcal{O}.
$$
This leads to
\begin{align}
{\mathbf A}_r({\mathbf a}){\mathbf A}_{r+1}({\mathbf m})\left(\frac{{\mathbf A}_{r}({\mathbf a}+p({\mathbf K}-{\mathbf m}))}{{\mathbf A}_{r}({\mathbf a})}-\frac{{\mathbf A}_{r+1}({\mathbf K}-{\mathbf m})}{{\mathbf A}_{r+1}({\mathbf 0})}\right)
&\in p{\mathbf g}_{r+1}({\mathbf K}-{\mathbf m}){\mathbf A}_{r+1}({\mathbf m})\mathcal{O}\notag\\
&\subset p{\mathbf g}_{r+1}({\mathbf m})\mathcal{O}\label{der1}
\end{align}
and
\begin{align}\label{der2}
{\mathbf A}_r({\mathbf a}){\mathbf A}_{r+1}({\mathbf K}-{\mathbf m})\left(\frac{{\mathbf A}_{r}({\mathbf a}+{\mathbf m}p)}{{\mathbf A}_{r}({\mathbf a})}-\frac{{\mathbf A}_{r+1}({\mathbf m})}{{\mathbf A}_{r+1}({\mathbf 0})}\right)
&\in p{\mathbf g}_{r+1}({\mathbf m}){\mathbf A}_{r+1}({\mathbf K}-{\mathbf m})\mathcal{O}\notag\\
&\subset p{\mathbf g}_{r+1}({\mathbf m})\mathcal{O},
\end{align}
where we used condition $(ii)$ for \eqref{der1} and \eqref{der2}, which leads to ${\mathbf A}_{r+1}({\mathbf m})\in{\mathbf g}_{r+1}({\mathbf m})\mathcal{O}$ and ${\mathbf A}_{r+1}({\mathbf K}-{\mathbf m})\in{\mathbf g}_{r+1}({\mathbf K}-{\mathbf m})\mathcal{O}\subset\mathcal{O}$. Applying \eqref{der1} and \eqref{der2} to \eqref{der0}, we obtain ${\mathbf S}_r({\mathbf a},{\mathbf K},0,p,{\mathbf m})\in p{\mathbf g}_{r+1}({\mathbf m})$, which finishes the proof of the lemma.

\subsubsection{Proof of Lemma \ref{Assertion 2}}
We have
\begin{multline}\label{wesh wesh yo}
{\mathbf U}_r({\mathbf a},{\mathbf K}+{\mathbf m}p^{s},p,{\mathbf j}+{\mathbf m}p^{s})-\frac{{\mathbf A}_{r+1}({\mathbf j}+{\mathbf m}p^{s})}{{\mathbf A}_{r+1}({\mathbf j})}{\mathbf U}_r({\mathbf a},{\mathbf K},p,{\mathbf j})\\
=-{\mathbf A}_{r+1}({\mathbf K}-{\mathbf j}){\mathbf A}_r({\mathbf a}+{\mathbf j}p)\left(\frac{{\mathbf A}_r({\mathbf a}+{\mathbf j}p+{\mathbf m}p^{s+1})}{{\mathbf A}_r({\mathbf a}+{\mathbf j}p)}-\frac{{\mathbf A}_{r+1}({\mathbf j}+{\mathbf m}p^{s})}{{\mathbf A}_{r+1}({\mathbf j})}\right).
\end{multline}
As ${\mathbf j}\in\Psi_s(\mathcal{N})$, hypothesis $(a)$ implies that the right-hand side of equality \eqref{wesh wesh yo} lies in 
$$
{\mathbf A}_{r+1}({\mathbf K}-{\mathbf j}){\mathbf A}_r({\mathbf a}+{\mathbf j}p)p^{s+1}\frac{{\mathbf g}_{s+r+1}({\mathbf m})}{{\mathbf A}_r({\mathbf a}+{\mathbf j}p)}\mathcal{O}. 
$$
Furthermore, according to condition $(ii)$, we have ${\mathbf A}_{r+1}({\mathbf K}-{\mathbf j})\in{\mathbf g}_{r+1}({\mathbf K}-{\mathbf j})\mathcal{O}\subset\mathcal{O}$. These estimates prove that the left-hand side of \eqref{wesh wesh yo} lies in $p^{s+1}{\mathbf g}_{s+r+1}({\mathbf m})\mathcal{O}$, which completes the proof of the lemma.

\subsubsection{Proof of Lemma \ref{iii 0}}
Let us fix $r,s\in\mathbb{N}$, $s\geq 1$, such that $\alpha_s$ is true. 

For all $u\in\{0,\dots,s\}$, we write $\mathcal{A}_u$ for the assertion: for all ${\mathbf n}\in\{0,\dots,p^{s-u}-1\}^d$, we have
$$
\sum_{{\mathbf 0}\leq{\mathbf j}\leq (p^{u}-1){\mathbf 1}}{\mathbf U}_r({\mathbf a},{\mathbf K},p,{\mathbf j}+{\mathbf n}p^{u}+{\mathbf m}p^s)={\mathbf S}_r({\mathbf a},{\mathbf K},u,p,{\mathbf n}+{\mathbf m}p^{s-u}).
$$
We will prove by induction on $u$ that, for all $u\in\{0,\dots,s\}$, the assertion $\mathcal{A}_u$ is true. 

If $u=0$, then there is nothing to prove so $\mathcal{A}_0$ is true. Let $u\in\{0,\dots,s-1\}$ such that $\mathcal{A}_u$ is true. Let us prove that $\mathcal{A}_{u+1}$ is true. For all ${\mathbf n}\in\{0,\dots,p^{s-u-1}-1\}^d$, we have
\begin{align}
{\mathbf S}_r({\mathbf a},{\mathbf K},u+1,&p,{\mathbf n}+{\mathbf m}p^{s-u-1})
=\sum_{{\mathbf 0}\leq{\mathbf v}\leq(p-1){\mathbf 1}}{\mathbf S}_r({\mathbf a},{\mathbf K},u,p,{\mathbf v}+{\mathbf n}p+{\mathbf m}p^{s-u})\notag\\
&=\sum_{{\mathbf 0}\leq{\mathbf v}\leq(p-1){\mathbf 1}}\sum_{{\mathbf 0}\leq{\mathbf j}\leq (p^{u}-1){\mathbf 1}}{\mathbf U}_r({\mathbf a},{\mathbf K},p,{\mathbf j}+{\mathbf v}p^{u}+{\mathbf n}p^{u+1}+{\mathbf m}p^s)\label{floux1}\\
&=\sum_{{\mathbf 0}\leq{\mathbf j}\leq (p^{u+1}-1){\mathbf 1}}{\mathbf U}_r({\mathbf a},{\mathbf K},p,{\mathbf j}+{\mathbf n}p^{u+1}+{\mathbf m}p^s)\label{floux2},
\end{align}
where we used assertion $\mathcal{A}_u$ for \eqref{floux1}. Equality \eqref{floux2} proves that $\mathcal{A}_{u+1}$ is true, which finishes the induction on $u$. 

If $\Psi_s(\mathcal{N})=\{0,\dots,p^s-1\}^d$, then Lemma \ref{iii 0} is trivial. In the sequel of this proof, we assume that $\Psi_{s}(\mathcal{N})\neq\{0,\dots,p^s-1\}^d$. We have ${\mathbf u}\in\{0,\dots,p^s-1\}^d\setminus\Psi_s(\mathcal{N})$ if and only if there exists $({\mathbf n},t)\in\mathcal{N}$, $t\leq s$, and ${\mathbf j}\in\{0,\dots,p^{s-t}-1\}^d$ such that ${\mathbf u}={\mathbf j}+p^{s-t}{\mathbf n}$. We write $\mathcal{N}_s$ the set of all $({\mathbf n},t)\in\mathcal{N}$ with $t\leq s$. So we have 
$$
\{0,\dots,p^s-1\}^d\setminus\Psi_s(\mathcal{N})=\bigcup_{({\mathbf n},t)\in\mathcal{N}_s}\{{\mathbf j}+p^{s-t}{\mathbf n}\,:\,{\mathbf j}\in\{0,\dots,p^{s-t}-1\}^d\}.
$$
In particular, the set $\mathcal{N}_s$ is nonempty.

We will prove that there exists $k\in\mathbb{N}$, $k\geq 1$, and $({\mathbf n}_1,t_1),\dots,({\mathbf n}_k,t_k)\in\mathcal{N}_s$ such that the sets $J({\mathbf n}_i,t_i):=\{{\mathbf j}+p^{s-t_i}{\mathbf n}_i\,:\,{\mathbf j}\in\{0,\dots,p^{s-t_i}-1\}^d\}$ induce a partition of $\{0,\dots,p^s-1\}^d\setminus\Psi_s(\mathcal{N})$. We observe that $\mathcal{N}_s\subset\bigcup_{t=1}^s\left(\{0,\dots,p^t-1\}\times\{t\}\right)$ and thus $\mathcal{N}_s$ is finite. Therefore, we only have to prove that if $({\mathbf n},t),({\mathbf n}',t')\in\mathcal{N}_s$, ${\mathbf j}\in\{0,\dots,p^{s-t}-1\}^d$ and ${\mathbf j}'\in\{0,\dots,p^{s-t'}-1\}^d$ verify ${\mathbf j}+p^{s-t}{\mathbf n}={\mathbf j}'+p^{s-t'}{\mathbf n}'$, then we have $J({\mathbf n},t)\subset J({\mathbf n}',t')$ or $J({\mathbf n}',t')\subset J({\mathbf n},t)$. Let us assume, for example, that $t\leq t'$. Then there exists ${\mathbf j}_0\in\{0,\dots,p^{t'-t}-1\}^d$ such that ${\mathbf j}={\mathbf j}'+p^{s-t'}{\mathbf j}_0$, so that $p^{s-t'}{\mathbf n}'=p^{s-t}{\mathbf n}+p^{s-t'}{\mathbf j}_0$ and thus $J({\mathbf n}',t')\subset J({\mathbf n},t)$. Also, if $t\geq t'$, then we have $J({\mathbf n},t)\subset J({\mathbf n}',t')$. Thus, we get
\begin{equation}\label{joint2}
{\mathbf S}_r({\mathbf a},{\mathbf K},s,p,{\mathbf m})=\sum_{{\mathbf j}\in\Psi_s(\mathcal{N})}{\mathbf U}_r({\mathbf a},{\mathbf K},p,{\mathbf j}+{\mathbf m}p^s)+\sum_{{\mathbf j}\in\{0,\dots,p^s-1\}^d\setminus\Psi_s(\mathcal{N})}{\mathbf U}_r({\mathbf a},{\mathbf K},p,{\mathbf j}+{\mathbf m}p^s),
\end{equation}
with
\begin{equation}\label{joint1}
\sum_{{\mathbf j}\in\{0,\dots,p^s-1\}^d\setminus\Psi_s(\mathcal{N})}{\mathbf U}_r({\mathbf a},{\mathbf K},p,{\mathbf j}+{\mathbf m}p^s)=\sum_{i=1}^k\sum_{{\mathbf j}\in\{0,\dots,p^{s-t_i}-1\}^d}{\mathbf U}_r({\mathbf a},{\mathbf K},p,{\mathbf j}+p^{s-t_i}{\mathbf n}_i+{\mathbf m}p^s).
\end{equation}
We now prove that, for all $i\in\{1,\dots,k\}$, we have
\begin{equation}\label{étoile 287}
\sum_{{\mathbf j}\in\{0,\dots,p^{s-t_i}-1\}^d}{\mathbf U}_r({\mathbf a},{\mathbf K},p,{\mathbf j}+p^{s-t_i}{\mathbf n}_i+{\mathbf m}p^s)\in p^{s+1}{\mathbf g}_{s+r+1}({\mathbf m})\mathcal{O}.
\end{equation}

Given $i\in\{1,\dots,k\}$, assertion $\mathcal{A}_{s-t_i}$ leads to
$$
\sum_{{\mathbf 0}\leq{\mathbf j}\leq(p^{s-t_i}-1){\mathbf 1}}{\mathbf U}_r({\mathbf a},{\mathbf K},p,{\mathbf j}+p^{s-t_i}{\mathbf n}_i+{\mathbf m}p^s)={\mathbf S}_r({\mathbf a},{\mathbf K},s-t_i,p,{\mathbf n}_i+{\mathbf m}p^{t_i}).
$$
As $t_i\geq 1$, we get, \textit{via} $\alpha_s$, that 
$$
{\mathbf S}_r({\mathbf a},{\mathbf K},s-t_i,p,{\mathbf n}_i+{\mathbf m}p^{t_i})\in p^{s-t_i+1}{\mathbf g}_{s-t_i+r+1}({\mathbf n}_i+{\mathbf m}p^{t_i})\mathcal{O}.
$$
Applying assertion $(b)$ with $t_i$ instead of $t$ and $r+s-t_i+1$ instead of $r$, we obtain 
$$
p^{s-t_i+1}{\mathbf g}_{s-t_i+r+1}({\mathbf n}_i+{\mathbf m}p^{t_i})\in p^{s-t_i+1}p^{t_i}{\mathbf g}_{s+r+1}({\mathbf m})\mathcal{O}=p^{s+1}{\mathbf g}_{s+r+1}({\mathbf m})\mathcal{O}.
$$
Thus, for all $i\in\{1,\dots,k\}$, we have \eqref{étoile 287}.

Congruence \eqref{étoile 287}, associated with \eqref{joint1} and \eqref{joint2}, proves that
$$
{\mathbf S}_r({\mathbf a},{\mathbf K},s,p,{\mathbf m})\equiv\sum_{{\mathbf j}\in\Psi_s(\mathcal{N})}{\mathbf U}_r({\mathbf a},{\mathbf K},p,{\mathbf j}+{\mathbf m}p^s)\mod p^{s+1}{\mathbf g}_{s+r+1}({\mathbf m})\mathcal{O},
$$
which completes the proof of Lemma \ref{iii 0}.
\medskip

\subsubsection{Proof of Lemma \ref{Assertion 3}}
During this proof, ${\mathbf i}$ indicates an element of $\{0,\dots,p-1\}^d$ and ${\mathbf u}$ indicates an element of $\{0,\dots,p^{s-t-1}-1\}^d$. For $t<s$, we write $\beta_{t,s}$ as follows
\begin{multline}
{\mathbf S}_r({\mathbf a},{\mathbf K}+{\mathbf m}p^{s},s,p,{\mathbf m})\equiv\\
\sum_{{\mathbf i}+{\mathbf u}p\in\Psi_{s-t}(\mathcal{N})}\frac{{\mathbf A}_{t+r+1}({\mathbf i}+{\mathbf u} p+{\mathbf m}p^{s-t})}{{\mathbf A}_{t+r+1}({\mathbf i}+{\mathbf u} p)}{\mathbf S}_r({\mathbf a},{\mathbf K},t,p,{\mathbf i}+{\mathbf u} p)\mod p^{s+1}{\mathbf g}_{s+r+1}({\mathbf m})\mathcal{O}.\label{beta t s}
\end{multline}
We want to prove $\beta_{t+1,s}$, which is
\begin{multline*}
{\mathbf S}_r({\mathbf a},{\mathbf K}+{\mathbf m}p^{s},s,p,{\mathbf m})\equiv\\
\sum_{{\mathbf u}\in\Psi_{s-t-1}(\mathcal{N})}\frac{{\mathbf A}_{t+r+2}({\mathbf u}+{\mathbf m}p^{s-t-1})}{{\mathbf A}_{t+r+2}({\mathbf u})}{\mathbf S}_r({\mathbf a},{\mathbf K},t+1,p,{\mathbf u})\mod p^{s+1}{\mathbf g}_{s+r+1}({\mathbf m})\mathcal{O}.
\end{multline*}
We note that ${\mathbf S}_r({\mathbf a},{\mathbf K},t+1,p,{\mathbf u})=\sum_{{\mathbf 0}\leq{\mathbf i}\leq(p-1){\mathbf 1}}{\mathbf S}_r({\mathbf a},{\mathbf K},t,p,{\mathbf i}+{\mathbf u} p)$. Thus, writing
\begin{multline*}
X:={\mathbf S}_r({\mathbf a},{\mathbf K}+{\mathbf m}p^{s},s,p,{\mathbf m})\\
-\sum_{{\mathbf 0}\leq{\mathbf i}\leq(p-1){\mathbf 1}}\sum_{{\mathbf u}\in\Psi_{s-t-1}(\mathcal{N})}\frac{{\mathbf A}_{t+r+2}({\mathbf u}+{\mathbf m}p^{s-t-1})}{{\mathbf A}_{t+r+2}({\mathbf u})}{\mathbf S}_r({\mathbf a},{\mathbf K},t,p,{\mathbf i}+{\mathbf u} p),
\end{multline*}
we only have to prove that $X\in p^{s+1}{\mathbf g}_{s+r+1}({\mathbf m})\mathcal{O}$. We have 
\begin{equation}\label{Psi+1}
{\mathbf i}+{\mathbf u}p\in\Psi_{s-t}(\mathcal{N})\Rightarrow{\mathbf u}\in\Psi_{s-t-1}(\mathcal{N}).
\end{equation}
Indeed, if ${\mathbf u}\notin\Psi_{s-t-1}(\mathcal{N})$, then there exists $({\mathbf n},k)\in\mathcal{N}$, $k\leq s-t-1$, and ${\mathbf j}\in\{0,\dots,p^{s-t-1-k}-1\}^d$ such that ${\mathbf u}={\mathbf j}+p^{s-t-1-k}{\mathbf n}$. Thus we have ${\mathbf i}+{\mathbf u}p={\mathbf i}+{\mathbf j}p+p^{s-t-k}{\mathbf n}$, which leads to ${\mathbf i}+{\mathbf u}p\notin\Psi_{s-t}(\mathcal{N})$. Hence, according to $\beta_{t,s}$ written as \eqref{beta t s}, we obtain 
\begin{multline*}
X\equiv \sum_{{\mathbf i}+{\mathbf u}p\in\Psi_{s-t}(\mathcal{N})}{\mathbf S}_r({\mathbf a},{\mathbf K},t,p,{\mathbf i}+{\mathbf u}p)\left(\frac{{\mathbf A}_{t+r+1}({\mathbf i}+{\mathbf u}p+{\mathbf m}p^{s-t})}{{\mathbf A}_{t+r+1}({\mathbf i}+{\mathbf u}p)}-\frac{{\mathbf A}_{t+r+2}({\mathbf u}+{\mathbf m}p^{s-t-1})}{{\mathbf A}_{t+r+2}({\mathbf u})}\right)\\
+\underset{{\mathbf i}+{\mathbf u}p\notin\Psi_{s-t}(\mathcal{N})}{\sum_{{\mathbf u}\in\Psi_{s-t-1}(\mathcal{N})}}\frac{{\mathbf A}_{t+r+2}({\mathbf u}+{\mathbf m}p^{s-t-1})}{{\mathbf A}_{t+r+2}({\mathbf u})}{\mathbf S}_r({\mathbf a},{\mathbf K},t,p,{\mathbf i}+{\mathbf u}p)\mod p^{s+1}{\mathbf g}_{s+r+1}({\mathbf m})\mathcal{O}.
\end{multline*}
Furthermore, applying $(a_1)$ with $s-t-1$ instead of $s$ and $t+r+1$ instead of $r$, we get
$$
\frac{{\mathbf A}_{t+r+1}({\mathbf i}+{\mathbf u}p+{\mathbf m}p^{s-t})}{{\mathbf A}_{t+r+1}({\mathbf i}+{\mathbf u} p)}-\frac{{\mathbf A}_{t+r+2}({\mathbf u}+{\mathbf m}p^{s-t-1})}{{\mathbf A}_{t+r+2}({\mathbf u})}\in p^{s-t}\frac{{\mathbf g}_{s+r+1}({\mathbf m})}{{\mathbf g}_{t+r+1}({\mathbf i}+{\mathbf u}p)}\mathcal{O}.
$$
In addition, as $t<s$ and since $\alpha_s$ is true, we have 
\begin{equation}\label{varii1}
{\mathbf S}_r({\mathbf a},{\mathbf K},t,p,{\mathbf i}+{\mathbf u}p)\in p^{t+1}{\mathbf g}_{t+r+1}({\mathbf i}+{\mathbf u} p)\mathcal{O}
\end{equation}
and we obtain 
\begin{equation}\label{varii2}
X\equiv \underset{{\mathbf i}+{\mathbf u}p\notin\Psi_{s-t}(\mathcal{N})}{\sum_{{\mathbf u}\in\Psi_{s-t-1}(\mathcal{N})}}\frac{{\mathbf A}_{t+r+2}({\mathbf u}+{\mathbf m}p^{s-t-1})}{{\mathbf A}_{t+r+2}({\mathbf u})}{\mathbf S}_r({\mathbf a},{\mathbf K},t,p,{\mathbf i}+{\mathbf u}p)\mod p^{s+1}{\mathbf g}_{s+r+1}({\mathbf m})\mathcal{O}.
\end{equation}

Finally, when ${\mathbf i}+{\mathbf u}p\notin\Psi_{s-t}(\mathcal{N})$, we can apply condition $(a_2)$ with $s-t-1$ instead of $s$, ${\mathbf i}$ instead of ${\mathbf v}$ and $r+t+1$ instead of $r$, which leads to
\begin{equation}\label{varii3}
\frac{{\mathbf A}_{t+r+2}({\mathbf u}+{\mathbf m}p^{s-t-1})}{{\mathbf A}_{t+r+2}({\mathbf u})}\in p^{s-t}\frac{{\mathbf g}_{s+r+1}({\mathbf m})}{{\mathbf g}_{t+r+1}({\mathbf i}+{\mathbf u}p)}\mathcal{O}.
\end{equation}
Applying \eqref{varii1} and \eqref{varii3} to \eqref{varii2}, we obtain $X\in p^{s+1}{\mathbf g}_{s+r+1}({\mathbf m})\mathcal{O}$. This finishes the proof of Lemma \ref{Assertion 3} and thus the one of Theorem \ref{theo généralisé}.

\section{A $p$-adic reformulation of Theorems \ref{critère}, \ref{critère2} and \ref{cas e>f}}\label{equiv crit}

Let $e$ and $f$ be two disjoint sequences of nonzero vectors in $\mathbb{N}^d$ such that $\mathcal{Q}_{e,f}$ is a family of integers. We fix ${\mathbf L}\in\mathcal{E}_{e,f}$ throughout this section. We recall that, for all $k\in\{1,\dots,d\}$, we have $q_{e,f,k}(\mathbf{z})\in z_k\mathbb{Z}[[{\mathbf z}]]$, respectively $q_{{\mathbf L},e,f}({\mathbf z})\in\mathbb{Z}[[{\mathbf z}]]$, if and only if, for all prime number $p$, we have $q_{e,f,k}(\mathbf{z})\in z_k\mathbb{Z}_p[[{\mathbf z}]]$, respectively $q_{{\mathbf L},e,f}({\mathbf z})\in\mathbb{Z}_p[[{\mathbf z}]]$. 

We will define, for all prime number $p$, two elements $\Phi_{p,k}({\mathbf a}+p{\mathbf K})$ and $\Phi_{{\mathbf L},p}({\mathbf a}+p{\mathbf K})$ of $\mathbb{Q}_p$, where ${\mathbf a}\in\{0,\dots,p-1\}^d$ and ${\mathbf K}\in\mathbb{N}^d$, and we will prove that $q_{e,f,k}(\mathbf{z})\in z_k\mathbb{Z}_p[[{\mathbf z}]]$, respectively $q_{{\mathbf L},e,f}({\mathbf z})\in\mathbb{Z}_p[[{\mathbf z}]]$, if and only if, for all ${\mathbf a}\in\{0,\dots,p-1\}^d$ and all ${\mathbf K}\in\mathbb{N}^d$, we have $\Phi_{p,k}({\mathbf a}+p{\mathbf K})\in p\mathbb{Z}_p$, respectively $\Phi_{{\mathbf L},p}({\mathbf a}+p{\mathbf K})\in p\mathbb{Z}_p$. 

To simplify notations, we will write $\mathcal{E}:=\mathcal{E}_{e,f}$, $\mathcal{D}:=\mathcal{D}_{e,f}$, $\Delta:=\Delta_{e,f}$, $\mathcal{Q}:=\mathcal{Q}_{e,f},F:=F_{e,f},G_k:=G_{e,f,k}$, $G_{{\mathbf L}}:=G_{{\mathbf L},e,f}$, $q_k:=q_{e,f,k}$ and $q_{{\mathbf L}}:=q_{{\mathbf L},e,f}$, as throughout the rest of the article. We fix a prime number $p$ in this section.

Before proving Theorems \ref{critère}, \ref{critère2} and \ref{cas e>f}, we will reformulate them. The following result is due to Krattenthaler and Rivoal's Lemma \cite[Lemma 2, p. 7]{Tanguy 2}; it is the analogue in several variables of a lemma of Dieudonné and Dwork \cite[Chap. IV, Sec. 2, Lemma~3]{Kobliz}; \cite[Chap. 14, Sec. 2]{Lang}.

\begin{lemme}\label{lemma 2 Tanguy}
Given two formal power series $F({\mathbf z})\in 1+\sum_{i=1}^dz_i\mathbb{Z}[[{\mathbf z}]]$ and $G({\mathbf z})\in\sum_{i=1}^dz_i\mathbb{Q}[[{\mathbf z}]]$, we define $q({\mathbf z}):=\exp(G({\mathbf z})/F({\mathbf z}))$. Then we have $q({\mathbf z})\in 1+\sum_{i=1}^dz_i\mathbb{Z}_p[[{\mathbf z}]]$ if and only if $F({\mathbf z})G({\mathbf z}^p)-pF({\mathbf z}^p)G({\mathbf z})\in p\sum_{i=1}^dz_i\mathbb{Z}_p[[{\mathbf z}]]$.
\end{lemme}

Lemma \ref{lemma 2 Tanguy} will enable us to ``eliminate''\, the exponential in the formulas 
$$
q_k({\mathbf z})=z_k\exp(G_k({\mathbf z})/F({\mathbf z}))\quad\textup{and}\quad q_{{\mathbf L}}({\mathbf z})=\exp(G_{{\mathbf L}}({\mathbf z})/F({\mathbf z})). 
$$
Since $\Delta\geq 0$ on $[0,1]^d$, we obtain, according to Landau's criterion, $\mathcal{Q}$ as a family of integers and thus $F({\mathbf z})\in 1+\sum_{i=1}^dz_i\mathbb{Z}[[{\mathbf z}]]$. Furthermore, according to identities \eqref{definitionG} and \eqref{definitionG L} defining the power series $G_k$ and $G_{{\mathbf L}}$, we have $G_k({\mathbf 0})=G_{{\mathbf L}}({\mathbf 0})=0$ and so $G_k({\mathbf z})$ and $G_{{\mathbf L}}({\mathbf z})$ lie in $\sum_{i=1}^dz_i\mathbb{Q}[[{\mathbf z}]]$. Thereby, following Lemma \ref{lemma 2 Tanguy}, we have $q_k({\mathbf z})\in z_k\mathbb{Z}_p[[{\mathbf z}]]$, respectively $q_{{\mathbf L}}({\mathbf z})\in\mathbb{Z}_p[[{\mathbf z}]]$, if and only if we have $F({\mathbf z})G_k({\mathbf z}^p)-pF({\mathbf z}^p)G_k({\mathbf z})\in p\sum_{i=1}^dz_i\mathbb{Z}_p[[{\mathbf z}]]$, respectively $F({\mathbf z})G_{{\mathbf L}}({\mathbf z}^p)-pF({\mathbf z}^p)G_{{\mathbf L}}({\mathbf z})\in p\sum_{i=1}^dz_i\mathbb{Z}_p[[{\mathbf z}]]$.

According to identity \eqref{definitionG} which defines $G_k$, the coefficient of ${\mathbf z}^{{\mathbf a}+p{\mathbf K}}$ in $F({\mathbf z})G_k({\mathbf z}^p)-pF({\mathbf z}^p)G_k({\mathbf z})$ is
\begin{multline*}
\Phi_{p,k}({\mathbf a}+p{\mathbf K}):=\\
\sum_{{\mathbf 0}\leq {\mathbf j}\leq{\mathbf K}}\mathcal{Q}({\mathbf K}-{\mathbf j})\mathcal{Q}({\mathbf a}+p{\mathbf j})\left(\sum_{i=1}^{q_1}{\mathbf e}_i^{(k)}(H_{({\mathbf K}-{\mathbf j})\cdot{\mathbf e}_i}-pH_{({\mathbf a}+p{\mathbf j})\cdot{\mathbf e}_i})-\sum_{i=1}^{q_2}{\mathbf f}_i^{(k)}(H_{({\mathbf K}-{\mathbf j})\cdot{\mathbf f}_i}-pH_{({\mathbf a}+p{\mathbf j})\cdot{\mathbf f}_i})\right)
\end{multline*}
and, according to identity \eqref{definitionG L} defining $G_{{\mathbf L}}$, the coefficient of ${\mathbf z}^{{\mathbf a}+p{\mathbf K}}$ in $F({\mathbf z})G_{{\mathbf L}}({\mathbf z}^p)-pF({\mathbf z}^p)G_{{\mathbf L}}({\mathbf z})$ is
$$
\Phi_{{\mathbf L},p}({\mathbf a}+{\mathbf K}p):=\sum_{{\mathbf 0}\leq{\mathbf j}\leq{\mathbf K}}\mathcal{Q}({\mathbf K}-{\mathbf j})\mathcal{Q}({\mathbf a}+{\mathbf j}p)(H_{{\mathbf L}\cdot({\mathbf K}-{\mathbf j})}-pH_{{\mathbf L}\cdot({\mathbf a}+{\mathbf j}p)}).
$$

Thus we have $q_k({\mathbf z})\in z_k\mathbb{Z}_p[[{\mathbf z}]]$, respectively $q_{{\mathbf L}}({\mathbf z})\in \mathbb{Z}_p[[{\mathbf z}]]$, if and only if, for all ${\mathbf a}\in\{0,\dots,p-1\}^d$ and ${\mathbf K}\in\mathbb{N}^d$, we have $\Phi_{p,k}({\mathbf a}+p{\mathbf K})\in p\mathbb{Z}_p$, respectively $\Phi_{{\mathbf L},p}({\mathbf a}+p{\mathbf K})\in p\mathbb{Z}_p$.

\section{A technical lemma}\label{section technical lemma}

The aim of this section is to prove the following lemma which we will use for the proofs of assertions $(i)$ and $(ii)$ of Theorem~\ref{critère2}.

\begin{lemme}\label{lemme 16}
Let $e$ and $f$ be two sequences of vectors in $\mathbb{N}^d$ such that $|e|=|f|$. Then, for all $s\in\mathbb{N}$, ${\mathbf c}\in\{0,\dots,p^s-1\}^d$ and ${\mathbf m}\in\mathbb{N}^d$, we have
$$
\frac{\mathcal{Q}_{e,f}({\mathbf c})}{\mathcal{Q}_{e,f}({\mathbf c}p)}\frac{\mathcal{Q}_{e,f}({\mathbf c}p+{\mathbf m}p^{s+1})}{\mathcal{Q}_{e,f}({\mathbf c}+{\mathbf m}p^s)}\in 1+p^{s+1}\mathbb{Z}_p.
$$
\end{lemme}

To prove Lemma \ref{lemme 16}, we will use certain properties of the $p$-adic gamma function defined as follows, $\Gamma_p(n):=(-1)^n\gamma_p(n)$, where $\gamma_p(n):=\underset{(k,p)=1}{\prod_{k=1}^{n-1}}k$. The function $\Gamma_p$ can be extend to the whole set $\mathbb{Z}_p$ but we shall not need it here.

\begin{lemme}\label{lemme7}
\begin{itemize}
\item[$(i)$] For all $n\in\mathbb{N}$, we have the formula $\frac{(np)!}{n!}=p^n\gamma_p(1+np)$.
\item[$(ii)$] For all $k,n,s\in\mathbb{N}$, we have $\Gamma_p(k+np^s)\equiv\Gamma_p(k)\;\mod\;p^s$.
\end{itemize}
\end{lemme}

Assertion $(i)$ of Lemma \ref{lemme7} is obtained by observing that $\gamma_p(1+np)=\frac{(np)!}{n!p^n}$. Assertion $(ii)$ of Lemma \ref{lemme7} is Lemma 1.1 from \cite{Lang}. We are now able to prove Lemma \ref{lemme 16}.

\begin{proof}[Proof of Lemma \ref{lemme 16}]
We have
\begin{align}
\frac{\mathcal{Q}_{e,f}({\mathbf c}p+{\mathbf m}p^{s+1})}{\mathcal{Q}_{e,f}({\mathbf c}+{\mathbf m}p^{s})}
&=\prod_{i=1}^{q_1}\frac{({\mathbf e}_i\cdot({\mathbf c}p+{\mathbf m}p^{s+1}))!}{({\mathbf e}_i\cdot({\mathbf c}+{\mathbf m}p^{s}))!}\prod_{i=1}^{q_2}\frac{({\mathbf f}_i\cdot({\mathbf c}+{\mathbf m}p^{s}))!}{({\mathbf f}_i\cdot({\mathbf c}p+{\mathbf m}p^{s+1}))!}\notag\\
&=\frac{\prod_{i=1}^{q_1}p^{{\mathbf e}_i\cdot({\mathbf c}+{\mathbf m}p^{s})}\gamma_p(1+p{\mathbf e}_i\cdot({\mathbf c}+{\mathbf m}p^{s}))}{\prod_{i=1}^{q_2}p^{{\mathbf f}_i\cdot({\mathbf c}+{\mathbf m}p^{s})}\gamma_p(1+p{\mathbf f}_i\cdot({\mathbf c}+{\mathbf m}p^{s}))}\notag\\
&=p^{(|e|-|f|)\cdot{\mathbf m}p^{s}}\frac{\prod_{i=1}^{q_1}(p^{{\mathbf e}_i\cdot{\mathbf c}}(-1)^{1+p{\mathbf e}_i\cdot({\mathbf c}+{\mathbf m}p^{s})}\Gamma_p(1+p{\mathbf e}_i\cdot({\mathbf c}+{\mathbf m}p^{s})))}{\prod_{i=1}^{q_2}\left(p^{{\mathbf f}_i\cdot{\mathbf c}}(-1)^{1+p{\mathbf f}_i\cdot({\mathbf c}+{\mathbf m}p^{s})}\Gamma_p(1+p{\mathbf f}_i\cdot({\mathbf c}+{\mathbf m}p^{s}))\right)}\notag\\
&=(-1)^{(|e|-|f|)\cdot{\mathbf m}p^{s+1}}\frac{\prod_{i=1}^{q_1}\left(p^{{\mathbf e}_i\cdot{\mathbf c}}(-1)^{1+{\mathbf e}_i\cdot{\mathbf c}p}\;\Gamma_p(1+p{\mathbf e}_i\cdot({\mathbf c}+{\mathbf m}p^{s}))\right)}{\prod_{i=1}^{q_2}\left(p^{{\mathbf f}_i\cdot{\mathbf c}}(-1)^{1+{\mathbf f}_i\cdot{\mathbf c}p}\;\Gamma_p(1+p{\mathbf f}_i\cdot({\mathbf c}+{\mathbf m}p^{s}))\right)}\label{expli le pk1}\\
&=\frac{\prod_{i=1}^{q_1}p^{{\mathbf e}_i\cdot{\mathbf c}}(-1)^{1+{\mathbf e}_i\cdot{\mathbf c}p}}{\prod_{i=1}^{q_2}p^{{\mathbf f}_i\cdot{\mathbf c}}(-1)^{1+{\mathbf f}_i\cdot{\mathbf c}p}}\cdot\frac{\prod_{i=1}^{q_1}\Gamma_p(1+p{\mathbf e}_i\cdot({\mathbf c}+{\mathbf m}p^{s}))}{\prod_{i=1}^{q_2}\Gamma_p(1+p{\mathbf f}_i\cdot({\mathbf c}+{\mathbf m}p^{s}))}\label{expli le pk2},
\end{align}
where we used the identity $|e|-|f|={\mathbf 0}$ for \eqref{expli le pk1} and \eqref{expli le pk2}. According to assertion $(ii)$ of Lemma~\ref{lemme7}, for all ${\mathbf n}\in\mathbb{N}^d$, we have $\Gamma_p(1+{\mathbf n}\cdot{\mathbf c}p+{\mathbf n}\cdot{\mathbf m}p^{s+1})\equiv\Gamma_p(1+{\mathbf n}\cdot{\mathbf c}p)\mod\;p^{s+1}$. So we get
$$
\frac{\prod_{i=1}^{q_1}\Gamma_p(1+{\mathbf e}_i\cdot{\mathbf c}p+{\mathbf e}_i\cdot{\mathbf m}p^{s+1})}{\prod_{i=1}^{q_2}\Gamma_p(1+{\mathbf f}_i\cdot{\mathbf c}p+{\mathbf f}_i\cdot{\mathbf m}p^{s+1})}=\frac{\prod_{i=1}^{q_1}\left(\Gamma_p(1+{\mathbf e}_i\cdot{\mathbf c}p)+O(p^{s+1})\right)}{\prod_{i=1}^{q_2}\left(\Gamma_p(1+{\mathbf f}_i\cdot{\mathbf c}p)+O(p^{s+1})\right)},
$$ 
where we write $x=O(p^{k})$ when $x\in p^k\mathbb{Z}_p$. Furthermore, according to the definition of $\Gamma_p$, for all ${\mathbf n}\in\mathbb{N}^d$, we have $\Gamma_p(1+{\mathbf n}\cdot{\mathbf c}p)\in\mathbb{Z}_p^{\times}$. Then we obtain
$$
\frac{\prod_{i=1}^{q_1}\left(\Gamma_p(1+{\mathbf e}_i\cdot{\mathbf c}p)+O(p^{s+1})\right)}{\prod_{i=1}^{q_2}\left(\Gamma_p(1+{\mathbf f}_i\cdot{\mathbf c}p)+O(p^{s+1})\right)}=\frac{\prod_{i=1}^{q_1}\Gamma_p(1+{\mathbf e}_i\cdot{\mathbf c}p)}{\prod_{i=1}^{q_2}\Gamma_p(1+{\mathbf f}_i\cdot{\mathbf c}p)}(1+O(p^{s+1}))
$$
and thus,
\begin{align*}
\frac{\mathcal{Q}_{e,f}({\mathbf c}p+{\mathbf m}p^{s+1})}{\mathcal{Q}_{e,f}({\mathbf c}+{\mathbf m}p^{s})}
&=\frac{\prod_{i=1}^{q_1}p^{{\mathbf e}_i\cdot{\mathbf c}}(-1)^{1+{\mathbf e}_i\cdot{\mathbf c}p}}{\prod_{i=1}^{q_2}p^{{\mathbf f}_i\cdot{\mathbf c}}(-1)^{1+{\mathbf f}_i\cdot{\mathbf c}p}}\cdot\frac{\prod_{i=1}^{q_1}\Gamma_p(1+{\mathbf e}_i\cdot{\mathbf c}p)}{\prod_{i=1}^{q_2}\Gamma_p(1+{\mathbf f}_i\cdot{\mathbf c}p)}(1+O(p^{s+1}))\\
&=\frac{\prod_{i=1}^{q_1}p^{{\mathbf e}_i\cdot{\mathbf c}}}{\prod_{i=1}^{q_2}p^{{\mathbf f}_i\cdot{\mathbf c}}}\cdot\frac{\prod_{i=1}^{q_1}\gamma_p(1+{\mathbf e}_i\cdot{\mathbf c}p)}{\prod_{i=1}^{q_2}\gamma_p(1+{\mathbf f}_i\cdot{\mathbf c}p)}(1+O(p^{s+1}))\\
&=\frac{\mathcal{Q}_{e,f}({\mathbf c}p)}{\mathcal{Q}_{e,f}({\mathbf c})}(1+O(p^{s+1})).
\end{align*}
This completes the proof of the lemma.
\end{proof}

\section{Proof of assertions $(i)$ of Theorems \ref{critère} and \ref{critère2}}\label{rte65}

We assume the hypothesis of Theorems \ref{critère} and \ref{critère2}. Furthermore, we assume that, for all ${\mathbf x}\in\mathcal{D}$, we have $\Delta({\mathbf x})\geq 1$. As we said in Section \ref{\'Enoncé du critère}, assertion $(i)$ of Theorem \ref{critère2} implies assertion $(i)$ of Theorem \ref{critère}. So the aim of this section is to prove that, for all ${\mathbf L}\in\mathcal{E}$, we have $q_{{\mathbf L}}({\mathbf z})\in\mathbb{Z}[[{\mathbf z}]]$. Following Section \ref{equiv crit}, we only have to prove that, for all ${\mathbf L}\in\mathcal{E}$, all prime number $p$, all ${\mathbf a}\in\{0,\dots,p-1\}^d$ and ${\mathbf K}\in\mathbb{N}^d$, we have $\Phi_{{\mathbf L},p}({\mathbf a}+p{\mathbf K})\in p\mathbb{Z}_p$. We fix a ${\mathbf L}\in\mathcal{E}$ in this section.

\subsection{New reformulation of the problem}

For all prime number $p$, all $s\in\mathbb{N}$, ${\mathbf a}\in\{0,\dots,p-1\}^d$ and ${\mathbf K},{\mathbf m}\in\mathbb{N}^d$, we define
$$
S({\mathbf a},{\mathbf K},s,p,{\mathbf m}):=\sum_{{\mathbf m}p^s\leq{\mathbf j}\leq({\mathbf m}+{\mathbf 1})p^s-{\mathbf 1}}(\mathcal{Q}({\mathbf a}+{\mathbf j}p)\mathcal{Q}({\mathbf K}-{\mathbf j})-\mathcal{Q}({\mathbf j})\mathcal{Q}({\mathbf a}+({\mathbf K}-{\mathbf j})p)),
$$
where we extend $\mathcal{Q}$ to $\mathbb{Z}^d$ by $\mathcal{Q}({\mathbf n})=0$ if there is an $i\in\{1,\dots,d\}$ such that $n_i<0$.

The aim of this section is to produce, for all prime number $p$, a function $g_p$ from $\mathbb{N}^d$ to $\mathbb{Z}_p$ such that: if, for all prime $p$, all $s\in\mathbb{N}$, ${\mathbf a}\in\{0,\dots,p-1\}^d$ and ${\mathbf K},{\mathbf m}\in\mathbb{N}^d$, we have $S({\mathbf a},{\mathbf K},s,p,{\mathbf m})\in p^{s+1}g_p({\mathbf m})\mathbb{Z}_p$, then we have $\Phi_{{\mathbf L},p}({\mathbf a}+{\mathbf K}p)\in p\mathbb{Z}_p$. Thus the proof of assertion $(i)$ of Theorem \ref{critère2} will amount to finding a suitable lower bound of the $p$-adic valuation of $S({\mathbf a},{\mathbf K},s,p,{\mathbf m})$ for all prime $p$. This reduction method is an adaptation of the approach to the problem made by Dwork in \cite{Dwork}.

\subsubsection{A reformulation of $\Phi_{{\mathbf L},p}({\mathbf a}+{\mathbf K}p)$ modulo $p\mathbb{Z}_p$}

This step is the analogue of a reformulation made by Krattenthaler and Rivoal in Section $2$ from \cite{Tanguy}. We fix a prime number $p$. We will prove that
\begin{equation}\label{but réécriture}
\Phi_{{\mathbf L},p}({\mathbf a}+{\mathbf K}p)\equiv-\sum_{{\mathbf 0}\leq{\mathbf j}\leq{\mathbf K}}H_{{\mathbf L}\cdot{\mathbf j}}\left(\mathcal{Q}({\mathbf a}+{\mathbf j}p)\mathcal{Q}({\mathbf K}-{\mathbf j})-\mathcal{Q}({\mathbf j})\mathcal{Q}({\mathbf a}+({\mathbf K}-{\mathbf j})p)\right)\mod p\mathbb{Z}_p.
\end{equation}

For all ${\mathbf a}\in\{0,\dots,p-1\}^d$ and ${\mathbf j}\in \mathbb{N}^d$, we have 
\begin{align}
pH_{{\mathbf L}\cdot({\mathbf a}+{\mathbf j}p)}
&=p\left(\sum_{i=1}^{{\mathbf L}\cdot{\mathbf j}p}\frac{1}{i}+\sum_{i=1}^{{\mathbf L}\cdot{\mathbf a}}\frac{1}{{\mathbf L}\cdot{\mathbf j}p+i}\right)\notag\\
&\equiv p\left(\sum_{i=1}^{{\mathbf L}\cdot{\mathbf j}}\frac{1}{ip}+\sum_{i=1}^{\lfloor{\mathbf L}\cdot{\mathbf a}/p\rfloor}\frac{1}{{\mathbf L}\cdot{\mathbf j}p+ip}\right)\mod p\mathbb{Z}_p\notag\\
&\equiv H_{{\mathbf L}\cdot{\mathbf j}}+\sum_{i=1}^{\lfloor{\mathbf L}\cdot{\mathbf a}/p\rfloor}\frac{1}{{\mathbf L}\cdot{\mathbf j}+i}\mod p\mathbb{Z}_p.\label{37}
\end{align}

We need a result that we shall prove further by means of Lemma \ref{plus général} stated in Section \ref{partdemo}:

For all ${\mathbf L}\in\mathcal{E}$, ${\mathbf a}\in\{0,\dots,p-1\}^d$ and ${\mathbf j}\in\mathbb{N}^d$, we have
\begin{equation}\label{intermed réécriture}
\mathcal{Q}({\mathbf a}+{\mathbf j}p)\sum_{i=1}^{\lfloor{\mathbf L}\cdot{\mathbf a}/p\rfloor}\frac{1}{{\mathbf L}\cdot{\mathbf j}+i}\in p\mathbb{Z}_p.
\end{equation}

Applying \eqref{intermed réécriture} to \eqref{37} and with the fact that $\mathcal{Q}({\mathbf a}+{\mathbf j}p)\in\mathbb{Z}_p$ and $\mathcal{Q}({\mathbf K}-{\mathbf j})\in\mathbb{Z}_p$, we obtain 
$$
\mathcal{Q}({\mathbf K}-{\mathbf j})\mathcal{Q}({\mathbf a}+{\mathbf j}p)pH_{{\mathbf L}\cdot({\mathbf a}+{\mathbf j}p)}\equiv \mathcal{Q}({\mathbf K}-{\mathbf j})\mathcal{Q}({\mathbf a}+{\mathbf j}p)H_{{\mathbf L}\cdot{\mathbf j}}\mod p\mathbb{Z}_p.
$$ 
This leads to
\begin{align*}
\Phi_{{\mathbf L},p}({\mathbf a}+{\mathbf K}p)
&=\sum_{{\mathbf 0}\leq{\mathbf j}\leq{\mathbf K}}\mathcal{Q}({\mathbf K}-{\mathbf j})\mathcal{Q}({\mathbf a}+{\mathbf j}p)(H_{{\mathbf L}\cdot({\mathbf K}-{\mathbf j})}-pH_{{\mathbf L}\cdot ({\mathbf a}+{\mathbf j}p)})\\
&\equiv\sum_{{\mathbf 0}\leq{\mathbf j}\leq{\mathbf K}}\mathcal{Q}({\mathbf K}-{\mathbf j})\mathcal{Q}({\mathbf a}+{\mathbf j}p)(H_{{\mathbf L}\cdot({\mathbf K}-{\mathbf j})}-H_{{\mathbf L}\cdot{\mathbf j}})\mod p\mathbb{Z}_p\\
&\equiv -\sum_{{\mathbf 0}\leq{\mathbf j}\leq{\mathbf K}}H_{{\mathbf L}\cdot{\mathbf j}}\left(\mathcal{Q}({\mathbf a}+{\mathbf j}p)\mathcal{Q}({\mathbf K}-{\mathbf j})-\mathcal{Q}({\mathbf j})\mathcal{Q}({\mathbf a}+({\mathbf K}-{\mathbf j})p)\right) \mod p\mathbb{Z}_p,
\end{align*}
which is the expected equation \eqref{but réécriture}.

We now use a Krattenthaler and Rivoal's combinatorial lemma (see \cite[Lemma 5, p. 14]{Tanguy 2}) which enables us to write 
\begin{multline*}
\sum_{{\mathbf 0}\leq{\mathbf j}\leq{\mathbf K}}H_{{\mathbf L}\cdot{\mathbf j}}\left(\mathcal{Q}({\mathbf a}+{\mathbf j}p)\mathcal{Q}({\mathbf K}-{\mathbf j})-\mathcal{Q}({\mathbf j})\mathcal{Q}({\mathbf a}+({\mathbf K}-{\mathbf j})p)\right)\\
=\sum_{s=0}^{r-1}\sum_{{\mathbf 0}\leq{\mathbf m}\leq(p^{r-s}-1){\mathbf 1}}W_{{\mathbf L}}({\mathbf a},{\mathbf K},s,p,{\mathbf m}),
\end{multline*}
where $r$ is such that $p^{r-1}>\max(K_1,\dots,K_d)$ and
$$
W_{{\mathbf L}}({\mathbf a},{\mathbf K},s,p,{\mathbf m}):=S({\mathbf a},{\mathbf K},s,p,{\mathbf m})(H_{{\mathbf L}\cdot{\mathbf m}p^s}-H_{{\mathbf L}\cdot\lfloor{\mathbf m}/p\rfloor p^{s+1}}).
$$
If we prove that, for all $s\in\mathbb{N}$ and ${\mathbf m}\in\mathbb{N}^d$, we have $W_{{\mathbf L}}({\mathbf a},{\mathbf K},s,p,{\mathbf m})\in p\mathbb{Z}_p$, then we will have $\Phi_{{\mathbf L},p}({\mathbf a}+{\mathbf K}p)\in p\mathbb{Z}_p$, as expected.

For all ${\mathbf m}\in\mathbb{N}^d$, we set $\mu_p({\mathbf m}):=\sum_{\ell=1}^{\infty}{\mathbf 1}_{\mathcal{D}}(\{{\mathbf m}/p^{\ell}\})$ and $g_p({\mathbf m}):=p^{\mu_p({\mathbf m})}$, where ${\mathbf 1}_{\mathcal{D}}$ is the characteristic function of $\mathcal{D}$. We now use the following lemma which we will prove in Section \ref{partdemo}.

\begin{lemme}\label{lemme valuation H}
For all prime number $p$, all ${\mathbf L}\in\mathcal{E}$, ${\mathbf m}\in\mathbb{N}^d$ and $s\in\mathbb{N}$, we have
$$
p^{s+1}g_p({\mathbf m})\left(H_{{\mathbf L}\cdot{\mathbf m}p^s}-H_{{\mathbf L}\cdot\lfloor{\mathbf m}/p\rfloor p^{s+1}}\right)\in p\mathbb{Z}_p.
$$
\end{lemme}

According to Lemma \ref{lemme valuation H}, if we prove that, for all ${\mathbf a}\in\{0,\dots,p-1\}^d$, ${\mathbf K},{\mathbf m}\in\mathbb{N}^d$ and $s\in\mathbb{N}$, we have $S({\mathbf a},{\mathbf K},s,p,{\mathbf m})\in p^{s+1}g_p({\mathbf m})\mathbb{Z}_p$, then we will have $q_{{\mathbf L}}({\mathbf z})\in\mathbb{Z}_p[[{\mathbf z}]]$, which is the announced reformulation.

\subsubsection{Proofs of \eqref{intermed réécriture} and Lemma \ref{lemme valuation H}}\label{partdemo}

We state a result which enables us to prove \eqref{intermed réécriture} and Lemma \ref{lemme valuation H}.

\begin{lemme}\label{plus général}
Given $s\in\mathbb{N}$, $s\geq 1$, ${\mathbf a}\in\{0,\dots,p^s-1\}^d$, ${\mathbf m}\in\mathbb{N}^d$ and ${\mathbf L}\in\mathcal{E}$. If we have $\lfloor{\mathbf L}\cdot{\mathbf a}/p^s\rfloor\geq 1$, then, for all $u\in\{1,\dots,\lfloor{\mathbf L}\cdot{\mathbf a}/p^s\rfloor\}$ and $\ell\in\{s,\dots,s+v_p({\mathbf L}\cdot{\mathbf m}+u)\}$, we have 
$$
\left\{\frac{{\mathbf a}+{\mathbf m}p^s}{p^{\ell}}\right\}\in\mathcal{D}.
$$
\end{lemme}

\begin{proof}
We recall that $\mathcal{D}$ is the set of all ${\mathbf x}\in[0,1[^d$ such that there exists an element ${\mathbf d}$ of $e$ or $f$ satisfying ${\mathbf d}\cdot{\mathbf x}\geq 1$. We have $\left\{\frac{{\mathbf a}+{\mathbf m}p^s}{p^{\ell}}\right\}\in[0,1[^d$, so we only have to prove that ${\mathbf L}\cdot\left\{\frac{{\mathbf a}+{\mathbf m}p^s}{p^{\ell}}\right\}\geq 1$. Indeed, as ${\mathbf L}\in\mathcal{E}$, there exists ${\mathbf d}\in\{{\mathbf e}_1,\dots,{\mathbf e}_{q_1},{\mathbf f}_1,\dots,{\mathbf f}_{q_2}\}$ such that ${\mathbf d}\geq{\mathbf L}$, which leads to
$$
{\mathbf L}\cdot\left\{\frac{{\mathbf a}+p^s{\mathbf m}}{p^{\ell}}\right\}\geq 1\Rightarrow{\mathbf d}\cdot\left\{\frac{{\mathbf a}+p^s{\mathbf m}}{p^{\ell}}\right\}\geq 1\Rightarrow\left\{\frac{{\mathbf a}+p^s{\mathbf m}}{p^{\ell}}\right\}\in\mathcal{D}.
$$
We write ${\mathbf m}=\sum_{j=0}^{\infty}{\mathbf m}_{j}p^j$ with ${\mathbf m}_j\in\{0,\dots,p-1\}^d$. We have 
$$
\left\{\frac{{\mathbf a}+{\mathbf m}p^s}{p^{\ell}}\right\}=\frac{{\mathbf a}+p^s\sum_{j=0}^{\ell-s-1}{\mathbf m}_{j}p^j}{p^\ell}.
$$
We have $p^{\ell-s}$ divide $(u+{\mathbf L}\cdot{\mathbf m})$ and so $p^{\ell-s}$ divides 
$$
u+{\mathbf L}\cdot{\mathbf m}-{\mathbf L}\cdot\left(\sum_{j=\ell-s}^{\infty}{\mathbf m}_{j}p^j\right)=u+{\mathbf L}\cdot\left(\sum_{j=0}^{\ell-s-1}{\mathbf m}_{j}p^j\right).
$$ 
Thus, we obtain $p^{\ell-s}\leq u+{\mathbf L}\cdot\left(\sum_{j=0}^{\ell-s-1}{\mathbf m}_{j}p^j\right)\leq \frac{1}{p^s}{\mathbf L}\cdot{\mathbf a}+{\mathbf L}\cdot\left(\sum_{j=0}^{\ell-s-1}{\mathbf m}_{j}p^j\right)$ and we have
$$
1\leq\frac{{\mathbf L}\cdot{\mathbf a}+p^s{\mathbf L}\cdot\left(\sum_{j=0}^{\ell-s-1}{\mathbf m}_{j}p^j\right)}{p^{\ell}}={\mathbf L}\cdot\left\{\frac{{\mathbf a}+{\mathbf m}p^s}{p^{\ell}}\right\}.
$$
\end{proof}

We will now apply Lemma \ref{plus général} to prove \eqref{intermed réécriture}.

\begin{proof}[Proof of \eqref{intermed réécriture}]
Given ${\mathbf L}\in\mathcal{E}$, ${\mathbf a}\in\{0,\dots,p-1\}^d$ and $j\in\mathbb{N}^d$, we have to prove that $\mathcal{Q}({\mathbf a}+{\mathbf j}p)\sum_{i=1}^{\lfloor{\mathbf L}\cdot{\mathbf a}/p\rfloor}\frac{1}{{\mathbf L}\cdot{\mathbf j}+i}\in p\mathbb{Z}_p$. If $\lfloor{\mathbf L}\cdot{\mathbf a}/p\rfloor=0$, it is evident. Thus let us assume that $\lfloor{\mathbf L}\cdot{\mathbf a}/p\rfloor\geq 1$. Applying Lemma \ref{plus général} with $s=1$ and ${\mathbf m}={\mathbf j}$, we obtain that, for all $i\in\{1,\dots,\lfloor{\mathbf L}\cdot{\mathbf a}/p\rfloor\}$ and $\ell\in\{1,\dots,1+v_p(i+{\mathbf L}\cdot{\mathbf j})\}$, we have $\{({\mathbf a}+{\mathbf j}p)/p^{\ell}\}\in\mathcal{D}$ and so $\Delta(({\mathbf a}+{\mathbf j}p)/p^{\ell})\geq 1$. Since $\Delta\geq 0$ on $\mathbb{R}^d$, we get
$$
v_p(\mathcal{Q}({\mathbf a}+{\mathbf j}p))=\sum_{\ell=1}^{\infty}\Delta\left(\left\{\frac{{\mathbf a}+{\mathbf j}p}{p^{\ell}}\right\}\right)\geq\sum_{\ell=1}^{1+v_p({\mathbf L}\cdot{\mathbf j}+i)}\Delta\left(\left\{\frac{{\mathbf a}+{\mathbf j}p}{p^{\ell}}\right\}\right)\geq 1+v_p({\mathbf L}\cdot{\mathbf j}+i),
$$ 
which finishes the proof of \eqref{intermed réécriture}.
\end{proof}

\begin{proof}[Proof of Lemma \ref{lemme valuation H}]
Given ${\mathbf L}\in\mathcal{E}$, ${\mathbf m}\in\mathbb{N}^d$ and $s\in\mathbb{N}$, we have to prove that 
$$
p^{s+1}g_p({\mathbf m})(H_{{\mathbf L}\cdot{\mathbf m}p^s}-H_{{\mathbf L}\cdot\lfloor{\mathbf m}/p\rfloor p^{s+1}})\in p\mathbb{Z}_p. 
$$
We write ${\mathbf m}={\mathbf b}+{\mathbf q}p$ where ${\mathbf b}\in\{0,\dots,p-1\}^d$ and ${\mathbf q}\in\mathbb{N}^d$. Then we have ${\mathbf L}\cdot{\mathbf m}p^s={\mathbf L}\cdot{\mathbf b}p^s+{\mathbf L}\cdot{\mathbf q}p^{s+1}$ and ${\mathbf L}\cdot\lfloor{\mathbf m}/p\rfloor p^{s+1}={\mathbf L}\cdot{\mathbf q}p^{s+1}$. Therefore, we get 
$$
H_{{\mathbf L}\cdot{\mathbf m}p^s}-H_{{\mathbf L}\cdot\lfloor{\mathbf m}/p\rfloor p^{s+1}}=\sum_{j=1}^{{\mathbf L}\cdot{\mathbf b}p^s}\frac{1}{{\mathbf L}\cdot{\mathbf q}p^{s+1}+j}\equiv\sum_{i=1}^{\lfloor{\mathbf L}\cdot{\mathbf b}/p\rfloor}\frac{1}{{\mathbf L}\cdot{\mathbf q}p^{s+1}+ip^{s+1}}\mod \frac{1}{p^s}\mathbb{Z}_p
$$
and so $p^{s+1}g_p({\mathbf m})(H_{{\mathbf L}\cdot{\mathbf m}p^s}-H_{{\mathbf L}\cdot\lfloor{\mathbf m}/p\rfloor p^{s+1}})\equiv g_p({\mathbf b}+{\mathbf q}p)\sum_{i=1}^{\lfloor{\mathbf L}\cdot{\mathbf b}/p\rfloor}\frac{1}{{\mathbf L}\cdot{\mathbf q}+i}\mod p\mathbb{Z}_p$. We now have to prove that $g_p({\mathbf b}+{\mathbf q}p)\sum_{i=1}^{\lfloor{\mathbf L}\cdot{\mathbf b}/p\rfloor}\frac{1}{{\mathbf L}\cdot{\mathbf q}+i}\in p\mathbb{Z}_p$. If $\lfloor{\mathbf L}\cdot{\mathbf b}/p\rfloor=0$, it is evident. Let us assume that $\lfloor{\mathbf L}\cdot{\mathbf b}/p\rfloor\geq 1$. Applying Lemma \ref{plus général} with $s=1$ and ${\mathbf q}$ instead of ${\mathbf m}$, we obtain that, for all $i\in\{1,\dots,\lfloor{\mathbf L}\cdot{\mathbf b}/p\rfloor\}$ and all $\ell\in\{1,\dots,1+v_p(i+{\mathbf L}\cdot{\mathbf q})\}$, we have $\{({\mathbf b}+{\mathbf q}p)/p^{\ell}\}\in\mathcal{D}$ and thus 
\begin{multline*}
v_p(g_p({\mathbf b}+{\mathbf q}p))=\mu_p({\mathbf b}+{\mathbf q}p)=\sum_{\ell=1}^{\infty}{\mathbf 1}_{\mathcal{D}}\left(\left\{\frac{{\mathbf b}+{\mathbf q}p}{p^{\ell}}\right\}\right)\\
\geq\sum_{\ell=1}^{1+v_p({\mathbf L}\cdot{\mathbf q}+i)}{\mathbf 1}_{\mathcal{D}}\left(\left\{\frac{{\mathbf b}+{\mathbf q}p}{p^{\ell}}\right\}\right)\geq 1+v_p({\mathbf L}\cdot{\mathbf q}+i),
\end{multline*}
which completes the proof of Lemma \ref{lemme valuation H}.
\end{proof}

\subsection{Application of Theorem \ref{theo généralisé}}

We will use Theorem \ref{theo généralisé} to finish the proof of assertions $(i)$ of Theorems \ref{critère} and \ref{critère2}. In the following sections, we will prove that, setting $\mathbf{A}_r=\mathcal{Q}$ and $\mathbf{g}_r=g_p$ for all $r\geq 0$, then there exists $\mathcal{N}\subset\bigcup_{t\geq 1}\big(\{0,\dots,p^t-1\}^d\times\{t\}\big)$ such that the sequences $(\mathbf{A}_r)_{r\geq 0}$ and $(\mathbf{g}_r)_{r\geq 0}$ satisfy assertions $(i)$, $(ii)$ and $(iii)$ of Theorem \ref{theo généralisé}. Thus, we will obtain $S(\mathbf{a},\mathbf{K},s,p,\mathbf{m})\in p^{s+1}g_p(\mathbf{m})\mathbb{Z}_p$, as expected. 

In the following sections, we check the assumptions for the application of Theorem 4.

\subsection{Verification of assertions $(i)$ and $(ii)$ of Theorem \ref{theo généralisé}}\label{verif 1 2}

We fix a prime number $p$ and we write $g:=g_p$ and $\mu:=\mu_p$. For all $r\geq 0$, we set ${\mathbf A}_r=\mathcal{Q}$ and ${\mathbf g}_r=g$. In this section, we will prove that the sequences $({\mathbf A}_r)_{r\geq 0}$ and $({\mathbf g}_r)_{r\geq 0}$ verify assertions $(i)$ and $(ii)$ of Theorem \ref{theo généralisé}. 
\medskip

For all $r\geq 0$, we have $|{\mathbf A}_r({\mathbf {\mathbf 0}})|_p=|\mathcal{Q}({\mathbf {\mathbf 0}})|_p=1$. Furthermore, for all ${\mathbf m}\in\mathbb{N}^d$, we have $v_p(g({\mathbf m}))=\mu({\mathbf m})\geq 0$, so we get $g({\mathbf m})\in\mathbb{Z}_p\setminus\{0\}$. We now have to prove that $A({\mathbf m})\in g({\mathbf m})\mathbb{Z}_p$, which amounts to proving that $\mu_p({\mathbf m})\leq v_p(\mathcal{Q}({\mathbf m}))$. This is true because, for all $\ell\in\mathbb{N}$, $\ell\geq 1$, we have $\Delta({\mathbf m}/p^{\ell})=\Delta(\{{\mathbf m}/p^{\ell}\})\geq{\mathbf 1}_{\mathcal{D}}(\{{\mathbf m}/p^{\ell}\})$, because $\Delta({\mathbf x})\geq 1$ for ${\mathbf x}\in\mathcal{D}$.

\subsection{Verification of assertion $(iii)$ of Theorem \ref{theo généralisé}}

We fix a prime number $p$ and we set
$$
\mathcal{N}:=\bigcup_{t\geq 1}\left(\left\{{\mathbf n}\in\{0,\dots,p^t-1\}^d\,:\,\forall\ell\in\{1,\dots,t\},\left\{\frac{{\mathbf n}}{p^{\ell}}\right\}\in\mathcal{D}\right\}\times\{t\}\right).
$$

\subsubsection{Verification of assertion $(b)$}

Let $({\mathbf n},t)\in\mathcal{N}$ and ${\mathbf m}\in\mathbb{N}^d$. We have to prove that $g({\mathbf n}+p^{t}{\mathbf m})\in p^{t}g({\mathbf m})\mathbb{Z}_p$. We have
\begin{align}
v_p(g({\mathbf n}+p^{t}{\mathbf m}))=\sum_{\ell= 1}^{\infty}{\mathbf 1}_{\mathcal{D}}\left(\left\{\frac{{\mathbf n}+p^{t}{\mathbf m}}{p^{\ell}}\right\}\right)&=\sum_{\ell=1}^t{\mathbf 1}_{\mathcal{D}}\left(\left\{\frac{{\mathbf n}}{p^{\ell}}\right\}\right)+\sum_{\ell=t+1}^{\infty}{\mathbf 1}_{\mathcal{D}}\left(\left\{\frac{{\mathbf n}+p^t{\mathbf m}}{p^{\ell}}\right\}\right)\notag\\
&=t+\sum_{\ell=t+1}^{\infty}{\mathbf 1}_{\mathcal{D}}\left(\left\{\frac{{\mathbf n}+p^t{\mathbf m}}{p^{\ell}}\right\}\right)\label{rrr2}.
\end{align}

Let us write ${\mathbf m}=\sum_{k=0}^{\infty}{\mathbf m}_kp^k$, where the ${\mathbf m}_k\in\{0,\dots,p-1\}^d$ are zero except for a finite number of $k$. For all $\ell\geq t+1$, we have 
$$
\left\{\frac{{\mathbf n}+p^t{\mathbf m}}{p^{\ell}}\right\}=\frac{{\mathbf n}+p^t\left(\sum_{k=0}^{\ell-t-1}{\mathbf m}_kp^k\right)}{p^{\ell}}\geq\frac{p^t\left(\sum_{k=0}^{\ell-t-1}{\mathbf m}_kp^k\right)}{p^{\ell}}=\left\{\frac{{\mathbf m}}{p^{\ell-t}}\right\}.
$$
Thus, for all $\ell\geq t+1$, if $\left\{\frac{{\mathbf m}}{p^{\ell-t}}\right\}\in\mathcal{D}$, then there exists ${\mathbf L}\in\mathcal{E}$ such that 
$$
1\leq{\mathbf L}\cdot\left\{\frac{{\mathbf m}}{p^{\ell-t}}\right\}\leq{\mathbf L}\cdot\left\{\frac{{\mathbf n}+p^t{\mathbf m}}{p^{\ell}}\right\}, 
$$
which gives us $\left\{\frac{{\mathbf n}+p^t{\mathbf m}}{p^{\ell}}\right\}\in\mathcal{D}$. We get
$$
\sum_{\ell=t+1}^{\infty}{\mathbf 1}_{\mathcal{D}}\left(\left\{\frac{{\mathbf n}+p^t{\mathbf m}}{p^{\ell}}\right\}\right)\geq\sum_{\ell=t+1}^{\infty}{\mathbf 1}_{\mathcal{D}}\left(\left\{\frac{{\mathbf m}}{p^{\ell-t}}\right\}\right)=\sum_{\ell=1}^{\infty}{\mathbf 1}_{\mathcal{D}}\left(\left\{\frac{{\mathbf m}}{p^{\ell}}\right\}\right)=v_p(g({\mathbf m})),
$$
which, associated with \eqref{rrr2}, leads to $v_p(g({\mathbf n}+p^t{\mathbf m}))\geq t+v_p(g({\mathbf m}))$, \textit{i.e.} $g({\mathbf n}+p^t{\mathbf m})\in p^tg({\mathbf m})\mathbb{Z}_p$, as expected.

\subsubsection{Verification of assertion $(a_2)$}

Given $s\in\mathbb{N}$, ${\mathbf u}\in\Psi_s(\mathcal{N})$ and ${\mathbf v}\in\{0,\dots,p-1\}^d$ such that ${\mathbf v}+p{\mathbf u}\notin\Psi_{s+1}(\mathcal{N})$, we have to prove that
\begin{equation}\label{montr a2}
\frac{\mathcal{Q}({\mathbf u}+p^s{\mathbf m})}{\mathcal{Q}({\mathbf u})}\in p^{s+1}\frac{g({\mathbf m})}{g({\mathbf v}+p{\mathbf u})}\mathbb{Z}_p.
\end{equation}

First, we give another expression for
$$
\Psi_s(\mathcal{N})=\{{\mathbf u}\in\{0,\dots,p^s-1\}^d\,:\,\forall({\mathbf n},t)\in\mathcal{N},t\leq s,\forall{\mathbf j}\in\{0,\dots,p^{s-t}-1\}^d,{\mathbf u}\neq{\mathbf j}+p^{s-t}{\mathbf n}\}.
$$
For that purpose, we need the following lemma.

\begin{lemme}\label{equiv cond}
Given $s\in\mathbb{N}$, $s\geq 1$, and ${\mathbf u}\in\{0,\dots,p^s-1\}^d$, we write ${\mathbf u}=\sum_{k=0}^{s-1}{\mathbf u}_kp^k$, with ${\mathbf u}_k\in\{0,\dots,p-1\}^d$. Then, the following assertions are equivalent.
\begin{itemize}
\item[$(1)$] We have $\left\{{\mathbf u}/p^s\right\}\in\mathcal{D}$.
\item[$(2)$] There exists $({\mathbf n},t)\in\mathcal{N},t\leq s$ and ${\mathbf j}\in\{0,\dots,p^{s-t}-1\}^d$ such that ${\mathbf u}={\mathbf j}+p^{s-t}{\mathbf n}$.
\end{itemize}
\end{lemme}

\begin{proof}[Proof of Lemma \ref{equiv cond}]
$(1)\Rightarrow(2)$ : For all $s\geq 1$, ${\mathbf u}\in\{0,\dots,p^s-1\}^d$ such that $\{{\mathbf u}/p^s\}\in\mathcal{D}$ and all $i\in\{0,\dots,s-1\}$, we write $\mathcal{A}_{s,i}({\mathbf u})$ for the assertion: for all $\ell\in\{1,\dots,s-i\}$, we have $\left\{\left(\sum_{k=i}^{s-1}{\mathbf u}_kp^{k-i}\right)/p^{\ell}\right\}\in\mathcal{D}$.

For all $s\geq 1$, we write $\mathcal{B}_s$ for the assertion: for all ${\mathbf u}\in\{0,\dots,p^s-1\}^d$ such that $\{{\mathbf u}/p^s\}\in\mathcal{D}$, there exists $i\in\{0,\dots,s-1\}$, such that $\mathcal{A}_{s,i}({\mathbf u})$ is true.
\medskip

First, we will prove by induction on $s$ that, for all $s\geq 1$, $\mathcal{B}_s$ is true. 

If $s=1$, then, for all ${\mathbf u}\in\{0,\dots,p-1\}^d$ such that $\{{\mathbf u}/p\}\in\mathcal{D}$, assertion $\mathcal{A}_{1,0}({\mathbf u})$ corresponds to no other assertion than $\{{\mathbf u}/p\}\in\mathcal{D}$ and thus is true. Hence, $\mathcal{B}_1$ is true. 

Given $s\geq 2$ such that $\mathcal{B}_1,\dots,\mathcal{B}_{s-1}$ are true, and ${\mathbf u}\in\{0,\dots,p^s-1\}^d$ verifying $\{{\mathbf u}/p^s\}\in\mathcal{D}$ such that $\mathcal{A}_{s,1}({\mathbf u}),\dots,\mathcal{A}_{s,s-1}({\mathbf u})$ are false, we will prove that assertion $\mathcal{A}_{s,0}({\mathbf u})$ is true. This will imply the validity of $\mathcal{B}_s$ and will finish the induction on~$s$. 

Let us give a proof by contradiction, assuming that there exists $\ell\in\{1,\dots,s\}$ such that 
$$
{\mathbf a}_{\ell}:=\frac{\sum_{k=0}^{\ell-1}{\mathbf u}_kp^k}{p^{\ell}}=\left\{\frac{\sum_{k=0}^{s-1}{\mathbf u}_kp^k}{p^{\ell}}\right\}\notin\mathcal{D}.
$$ 
We actually have $\ell\in\{1,\dots,s-1\}$ because $\{\mathbf{u}/p^s\}\in\mathcal{D}$. For all ${\mathbf L}\in\{{\mathbf e}_1,\dots,{\mathbf e}_{q_1},{\mathbf f}_1,\dots,{\mathbf f}_{q_2}\}$, we have ${\mathbf L}\cdot{\mathbf a}_{\ell}<1$. We write
$$
\left\{\frac{{\mathbf u}}{p^s}\right\}=\frac{{\mathbf u}}{p^s}=\frac{p^{\ell}{\mathbf a}_{\ell}+p^{\ell}\sum_{k=\ell}^{s-1}{\mathbf u}_kp^{k-\ell}}{p^{s}}=\frac{{\mathbf a}_{\ell}}{p^{s-\ell}}+\frac{\sum_{k=\ell}^{s-1}{\mathbf u}_kp^{k-\ell}}{p^{s-\ell}}.
$$
Since $\{{\mathbf u}/p^s\}\in\mathcal{D}$, there exists ${\mathbf L}\in\{{\mathbf e}_1,\dots,{\mathbf e}_{q_1},{\mathbf f}_1,\dots,{\mathbf f}_{q_2}\}$ such that
$$
1\leq{\mathbf L}\cdot\left\{\frac{{\mathbf u}}{p^s}\right\}=\frac{{\mathbf L}\cdot{\mathbf a}_{\ell}}{p^{s-\ell}}+{\mathbf L}\cdot\frac{\sum_{k=\ell}^{s-1}{\mathbf u}_kp^{k-\ell}}{p^{s-\ell}}<\frac{1}{p^{s-\ell}}+{\mathbf L}\cdot\frac{\sum_{k=\ell}^{s-1}{\mathbf u}_kp^{k-\ell}}{p^{s-\ell}},
$$
which leads to ${\mathbf L}\cdot\left(\sum_{k=\ell}^{s-1}{\mathbf u}_kp^{k-\ell}\right)>p^{s-\ell}-1$. Since ${\mathbf L}\cdot\left(\sum_{k=\ell}^{s-1}{\mathbf u}_kp^{k-\ell}\right)$ is an integer, we get ${\mathbf L}\cdot\left(\sum_{k=\ell}^{s-1}{\mathbf u}_kp^{k-\ell}\right)\geq p^{s-\ell}$, \textit{i.e.} $\left\{\left(\sum_{k=\ell}^{s-1}{\mathbf u}_kp^{k-\ell}\right)/p^{s-\ell}\right\}\in\mathcal{D}$. We write ${\mathbf v}:=\sum_{k=\ell}^{s-1}{\mathbf u}_kp^{k-\ell}\in\{0,\dots,p^{s-\ell}-1\}^d$. Thus we have $\{{\mathbf v}/p^{s-\ell}\}\in\mathcal{D}$ and, applying $\mathcal{B}_{s-\ell}$, we obtain that there exists $i\in\{0,\dots,s-\ell-1\}$ such that $\mathcal{A}_{s-\ell,i}({\mathbf v})$ is true, \textit{i.e.}, for all $r\in\{1,\dots,s-\ell-i\}$, we have $\left\{\left(\sum_{k=i}^{s-\ell-1}{\mathbf v}_kp^{k-i}\right)/p^{r}\right\}\in\mathcal{D}$. Furthermore, for all $k$, we have ${\mathbf v}_k={\mathbf u}_{\ell+k}$ and therefore $\sum_{k=i}^{s-\ell-1}{\mathbf v}_kp^{k-i}=\sum_{k=i+\ell}^{s-1}{\mathbf u}_kp^{k-i-\ell}$. Thereby, assertion $\mathcal{A}_{s-\ell,i}({\mathbf v})$ becomes: for all $r\in\{1,\dots,s-\ell-i\}$, we have $\left\{\left(\sum_{k=i+\ell}^{s-1}{\mathbf u}_kp^{k-i-\ell}\right)/p^{r}\right\}\in\mathcal{D}$; which corresponds to no other assertion than $\mathcal{A}_{s,i+\ell}({\mathbf u})$. Since we assumed that $\mathcal{A}_{s,1}({\mathbf u}),\dots,\mathcal{A}_{s,s-1}({\mathbf u})$ are false, we get a contradiction. Hence $\mathcal{A}_{s,0}({\mathbf u})$ is true and $\mathcal{B}_s$ is also true, which finishes the induction on $s$.
\medskip

As $\{{\mathbf u}/p^s\}\in\mathcal{D}$, assertion $\mathcal{B}_s$ tells us that an $i\in\{0,\dots,s-1\}$ exists such that $\mathcal{A}_{s,i}({\mathbf u})$ is true, \textit{i.e.} for all $\ell\in\{1,\dots,s-i\}$, we have $\left\{\left(\sum_{k=i}^{s-1}p^{k-i}{\mathbf u}_k\right)/p^{\ell}\right\}\in\mathcal{D}$. Thus we have $\left(\sum_{k=i}^{s-1}p^{k-i}{\mathbf u}_k,s-i\right)\in\mathcal{N}$ and ${\mathbf u}=\sum_{k=0}^{i-1}p^k{\mathbf u}_k+p^i\sum_{k=i}^{s-1}p^{k-i}{\mathbf u}_k$. Therefore, the assertion $(2)$ is valid with $s-i$ instead of $t$, $\sum_{k=i}^{s-1}p^{k-i}{\mathbf u}_k$ instead of ${\mathbf n}$ and $\sum_{k=0}^{i-1}p^k{\mathbf u}_k$ instead of~${\mathbf j}$.
\medskip

$(2)\Rightarrow(1)$ : We have
$$
\left\{\frac{{\mathbf u}}{p^s}\right\}=\frac{{\mathbf u}}{p^s}=\frac{{\mathbf j}+p^{s-t}{\mathbf n}}{p^s}\geq\frac{{\mathbf n}}{p^t}=\left\{\frac{{\mathbf n}}{p^t}\right\}\in\mathcal{D}
$$
and so $\left\{{\mathbf u}/p^s\right\}\in\mathcal{D}$, as expected.
\end{proof}

According to Lemma \ref{equiv cond}, we obtain
\begin{equation}\label{idPsi}
\Psi_s(\mathcal{N})=\{{\mathbf u}\in\{0,\dots,p^s-1\}^d\,:\,\{{\mathbf u}/p^s\}\notin\mathcal{D}\}.
\end{equation}

Thus, for all ${\mathbf u}\in\Psi_s(\mathcal{N})$ and $\ell\geq s$, we have $\{{\mathbf u}/p^{\ell}\}={\mathbf u}/p^{\ell}\leq{\mathbf u}/p^s=\{{\mathbf u}/p^s\}$, which gives us that, for all ${\mathbf L}\in\{{\mathbf e}_1,\dots,{\mathbf e}_{q_1},{\mathbf f}_1,\dots,{\mathbf f}_{q_2}\}$ and $\ell\geq s$, we have ${\mathbf L}\cdot\{{\mathbf u}/p^{\ell}\}\leq{\mathbf L}\cdot\{{\mathbf u}/p^s\}<1$ and so $\{{\mathbf u}/p^{\ell}\}\notin\mathcal{D}$. As a result, for all $\ell\geq s$, we have $\Delta(\{{\mathbf u}/p^{\ell}\})=0$ and thus
$$
v_p\left(\mathcal{Q}({\mathbf u})\right)=\sum_{\ell=1}^{\infty}\Delta\left(\left\{\frac{{\mathbf u}}{p^{\ell}}\right\}\right)=\sum_{\ell=1}^{s}\Delta\left(\left\{\frac{{\mathbf u}}{p^{\ell}}\right\}\right).
$$
Furthermore, we have
$$
v_p\left(\mathcal{Q}({\mathbf u}+p^s{\mathbf m})\right)=\sum_{\ell=1}^{\infty}\Delta\left(\left\{\frac{{\mathbf u}+p^s{\mathbf m}}{p^{\ell}}\right\}\right)=\sum_{\ell=1}^s\Delta\left(\left\{\frac{{\mathbf u}}{p^{\ell}}\right\}\right)+\sum_{\ell=s+1}^{\infty}\Delta\left(\left\{\frac{{\mathbf u}+p^s{\mathbf m}}{p^{\ell}}\right\}\right),
$$
which leads to
\begin{equation}\label{eqtrans}
v_p\left(\frac{\mathcal{Q}({\mathbf u}+p^s{\mathbf m})}{\mathcal{Q}({\mathbf u})}\right)=\sum_{\ell=s+1}^{\infty}\Delta\left(\left\{\frac{{\mathbf u}+p^s{\mathbf m}}{p^{\ell}}\right\}\right).
\end{equation}
We write ${\mathbf m}=\sum_{k=0}^{\infty}p^k{\mathbf m}_k$, with ${\mathbf m}_k\in\{0,\dots,p-1\}^d$. For all $\ell\geq s+1$, we have 
$$
\left\{\frac{{\mathbf u}+p^s{\mathbf m}}{p^{\ell}}\right\}=\frac{{\mathbf u}+p^s\sum_{k=0}^{\ell-1-s}p^k{\mathbf m}_k}{p^{\ell}}\geq\frac{\sum_{k=0}^{\ell-1-s}p^k{\mathbf m}_k}{p^{\ell-s}}=\left\{\frac{{\mathbf m}}{p^{\ell-s}}\right\}
$$
and thus 
\begin{align}
\sum_{\ell=s+1}^{\infty}\Delta\left(\left\{\frac{{\mathbf u}+p^s{\mathbf m}}{p^{\ell}}\right\}\right)&\geq\sum_{\ell=s+1}^{\infty}{\mathbf 1}_{\mathcal{D}}\left(\left\{\frac{{\mathbf u}+p^s{\mathbf m}}{p^{\ell}}\right\}\right)\label{explik2}\\
&\geq\sum_{\ell=s+1}^{\infty}{\mathbf 1}_{\mathcal{D}}\left(\left\{\frac{{\mathbf m}}{p^{\ell-s}}\right\}\right)=\sum_{\ell=1}^{\infty}{\mathbf 1}_{\mathcal{D}}\left(\left\{\frac{{\mathbf m}}{p^{\ell}}\right\}\right)=v_p(g({\mathbf m})),\label{trans23}
\end{align}
where inequality \eqref{explik2} is true because, for all ${\mathbf x}\in\mathcal{D}$, we have $\Delta({\mathbf x})\geq 1$. Applying \eqref{trans23} to \eqref{eqtrans}, we get
$$
v_p\left(\frac{\mathcal{Q}({\mathbf u}+p^s{\mathbf m})}{\mathcal{Q}({\mathbf u})}\right)\geq v_p(g({\mathbf m})).
$$
\medskip

Thus, to verify assertion $(a_2)$, we only have to prove that, for all ${\mathbf u}\in\Psi_s(\mathcal{N})$ and ${\mathbf v}\in\{0,\dots,p-1\}^d$ such that ${\mathbf v}+p{\mathbf u}\notin\Psi_{s+1}(\mathcal{N})$, we have $g({\mathbf v}+p{\mathbf u})\in p^{s+1}\mathbb{Z}_p$. 

We write ${\mathbf u}=\sum_{k=0}^{s-1}p^k{\mathbf u}_k$, with ${\mathbf u}_k\in\{0,\dots,p-1\}^d$. We have $\{({\mathbf v}+p{\mathbf u})/p\}={\mathbf v}/p$ and, for all $\ell\geq 2$, we have $\{({\mathbf v}+p{\mathbf u})/p^{\ell}\}=\left({\mathbf v}+p\sum_{k=0}^{\ell-2}p^k{\mathbf u}_k\right)/p^{\ell}$. We get
$$
v_p(g({\mathbf v}+p{\mathbf u}))=\sum_{\ell=1}^{\infty}{\mathbf 1}_{\mathcal{D}}\left(\left\{\frac{{\mathbf v}+p{\mathbf u}}{p^{\ell}}\right\}\right)\geq{\mathbf 1}_{\mathcal{D}}\left(\frac{{\mathbf v}}{p}\right)+\sum_{\ell=2}^{s+1}{\mathbf 1}_{\mathcal{D}}\left(\frac{{\mathbf v}+p\sum_{k=0}^{\ell-2}p^k{\mathbf u}_k}{p^{\ell}}\right).
$$

Thus, if we prove that ${\mathbf v}/p\in\mathcal{D}$ and that $\left({\mathbf v}+p\sum_{k=0}^{\ell-2}p^k{\mathbf u}_k\right)/p^{\ell}\in\mathcal{D}$ for all $\ell\in\{2,\dots,s+1\}$, then we would have $v_p(g({\mathbf v}+p{\mathbf u}))\geq s+1$.
\medskip
\begin{itemize}
\item Let us prove that ${\mathbf v}/p\in\mathcal{D}$. 
\end{itemize}
\medskip

As ${\mathbf v}+p{\mathbf u}\notin\Psi_{s+1}(\mathcal{N})$, we obtain, according to \eqref{idPsi}, that $\{({\mathbf v}+p{\mathbf u})/p^{s+1}\}\in\mathcal{D}$. Thus there exists ${\mathbf L}\in\{{\mathbf e}_1,\dots,{\mathbf e}_{q_1},{\mathbf f}_1,\dots,{\mathbf f}_{q_2}\}$ such that ${\mathbf L}\cdot\{({\mathbf v}+p{\mathbf u})/p^{s+1}\}\geq 1$. We get
\begin{equation}\label{trans32}
1\leq{\mathbf L}\cdot\frac{{\mathbf v}+p\sum_{k=0}^{s-1}p^k{\mathbf u}_k}{p^{s+1}}={\mathbf L}\cdot\frac{{\mathbf v}}{p^{s+1}}+{\mathbf L}\cdot\frac{\sum_{k=0}^{s-1}p^k{\mathbf u}_k}{p^s}={\mathbf L}\cdot\frac{{\mathbf v}}{p^{s+1}}+{\mathbf L}\cdot\left\{\frac{{\mathbf u}}{p^s}\right\}.
\end{equation}

As ${\mathbf u}\in\Psi_s(\mathcal{N})$, we have $\{{\mathbf u}/p^s\}\notin\mathcal{D}$ and so ${\mathbf L}\cdot\{{\mathbf u}/p^s\}<1$. We have ${\mathbf L}\cdot\{{\mathbf u}/p^s\}\in\frac{1}{p^s}\mathbb{N}$ thus ${\mathbf L}\cdot\{{\mathbf u}/p^s\}\leq(p^s-1)/p^s$ and we get, \textit{via} inequality \eqref{trans32}, that ${\mathbf L}\cdot{\mathbf v}/p^{s+1}\geq 1/p^s$, \textit{i.e.} ${\mathbf L}\cdot{\mathbf v}/p\geq 1$. Thereby, we have ${\mathbf v}/p\in\mathcal{D}$.
\medskip
\begin{itemize}
\item Let us prove that, for all $\ell\in\{2,\dots,s+1\}$, we have $\left({\mathbf v}+p\sum_{k=0}^{\ell-2}p^k{\mathbf u}_k\right)/p^{\ell}\in\mathcal{D}$.
\end{itemize}
\medskip
We assume that $s\geq 1$. Given $\ell\in\{2,\dots,s+1\}$, we have
\begin{equation}\label{louder}
1\leq{\mathbf L}\cdot\frac{{\mathbf v}+p\sum_{k=0}^{s-1}p^k{\mathbf u}_k}{p^{s+1}}={\mathbf L}\cdot\frac{{\mathbf v}+p\sum_{k=0}^{\ell-2}p^k{\mathbf u}_k}{p^{s+1}}+{\mathbf L}\cdot\frac{p\sum_{k=\ell-1}^{s-1}p^{k}{\mathbf u}_k}{p^{s+1}}.
\end{equation}
We have ${\mathbf u}\in\Psi_s(\mathcal{N})$ and ${\mathbf u}={\mathbf u}_0+p\sum_{k=1}^{s-1}p^{k-1}{\mathbf u}_k$. Thus, applying \eqref{Psi+1} with $t=0$, we obtain $\sum_{k=1}^{s-1}p^{k-1}{\mathbf u}_k\in\Psi_{s-1}(\mathcal{N})$. Iterating \eqref{Psi+1}, we finally get that $\sum_{k=\ell-1}^{s-1}p^{k-\ell+1}{\mathbf u}_k\in\Psi_{s-\ell+1}(\mathcal{N})$. Following Lemma \ref{equiv cond}, we get
$$
\frac{p\sum_{k=\ell-1}^{s-1}p^{k}{\mathbf u}_k}{p^{s+1}}=\frac{\sum_{k=\ell-1}^{s-1}p^{k-\ell+1}{\mathbf u}_k}{p^{s-\ell+1}}=\left\{\frac{\sum_{k=\ell-1}^{s-1}p^{k-\ell+1}{\mathbf u}_k}{p^{s-\ell+1}}\right\}\notin\mathcal{D}.
$$
In particular, we obtain $1>{\mathbf L}\cdot\left(\sum_{k=\ell-1}^{s-1}p^{k-\ell+1}{\mathbf u}_k\right)/p^{s-\ell+1}\in\frac{1}{p^{s-\ell+1}}\mathbb{N}$. Thus
we have ${\mathbf L}\cdot\left(\sum_{k=\ell-1}^{s-1}p^{k-\ell+1}{\mathbf u}_k\right)/p^{s-\ell+1}\leq (p^{s-\ell+1}-1)/p^{s-\ell+1}$. Using this latest inequality in \eqref{louder}, we get
$$
{\mathbf L}\cdot\frac{{\mathbf v}+p\sum_{k=0}^{\ell-2}p^k{\mathbf u}_k}{p^{s+1}}\geq\frac{1}{p^{s-\ell+1}}.
$$

Therefore, for all $\ell\in\{2,\dots,s+1\}$, we have 
\begin{equation}\label{recap765}
{\mathbf L}\cdot\left\{\frac{{\mathbf v}+p{\mathbf u}}{p^{\ell}}\right\}={\mathbf L}\cdot\frac{{\mathbf v}+p\sum_{k=0}^{\ell-2}p^k{\mathbf u}_k}{p^{\ell}}\geq 1
\end{equation}
and, for all $\ell\in\{2,\dots,s+1\}$, we obtain $\left\{\left({\mathbf v}+p{\mathbf u}\right)/p^{\ell}\right\}\in\mathcal{D}$. This completes the verification of assertion $(a_2)$.

\subsubsection{Verification of assertions $(a)$ and $(a_1)$}

For all $s\in\mathbb{N}$, ${\mathbf v}\in\{0,\dots,p-1\}^d$ and ${\mathbf u}\in\Psi_s(\mathcal{N})$, we set $\theta_s({\mathbf v}+{\mathbf u}p):=\mathcal{Q}({\mathbf v}+{\mathbf u}p)$ if ${\mathbf v}+{\mathbf u}p\notin\Psi_{s+1}(\mathcal{N})$, and $\theta_s({\mathbf v}+{\mathbf u}p):=g({\mathbf v}+{\mathbf u}p)$ if ${\mathbf v}+{\mathbf u}p\in\Psi_{s+1}(\mathcal{N})$.

The aim of this section is to prove the following assertion: for all $s\in\mathbb{N}$, ${\mathbf v}\in\{0,\dots,p-1\}^d$, ${\mathbf u}\in\Psi_s(\mathcal{N})$ and ${\mathbf m}\in\mathbb{N}^d$, we have
\begin{equation}\label{Hypothèse (iii)}
\frac{\mathcal{Q}({\mathbf v}+{\mathbf u}p+{\mathbf m}p^{s+1})}{\mathcal{Q}({\mathbf v}+{\mathbf u}p)}-\frac{\mathcal{Q}({\mathbf u}+{\mathbf m}p^{s})}{\mathcal{Q}({\mathbf u})}\in p^{s+1}\frac{g({\mathbf m})}{\theta_s({\mathbf v}+{\mathbf u}p)}\mathbb{Z}_p,
\end{equation}
which will prove assertions $(a)$ and $(a_1)$ of Theorem \ref{theo généralisé}. Indeed, for all ${\mathbf v}\in\{0,\dots,p-1\}^d$ and ${\mathbf u}\in\Psi_s(\mathcal{N})$, we have $\mathcal{Q}({\mathbf v}+{\mathbf u}p)\in g({\mathbf v}+{\mathbf u}p)\mathbb{Z}_p$ so that 
$$
p^{s+1}\frac{g({\mathbf m})}{\theta_s({\mathbf v}+{\mathbf u}p)}\in p^{s+1}\frac{g({\mathbf m})}{\mathcal{Q}({\mathbf v}+{\mathbf u}p)}\mathbb{Z}_p
$$
and \eqref{Hypothèse (iii)} implies $(a)$. Furthermore, according to the definition of $\theta_s$, when ${\mathbf v}+{\mathbf u}p\in\Psi_{s+1}(\mathcal{N})$, congruence \eqref{Hypothèse (iii)} implies $(a_1)$.

Congruence \eqref{Hypothèse 
 (iii)} is valid if and only if, for all ${\mathbf v}\in\{0,\dots,p-1\}^d$, ${\mathbf u}\in\Psi_s(\mathcal{N})$ and ${\mathbf m}\in\mathbb{N}^d$, we have
$$
\left(1-\frac{\mathcal{Q}({\mathbf v}+{\mathbf u}p)}{\mathcal{Q}({\mathbf u})}\frac{\mathcal{Q}({\mathbf u}+{\mathbf m}p^{s})}{\mathcal{Q}({\mathbf v}+{\mathbf u}p+{\mathbf m}p^{s+1})}\right)\frac{\mathcal{Q}({\mathbf v}+{\mathbf u}p+{\mathbf m}p^{s+1})}{\mathcal{Q}({\mathbf v}+{\mathbf u}p)}\in p^{s+1}\frac{g({\mathbf m})}{\theta_s({\mathbf v}+{\mathbf u}p)}\mathbb{Z}_p.
$$
In the sequel of the proof, we set 
$$
X_s({\mathbf v},{\mathbf u},{\mathbf m}):=\frac{\mathcal{Q}({\mathbf v}+{\mathbf u}p)}{\mathcal{Q}({\mathbf u})}\frac{\mathcal{Q}({\mathbf u}+{\mathbf m}p^{s})}{\mathcal{Q}({\mathbf v}+{\mathbf u}p+{\mathbf m}p^{s+1})}.
$$ 
Thus, to prove \eqref{Hypothèse (iii)}, we only have to prove that
\begin{equation}\label{eq 3.13}
(X_s({\mathbf v},{\mathbf u},{\mathbf m})-1)\frac{\mathcal{Q}({\mathbf v}+{\mathbf u}p+{\mathbf m}p^{s+1})}{g({\mathbf m})}\in p^{s+1}\frac{\mathcal{Q}({\mathbf v}+{\mathbf u}p)}{\theta_s({\mathbf v}+{\mathbf u}p)}\mathbb{Z}_p.
\end{equation}

In order to estimate the valuation of $X_s({\mathbf v},{\mathbf u},{\mathbf m})-1$, let us set, for all ${\mathbf v}\in\{0,\dots,p-1\}^d$, ${\mathbf u}\in\{0,\dots,p^s-1\}^d$, $s\in\mathbb{N}$ and ${\mathbf m}\in\mathbb{N}^d$, 
$$
Y_s({\mathbf v},{\mathbf u},{\mathbf m}):=\frac{\prod_{i=1}^{q_2}\prod_{j=1}^{\lfloor {\mathbf f}_i\cdot{\mathbf v}/p\rfloor}\left(1+\frac{{\mathbf f}_i\cdot{\mathbf m}p^{s}}{{\mathbf f}_i\cdot{\mathbf u}+j}\right)}{\prod_{i=1}^{q_1}\prod_{j=1}^{\lfloor {\mathbf e}_i\cdot{\mathbf v}/p\rfloor}\left(1+\frac{{\mathbf e}_i\cdot{\mathbf m}p^{s}}{{\mathbf e}_i\cdot{\mathbf u}+j}\right)}.
$$
Given $s\in\mathbb{N}$, ${\mathbf m}\in\mathbb{N}^d$ and ${\mathbf a}\in\{0,\dots,p^s-1\}^d$, we write $\eta_s({\mathbf a},{\mathbf m}):=\sum_{\ell=s+1}^{\infty}\Delta\left(\left\{\frac{{\mathbf a}+{\mathbf m}p^{s}}{p^{\ell}}\right\}\right)$. We state four lemmas, which we prove in Section \ref{démo lemme Y}.

\begin{lemme}\label{définition de Y}
For all $s\in\mathbb{N}$, ${\mathbf v}\in\{0,\dots,p-1\}^d$, ${\mathbf u}\in\Psi_s(\mathcal{N})$ and ${\mathbf m}\in\mathbb{N}^d$, we have $X_s({\mathbf v},{\mathbf u},{\mathbf m})\in Y_s({\mathbf v},{\mathbf u},{\mathbf m})\left(1+p^{s+1}\mathbb{Z}_p\right)$ and $v_p(Y_s({\mathbf v},{\mathbf u},{\mathbf m}))\geq\eta_{s}({\mathbf u},{\mathbf m})-\eta_{s+1}({\mathbf v}+{\mathbf u}p,{\mathbf m})$.
\end{lemme}

\begin{lemme}\label{lemme theta Q(a)}
Given $s\in\mathbb{N}$, ${\mathbf v}\in\{0,\dots,p-1\}^d$ and ${\mathbf u}\in\{0,\dots,p^{s}-1\}^d$, if there exists $j\in\{1,\dots,s+1\}$ such that $\{({\mathbf v}+{\mathbf u}p)/p^{j}\}\notin\mathcal{D}$, then we have $Y_s({\mathbf v},{\mathbf u},{\mathbf m})\in 1+p^{s-j+2}\mathbb{Z}_p$.
\end{lemme}

\begin{lemme}\label{valuation du quotient avec g}
For all $s\in\mathbb{N}$, ${\mathbf a}\in\{0,\dots,p^{s+1}-1\}^d$ and ${\mathbf m}\in\mathbb{N}^d$, we have
\begin{equation}\label{eq du lemme 20}
\eta_{s+1}({\mathbf a},{\mathbf m})\geq \mu({\mathbf m})
\end{equation}
and
\begin{equation}\label{cor valuation du quotient avec g}
v_p\left(\frac{\mathcal{Q}({\mathbf a}+{\mathbf m}p^{s+1})}{g({\mathbf m})}\right)\geq \sum_{\ell=1}^{s+1}\Delta\left(\left\{\frac{{\mathbf a}}{p^{\ell}}\right\}\right).
\end{equation}
\end{lemme}

\begin{lemme}\label{lemme Q(a) g(a)}
Given $s\in\mathbb{N}$ and ${\mathbf a}\in\Psi_s(\mathcal{N})$, we have $v_p(Q({\mathbf a}))=\sum_{\ell=1}^{s}\Delta\left(\left\{\frac{{\mathbf a}}{p^{\ell}}\right\}\right)$.
\end{lemme}

In order to prove \eqref{eq 3.13}, we will now distinguish two cases. 

\medskip
\begin{itemize}
\item \textit{Case} 1: Let us assume that there exists $j\in\{1,\dots,s+1\}$ such that 
\end{itemize}
\begin{equation}\label{ouhou}
\left\{\frac{{\mathbf v}+{\mathbf u}p}{p^j}\right\}\notin\mathcal{D}.
\end{equation}
Let $j_0$ be the smaller $j\in\{1,\dots,s+1\}$ verifying \eqref{ouhou}. According to Lemma \ref{lemme theta Q(a)} applied with $j_0$, we get $Y_s({\mathbf v},{\mathbf u},{\mathbf m})\in 1+p^{s-j_0+2}\mathbb{Z}_p$ and thus, following Lemma \ref{définition de Y}, $v_p(X_s({\mathbf v},{\mathbf u},{\mathbf m})-1)\geq s-j_0+2$. According to \eqref{cor valuation du quotient avec g}, we get 
\begin{align}
v_p\left((X_s({\mathbf v},{\mathbf u},{\mathbf m})-1)\frac{\mathcal{Q}({\mathbf v}+{\mathbf u}p+{\mathbf m}p^{s+1})}{g({\mathbf m})}\right)
&\geq v_p(X_s({\mathbf v},{\mathbf u},{\mathbf m})-1)+\sum_{\ell=1}^{s+1}\Delta\left(\left\{\frac{{\mathbf v}+{\mathbf u}p}{p^{\ell}}\right\}\right)\notag\\
&\geq s-j_0+2+\sum_{\ell=1}^{s+1}\Delta\left(\left\{\frac{{\mathbf v}+{\mathbf u}p}{p^{\ell}}\right\}\right).\label{intertheta}
\end{align}

For all $\ell\in\{1,\dots,j_0-1\}$, we have $\{({\mathbf v}+{\mathbf u}p)/p^{\ell}\}\in\mathcal{D}$ and so $\Delta(\{({\mathbf v}+{\mathbf u}p)/p^{\ell}\})\geq 1$. We get $\sum_{\ell=1}^{s+1}\Delta(\{({\mathbf v}+{\mathbf u}p)/p^{\ell}\})\geq j_0-1$ which, associated with \eqref{intertheta}, leads to
\begin{equation}\label{concl}
v_p\left((X_s({\mathbf v},{\mathbf u},{\mathbf m})-1)\frac{\mathcal{Q}({\mathbf v}+{\mathbf u}p+{\mathbf m}p^{s+1})}{g({\mathbf m})}\right)\geq s+1.
\end{equation}

If ${\mathbf v}+{\mathbf u}p\notin\Psi_{s+1}(\mathcal{N})$, then we have $\theta_s({\mathbf v}+{\mathbf u}p)=\mathcal{Q}({\mathbf v}+{\mathbf u}p)$ and $p^{s+1}\frac{\mathcal{Q}({\mathbf v}+{\mathbf u}p)}{\theta_s({\mathbf v}+{\mathbf u}p)}=p^{s+1}$. Hence, when ${\mathbf v}+{\mathbf u}p\notin\Psi_{s+1}(\mathcal{N})$, inequality \eqref{concl} implies \eqref{eq 3.13}.

We assume, throughout the end of the proof of Case $1$, that ${\mathbf v}+{\mathbf u}p\in\Psi_{s+1}(\mathcal{N})$, thus $\theta_s({\mathbf v}+{\mathbf u}p)=g({\mathbf v}+{\mathbf u}p)$. Let us prove that we have $v_p(g({\mathbf v}+{\mathbf u}p))\geq j_0-1$. Indeed, for all $\ell\in\{1,\dots,j_0-1\}$, we have $\{({\mathbf v}+{\mathbf u}p)/p^{\ell}\}\in\mathcal{D}$ and therefore 
$$
v_p(g({\mathbf v}+{\mathbf u}p))=\sum_{\ell=1}^{\infty}{\mathbf 1}_{\mathcal{D}}\left(\left\{\frac{{\mathbf v}+{\mathbf u}p}{p^{\ell}}\right\}\right)\geq j_0-1.
$$ 
Following \eqref{intertheta}, we get
\begin{align}
v_p\Bigg((X_s({\mathbf v},{\mathbf u},&{\mathbf m})-1)\frac{\mathcal{Q}({\mathbf v}+{\mathbf u}p+{\mathbf m}p^{s+1})}{g({\mathbf m})}\Bigg)\notag\\
&\geq s-j_0+2+v_p(g({\mathbf v}+{\mathbf u}p))+\left(\sum_{\ell=1}^{s+1}\Delta\left(\left\{\frac{{\mathbf v}+{\mathbf u}p}{p^{\ell}}\right\}\right)-v_p(g({\mathbf v}+{\mathbf u}p))\right)\notag\\
&\geq (s-j_0+2)+j_0-1+v_p\left(\frac{\mathcal{Q}({\mathbf v}+{\mathbf u}p)}{g({\mathbf v}+{\mathbf u}p)}\right)\label{explipli45}\\
&\geq s+1+v_p\left(\frac{\mathcal{Q}({\mathbf v}+{\mathbf u}p)}{g({\mathbf v}+{\mathbf u}p)}\right),\notag
\end{align}
where \eqref{explipli45} is valid because, applying Lemma \ref{lemme Q(a) g(a)} with $s+1$ instead of $s$ and ${\mathbf v}+{\mathbf u}p$ instead of ${\mathbf a}$, we get $v_p(\mathcal{Q}({\mathbf v}+{\mathbf u}p))=\sum_{\ell=1}^{s+1}\Delta(\{\frac{{\mathbf v}+{\mathbf u}p}{p^{\ell}}\})$. Thus we have \eqref{eq 3.13} in this case.
\medskip
\begin{itemize}
\item \textit{Case} 2: Let us assume that, for all $j\in\{1,\dots,s+1\}$, we have $\{({\mathbf v}+{\mathbf u}p)/p^{j}\}\in\mathcal{D}$.
\end{itemize} 
\medskip

In particular, we have ${\mathbf v}+{\mathbf u}p\notin\Psi_{s+1}(\mathcal{N})$ and thus $\theta_s({\mathbf v}+{\mathbf u}p)=\mathcal{Q}({\mathbf v}+{\mathbf u}p)$. Furthermore, we obtain $\sum_{\ell=1}^{s+1}\Delta(\{({\mathbf v}+{\mathbf u}p)/p^{\ell}\})\geq s+1$. 

If $v_p(Y_s({\mathbf v},{\mathbf u},{\mathbf m}))\geq 0$, then, following Lemma \ref{définition de Y}, $v_p(X_s({\mathbf v},{\mathbf u},{\mathbf m})-1)\geq 0$ and, according to \eqref{cor valuation du quotient avec g}, we have
\begin{align*}
v_p\left(\frac{\mathcal{Q}({\mathbf v}+{\mathbf u}p+{\mathbf m}p^{s+1})}{g({\mathbf m})}\right)&\geq \sum_{\ell=1}^{s+1}\Delta\left(\left\{\frac{{\mathbf v}+{\mathbf u}p}{p^{\ell}}\right\}\right)\geq s+1.
\end{align*}
thus we have \eqref{eq 3.13}.

Let us now assume that $v_p(Y_s({\mathbf v},{\mathbf u},{\mathbf m}))<0$. In this case, according to Lemma \ref{définition de Y}, we have
$$
v_p(X_s({\mathbf v},{\mathbf u},{\mathbf m})-1)=v_p(Y_s({\mathbf v},{\mathbf u},{\mathbf m}))\geq\eta_{s}({\mathbf u},{\mathbf m})-\eta_{s+1}({\mathbf v}+{\mathbf u}p,{\mathbf m}).
$$
Furthermore, 
\begin{align*}
v_p(\mathcal{Q}({\mathbf v}+{\mathbf u}p+{\mathbf m}p^{s+1}))
&=\sum_{\ell=1}^{\infty}\Delta\left(\left\{\frac{{\mathbf v}+{\mathbf u}p+{\mathbf m}p^{s+1}}{p^{\ell}}\right\}\right)\\
&=\sum_{\ell=1}^{s+1}\Delta\left(\left\{\frac{{\mathbf v}+{\mathbf u}p}{p^{\ell}}\right\}\right)+\sum_{\ell=s+2}^{\infty}\Delta\left(\left\{\frac{{\mathbf v}+{\mathbf u}p+{\mathbf m}p^{s+1}}{p^{\ell}}\right\}\right)\\
&=\sum_{\ell=1}^{s+1}\Delta\left(\left\{\frac{{\mathbf v}+{\mathbf u}p}{p^{\ell}}\right\}\right)+\eta_{s+1}({\mathbf v}+{\mathbf u}p,{\mathbf m}).
\end{align*}
Thereby, we get 
\begin{align*}
v_p\Bigg((X_s(&{\mathbf v},{\mathbf u},{\mathbf m})-1)\frac{\mathcal{Q}({\mathbf v}+{\mathbf u}p+{\mathbf m}p^{s+1})}{g({\mathbf m})}\Bigg)\\
&\geq\eta_{s}({\mathbf u},{\mathbf m})-\eta_{s+1}({\mathbf v}+{\mathbf u}p,{\mathbf m})+\sum_{\ell=1}^{s+1}\Delta\left(\left\{\frac{{\mathbf v}+{\mathbf u}p}{p^{\ell}}\right\}\right)+\eta_{s+1}({\mathbf v}+{\mathbf u}p,{\mathbf m})-\mu({\mathbf m})\\
&\geq s+1+\eta_{s}({\mathbf u},{\mathbf m})-\mu({\mathbf m}).
\end{align*}

If $s=0$, then we have ${\mathbf u}={\mathbf 0}$ and $\eta_{0}({\mathbf 0},{\mathbf m})=\sum_{\ell=1}^{\infty}\Delta(\{\frac{{\mathbf m}}{p^{\ell}}\})\geq\sum_{\ell=1}^{\infty}{\mathbf 1}_{\mathcal{D}}(\{\frac{{\mathbf m}}{p^{\ell}}\})=\mu({\mathbf m})$ and we have \eqref{eq 3.13}. On the other hand, if $s\geq 1$ then, applying Lemma \ref{valuation du quotient avec g} with $s-1$ instead of $s$ and ${\mathbf a}={\mathbf u}$, we get $\eta_{s}({\mathbf u},{\mathbf m})\geq\mu({\mathbf m})$, which implies \eqref{eq 3.13}. This finishes the proof of equation \eqref{Hypothèse (iii)} modulo those of the various lemmas.

\subsubsection{Proof of Lemmas \ref{définition de Y}, \ref{lemme theta Q(a)}, \ref{valuation du quotient avec g} and \ref{lemme Q(a) g(a)}}\label{démo lemme Y}

\begin{proof}[Proof of Lemma \ref{définition de Y}]
We have to prove that
$X_s({\mathbf v},{\mathbf u},{\mathbf m})\in Y_s({\mathbf v},{\mathbf u},{\mathbf m})(1+p^{s+1}\mathbb{Z}_p)$.

We have 
$$
X_s({\mathbf v},{\mathbf u},{\mathbf m})=\frac{\mathcal{Q}({\mathbf v}+{\mathbf u}p)}{\mathcal{Q}({\mathbf u}p)}\frac{\mathcal{Q}({\mathbf u}p+{\mathbf m}p^{s+1})}{\mathcal{Q}({\mathbf v}+{\mathbf u}p+{\mathbf m}p^{s+1})}\cdot\frac{\mathcal{Q}({\mathbf u}p)}{\mathcal{Q}({\mathbf u})}\frac{\mathcal{Q}({\mathbf u}+{\mathbf m}p^{s})}{\mathcal{Q}({\mathbf u}p+{\mathbf m}p^{s+1})}.
$$
Applying Lemma \ref{lemme 16} with ${\mathbf c}={\mathbf u}$ we obtain
$$
\frac{\mathcal{Q}({\mathbf u}p)}{\mathcal{Q}({\mathbf u})}\frac{\mathcal{Q}({\mathbf u}+{\mathbf m}p^{s})}{\mathcal{Q}({\mathbf u}p+{\mathbf m}p^{s+1})}\in 1+p^{s+1}\mathbb{Z}_p,
$$
so that 
\begin{equation}\label{pour valuation X}
X_s({\mathbf v},{\mathbf u},{\mathbf m})\in\frac{\mathcal{Q}({\mathbf v}+{\mathbf u}p)}{\mathcal{Q}({\mathbf u}p)}\frac{\mathcal{Q}({\mathbf u}p+{\mathbf m}p^{s+1})}{\mathcal{Q}({\mathbf v}+{\mathbf u}p+{\mathbf m}p^{s+1})}(1+p^{s+1}\mathbb{Z}_p).
\end{equation}
Furthermore, we have
\begin{align*}
\frac{\mathcal{Q}({\mathbf v}+{\mathbf u}p)}{\mathcal{Q}({\mathbf u}p)}\cdot &\frac{\mathcal{Q}({\mathbf u}p+{\mathbf m}p^{s+1})}{\mathcal{Q}({\mathbf v}+{\mathbf u}p+{\mathbf m}p^{s+1})}\\
&=\frac{\Big(\prod_{i=1}^{q_1}\prod_{k=1}^{{\mathbf e}_i\cdot{\mathbf v}}({\mathbf e}_i\cdot{\mathbf u}p+k)\Big)\left(\prod_{i=1}^{q_2}\prod_{k=1}^{{\mathbf f}_i\cdot{\mathbf v}}({\mathbf f}_i\cdot({\mathbf u}p+{\mathbf m}p^{s+1})+k)\right)}{\left(\prod_{i=1}^{q_2}\prod_{k=1}^{{\mathbf f}_i\cdot{\mathbf v}}({\mathbf f}_i\cdot{\mathbf u}p+k)\right)\Big(\prod_{i=1}^{q_1}\prod_{k=1}^{{\mathbf e}_i\cdot{\mathbf v}}({\mathbf e}_i\cdot({\mathbf u}p+{\mathbf m}p^{s+1})+k)\Big)}\\
&=\frac{\prod_{i=1}^{q_2}\prod_{k=1}^{{\mathbf f}_i\cdot{\mathbf v}}\left(1+\frac{{\mathbf f}_i\cdot{\mathbf m}p^{s+1}}{{\mathbf f}_i\cdot{\mathbf u}p+k}\right)}{\prod_{i=1}^{q_1}\prod_{k=1}^{{\mathbf e}_i\cdot{\mathbf v}}\left(1+\frac{{\mathbf e}_i\cdot{\mathbf m}p^{s+1}}{{\mathbf e}_i\cdot{\mathbf u}p+k}\right)}.
\end{align*}

If ${\mathbf d}\in\{{\mathbf e}_1,\dots,{\mathbf e}_{q_1},{\mathbf f}_1,\dots,{\mathbf f}_{q_2}\}$ and $k\in\{1,\dots,{\mathbf d}\cdot{\mathbf v}\}$, then $p$ divides ${\mathbf d}\cdot{\mathbf u}p+k$ if and only if there exists $j\in\{1,\dots,\lfloor {\mathbf d}\cdot{\mathbf v}/p\rfloor\}$ such that $k=jp$. Thus we have  
$$
\prod_{k=1}^{{\mathbf d}\cdot{\mathbf v}}\left(1+\frac{{\mathbf d}\cdot{\mathbf m}p^{s+1}}{{\mathbf d}\cdot{\mathbf u}p+k}\right)=\prod_{j=1}^{\lfloor {\mathbf d}\cdot{\mathbf v}/p\rfloor}\left(1+\frac{{\mathbf d}\cdot{\mathbf m}p^{s}}{{\mathbf d}\cdot{\mathbf u}+j}\right)(1+O(p^{s+1})).
$$
Therefore
\begin{align*}
\frac{\mathcal{Q}({\mathbf v}+{\mathbf u}p)}{\mathcal{Q}({\mathbf u}p)}\cdot\frac{\mathcal{Q}({\mathbf u}p+{\mathbf m}p^{s+1})}{\mathcal{Q}({\mathbf v}+{\mathbf u}p+{\mathbf m}p^{s+1})}
&=\frac{\prod_{i=1}^{q_2}\prod_{j=1}^{\lfloor {\mathbf f}_i\cdot{\mathbf v}/p\rfloor}\left(1+\frac{{\mathbf f}_i\cdot{\mathbf m}p^{s}}{{\mathbf f}_i\cdot{\mathbf u}+j}\right)}{\prod_{i=1}^{q_1}\prod_{j=1}^{\lfloor {\mathbf e}_i\cdot{\mathbf v}/p\rfloor}\left(1+\frac{{\mathbf e}_i\cdot{\mathbf m}p^{s}}{{\mathbf e}_i\cdot{\mathbf u}+j}\right)}(1+O(p^{s+1}))\\
&=Y_s({\mathbf v},{\mathbf u},{\mathbf m})(1+O(p^{s+1}))
\end{align*}
and so $X_s({\mathbf v},{\mathbf u},{\mathbf m})\in Y_s({\mathbf v},{\mathbf u},{\mathbf m})(1+p^{s+1}\mathbb{Z}_p)$, as expected.
\medskip

We will now prove that we also have 
$$
v_p(Y_s({\mathbf v},{\mathbf u},{\mathbf m}))\geq\eta_{s}({\mathbf u},{\mathbf m})-\eta_{s+1}({\mathbf v}+{\mathbf u}p,{\mathbf m}).
$$ 
We have seen above that $v_p(Y_s({\mathbf v},{\mathbf u},{\mathbf m}))=v_p(X_s({\mathbf v},{\mathbf u},{\mathbf m}))$. Furthermore, according to \eqref{pour valuation X}, we also have 
\begin{align*}
v_p(X_s({\mathbf v},{\mathbf u},{\mathbf m}))
&=v_p\left(\frac{\mathcal{Q}({\mathbf v}+{\mathbf u}p)}{\mathcal{Q}({\mathbf u}p)}\cdot\frac{\mathcal{Q}({\mathbf u}p+{\mathbf m}p^{s+1})}{\mathcal{Q}({\mathbf v}+{\mathbf u}p+{\mathbf m}p^{s+1})}\right)\\
&=v_p(\mathcal{Q}({\mathbf v}+{\mathbf u}p))-v_p(\mathcal{Q}({\mathbf u}p))+v_p(\mathcal{Q}({\mathbf u}p+{\mathbf m}p^{s+1}))\\
&-v_p(\mathcal{Q}({\mathbf v}+{\mathbf u}p+{\mathbf m}p^{s+1}))\\
&=\sum_{\ell=1}^{\infty}\Delta\left(\left\{\frac{{\mathbf v}+{\mathbf u}p}{p^{\ell}}\right\}\right)-\sum_{\ell=1}^{\infty}\Delta\left(\left\{\frac{{\mathbf u}p}{p^{\ell}}\right\}\right)+\sum_{\ell=1}^{\infty}\Delta\left(\left\{\frac{{\mathbf u}p+{\mathbf m}p^{s+1})}{p^{\ell}}\right\}\right)\\
&-\sum_{\ell=1}^{\infty}\Delta\left(\left\{\frac{{\mathbf v}+{\mathbf u}p+{\mathbf m}p^{s+1}}{p^{\ell}}\right\}\right).
\end{align*}
We have
\begin{multline*}
\sum_{\ell=1}^{\infty}\Delta\left(\left\{\frac{{\mathbf v}+{\mathbf u}p}{p^{\ell}}\right\}\right)-\sum_{\ell=1}^{\infty}\Delta\left(\left\{\frac{{\mathbf v}+{\mathbf u}p+{\mathbf m}p^{s+1}}{p^{\ell}}\right\}\right)\\
=\sum_{\ell=1}^{\infty}\Delta\left(\left\{\frac{{\mathbf v}+{\mathbf u}p}{p^{\ell}}\right\}\right)-\sum_{\ell=1}^{s+1}\Delta\left(\left\{\frac{{\mathbf v}+{\mathbf u}p}{p^{\ell}}\right\}\right)-\sum_{\ell=s+2}^{\infty}\Delta\left(\left\{\frac{{\mathbf v}+{\mathbf u}p+{\mathbf m}p^{s+1}}{p^{\ell}}\right\}\right)\\
=\sum_{\ell=s+2}^{\infty}\Delta\left(\left\{\frac{{\mathbf v}+{\mathbf u}p}{p^{\ell}}\right\}\right)-\sum_{\ell=s+2}^{\infty}\Delta\left(\left\{\frac{{\mathbf v}+{\mathbf u}p+{\mathbf m}p^{s+1}}{p^{\ell}}\right\}\right)\\
=\eta_{s+1}({\mathbf v}+{\mathbf u}p,{\mathbf 0})-\eta_{s+1}({\mathbf v}+{\mathbf u}p,{\mathbf m}),
\end{multline*}
and
\begin{align*}
\sum_{\ell=1}^{\infty}
\Delta\left(\left\{\frac{{\mathbf u}p}{p^{\ell}}\right\}\right)-&\sum_{\ell=1}^{\infty}\Delta\left(\left\{\frac{{\mathbf u}p+{\mathbf m}p^{s+1}}{p^{\ell}}\right\}\right)\\
&=\sum_{\ell=s+2}^{\infty}\Delta\left(\left\{\frac{{\mathbf u}p}{p^{\ell}}\right\}\right)-\sum_{\ell=s+2}^{\infty}\Delta\left(\left\{\frac{{\mathbf u}p+{\mathbf m}p^{s+1}}{p^{\ell}}\right\}\right)\\
&=\sum_{\ell=s+1}^{\infty}\Delta\left(\left\{\frac{{\mathbf u}}{p^{\ell}}\right\}\right)-\sum_{\ell=s+1}^{\infty}\Delta\left(\left\{\frac{{\mathbf u}+{\mathbf m}p^{s}}{p^{\ell}}\right\}\right)=\eta_{s}({\mathbf u},{\mathbf 0})-\eta_{s}({\mathbf u},{\mathbf m}).
\end{align*}
Thus, $v_p(Y_s({\mathbf v},{\mathbf u},{\mathbf m}))=\eta_{s+1}({\mathbf v}+{\mathbf u}p,{\mathbf 0})-\eta_{s}({\mathbf u},{\mathbf 0})+\eta_{s}({\mathbf u},{\mathbf m})-\eta_{s+1}({\mathbf v}+{\mathbf u}p,{\mathbf m})$. We now have to prove that if ${\mathbf u}\in\Psi_s(\mathcal{N})$, then we have $\eta_{s+1}({\mathbf v}+{\mathbf u}p,{\mathbf 0})-\eta_s({\mathbf u},{\mathbf 0})\geq 0$. As ${\mathbf u}\in\Psi_s(\mathcal{N})$, we have $\{{\mathbf u}/p^s\}\notin\mathcal{D}$. Hence, for all $\ell\geq s+1$ and all ${\mathbf L}\in\{{\mathbf e}_1,\dots,{\mathbf e}_{q_1},{\mathbf f}_1,\dots,{\mathbf f}_{q_2}\}$, we obtain
$$
{\mathbf L}\cdot\left\{\frac{{\mathbf u}}{p^{\ell}}\right\}={\mathbf L}\cdot\frac{{\mathbf u}}{p^{\ell}}\leq{\mathbf L}\cdot\frac{{\mathbf u}}{p^{s}}={\mathbf L}\cdot\left\{\frac{{\mathbf u}}{p^{s}}\right\}<1,
$$
\textit{i.e.}, for all $\ell\geq s+1$, $\{{\mathbf u}/p^{\ell}\}\notin\mathcal{D}$. Then we have $\eta_s({\mathbf u},{\mathbf 0})=\sum_{\ell=s+1}^{\infty}\Delta\left(\left\{\frac{{\mathbf u}}{p^{\ell}}\right\}\right)=0$
and 
$$
\eta_{s+1}({\mathbf v}+{\mathbf u}p,{\mathbf 0})-\eta_s({\mathbf u},{\mathbf 0})=\eta_{s+1}({\mathbf v}+{\mathbf u}p,{\mathbf 0})\geq 0,
$$
which completes the proof of Lemma \ref{définition de Y}. 
\end{proof}

\begin{proof}[Proof of Lemma \ref{lemme theta Q(a)}]
Given $s\in\mathbb{N}$, ${\mathbf v}\in\{0,\dots,p-1\}^d$ and ${\mathbf u}\in\{0,\dots,p^s-1\}^d$, we write ${\mathbf u}=\sum_{k=0}^{\infty}{\mathbf u}_kp^k$, where ${\mathbf u}_k\in\{0,\dots,p-1\}^d$. Given ${\mathbf L}\in\{{\mathbf e}_1,\dots,{\mathbf e}_{q_1},{\mathbf f}_1,\dots,{\mathbf f}_{q_2}\}$, we define $s+1$ non-negative integers by the formulas $b_{{\mathbf L},0}:=\lfloor {\mathbf L}\cdot{\mathbf v}/p\rfloor$ and $b_{{\mathbf L},k+1}:=\lfloor({\mathbf L}\cdot{\mathbf u}_k+b_{{\mathbf L},k})/p\rfloor$ for $k\in\{0,\dots,s-1\}$. For all $x\in\mathbb{R}$, we write $\lceil x\rceil$ the smallest integer greater than $x$ and we define $s+1$ non-negative integers by the formulas $a_{{\mathbf L},0}:=1$ and $a_{{\mathbf L},k+1}:=\lceil({\mathbf L}\cdot{\mathbf u}_k+a_{{\mathbf L},k})/p\rceil$. First, we will prove by induction on $r$ that assertion $\mathcal{A}_{r}$: 
$$
\prod_{n=1}^{\lfloor {\mathbf L}\cdot{\mathbf v}/p\rfloor}\left(1+\frac{{\mathbf L}\cdot{\mathbf m}p^{s}}{{\mathbf L}\cdot{\mathbf u}+n}\right)=\prod_{n=a_{{\mathbf L},r}}^{b_{{\mathbf L},r}}\left(1+\frac{{\mathbf L}\cdot{\mathbf m}p^{s-r}}{{\mathbf L}\cdot(\sum_{k=r}^{\infty}{\mathbf u}_k p^{k-r})+n}\right)\left(1+O(p^{s-r+1})\right)
$$
is true for all $r\in\{0,\dots,s\}$.

We have $b_{{\mathbf L},0}=\lfloor {\mathbf L}\cdot{\mathbf v}/p\rfloor$ and $a_{{\mathbf L},0}=1$, thus $\mathcal{A}_{0}$ is true.

Given $r\geq 0$, let us assume that $\mathcal{A}_{r}$ is true and prove $\mathcal{A}_{r+1}$. If $a_{\mathbf{L},r}>b_{\mathbf{L},r}$ then $a_{\mathbf{L},r+1}>b_{\mathbf{L},r+1}$ and $\mathcal{A}_r$ implies $\mathcal{A}_{r+1}$. Thus we can assume that $a_{\mathbf{L},r}\leq b_{\mathbf{L},r}$. If $n\in\{a_{{\mathbf L},r},\dots,b_{{\mathbf L},r}\}$, then $p$ divides ${\mathbf L}\cdot\left(\sum_{k=r}^{\infty}{\mathbf u}_kp^{k-r}\right)+n$ if and only if $p$ divides ${\mathbf L}\cdot{\mathbf u}_r+n$, \textit{i.e.} if and only if an $i\in\{\lceil({\mathbf L}\cdot{\mathbf u}_r+a_{{\mathbf L},r})/p\rceil,\dots,\lfloor({\mathbf L}\cdot{\mathbf u}_r+b_{{\mathbf L},r})/p\rfloor\}$ exists such that ${\mathbf L}\cdot{\mathbf u}_r+n=ip$. So we get
\begin{multline}\label{Alr}
\prod_{n=a_{{\mathbf L},r}}^{b_{{\mathbf L},r}}\left(1+\frac{{\mathbf L}\cdot{\mathbf m}p^{s-r}}{{\mathbf L}\cdot\left(\sum_{k=r}^{\infty}{\mathbf u}_kp^{k-r}\right)+n}\right)\\
=\prod_{i=a_{{\mathbf L},r+1}}^{b_{{\mathbf L},r+1}}\left(1+\frac{{\mathbf L}\cdot{\mathbf m}p^{s-r}}{{\mathbf L}\cdot\left(\sum_{k=r+1}^{\infty}{\mathbf u}_kp^{k-r}\right)+ip}\right)(1+O(p^{s-r}))\\
=\prod_{i=a_{{\mathbf L},r+1}}^{b_{{\mathbf L},r+1}}\left(1+\frac{{\mathbf L}\cdot{\mathbf m}p^{s-r-1}}{{\mathbf L}\cdot\left(\sum_{k=r+1}^{\infty}{\mathbf u}_kp^{k-r-1}\right)+i}\right)(1+O(p^{s-r})).
\end{multline}
According to $\mathcal{A}_{r}$ and \eqref{Alr}, we have $\mathcal{A}_{r+1}$, which finishes the induction on $r$.

Given ${\mathbf L}\in\{{\mathbf e}_1,\dots,{\mathbf e}_{q_1},{\mathbf f}_1,\dots,{\mathbf f}_{q_2}\}$, we will prove by induction on $k$ that assertion $\mathcal{B}_{k}$ : $a_{{\mathbf L},k}\geq 1$ and $b_{{\mathbf L},k}\leq \lfloor {\mathbf L}\cdot\{({\mathbf v}+{\mathbf u}p)/p^{k+1}\}\rfloor$ is true for all $k\in\{0,\dots,s\}$. 

We have $a_{{\mathbf L},0}=1$ and $b_{{\mathbf L},0}=\lfloor {\mathbf L}\cdot{\mathbf v}/p\rfloor=\lfloor {\mathbf L}\cdot\{({\mathbf v}+{\mathbf u}p)/p\}\rfloor$, so $\mathcal{B}_0$ is true.

Given $k\geq 0$, let us assume that $\mathcal{B}_k$ is true and let us prove $\mathcal{B}_{k+1}$. We have $a_{{\mathbf L},k+1}=\lceil({\mathbf L}\cdot{\mathbf u}_k+a_{{\mathbf L},k})/p\rceil$ and $b_{{\mathbf L},k+1}=\left\lfloor({\mathbf L}\cdot{\mathbf u}_k+b_{{\mathbf L},k})/p\right\rfloor$, thus $a_{{\mathbf L},k+1}\geq\lceil({\mathbf L}\cdot{\mathbf u}_k+1)/p\rceil\geq 1$ and
\begin{align*}
b_{{\mathbf L},k+1}\leq\left\lfloor\frac{{\mathbf L}\cdot{\mathbf u}_k}{p}+\frac{{\mathbf L}}{p}\cdot\left\{\frac{{\mathbf v}+{\mathbf u}p}{p^{k+1}}\right\}\right\rfloor&=\left\lfloor {\mathbf L}\cdot\left(\frac{{\mathbf u}_kp^{k+1}}{p^{k+2}}+\frac{{\mathbf v}+p\sum_{i=0}^{k-1}{\mathbf u}_ip^i}{p^{k+2}}\right)\right\rfloor\\
&=\left\lfloor {\mathbf L}\cdot\left\{\frac{{\mathbf v}+{\mathbf u}p}{p^{k+2}}\right\}\right\rfloor,
\end{align*}
which completes the induction on $k$.

Given $j\in\{1,\dots,s+1\}$ such that $\{({\mathbf v}+{\mathbf u}p)/p^j\}\notin\mathcal{D}$, for all ${\mathbf L}\in\{{\mathbf e}_1,\dots,{\mathbf e}_{q_1},{\mathbf f}_1,\dots,{\mathbf f}_{q_2}\}$, we obtain, \textit{via} $\mathcal{B}_{j-1}$, that $a_{{\mathbf L},j-1}\geq 1$ and $b_{{\mathbf L},j-1}\leq\lfloor {\mathbf L}\cdot\{({\mathbf v}+{\mathbf u}p)/p^{j}\}\rfloor=0$. Hence, following $\mathcal{A}_{j-1}$, we get 
$$
\prod_{n=1}^{\lfloor {\mathbf L}\cdot{\mathbf v}/p\rfloor}\left(1+\frac{{\mathbf L}\cdot{\mathbf m}p^{s}}{{\mathbf L}\cdot{\mathbf u}+n}\right)=1+O(p^{s-j+2})
$$
and thus
$$
Y_s({\mathbf v},{\mathbf u},{\mathbf m})=\frac{\prod_{i=1}^{q_2}\prod_{n=1}^{\lfloor {\mathbf f}_i\cdot{\mathbf v}/p\rfloor}\left(1+\frac{{\mathbf f}_i\cdot{\mathbf m}p^{s}}{{\mathbf f}_i\cdot{\mathbf u}+n}\right)}{\prod_{i=1}^{q_1}\prod_{n=1}^{\lfloor {\mathbf e}_i\cdot{\mathbf v}/p\rfloor}\left(1+\frac{{\mathbf e}_i\cdot{\mathbf m}p^{s}}{{\mathbf e}_i\cdot{\mathbf u}+n}\right)}=\frac{1+O(p^{s-j+2})}{1+O(p^{s-j+2})}=1+O(p^{s-j+2}), 
$$
which finishes the proof of Lemma \ref{lemme theta Q(a)}.
\end{proof}

\begin{proof}[Proof of Lemma \ref{valuation du quotient avec g}]
First, we will prove that we have \eqref{eq du lemme 20}. Let us write ${\mathbf m}=\sum_{k=0}^q{\mathbf m}_kp^k$, where ${\mathbf m}_k\in\{0,\dots,p-1\}^d$. We have
\begin{align*}
\eta_{s+1}({\mathbf a},{\mathbf m})-\mu({\mathbf m})
&=\sum_{\ell=s+2}^{\infty}\Delta\left(\left\{\frac{{\mathbf a}+{\mathbf m}p^{s+1}}{p^{\ell}}\right\}\right)-\sum_{\ell=1}^{\infty}\textsc{1}_{\mathcal{D}}\left(\left\{\frac{{\mathbf m}}{p^{\ell}}\right\}\right)\\
&=\sum_{\ell=s+2}^{\infty}\left(\Delta\left(\left\{\frac{{\mathbf a}+{\mathbf m}p^{s+1}}{p^{\ell}}\right\}\right)-\textsc{1}_{\mathcal{D}}\left(\left\{\frac{{\mathbf m}p^{s+1}}{p^{\ell}}\right\}\right)\right)\\
&=\sum_{\ell=s+2}^{\infty}\left(\Delta\left(\frac{{\mathbf a}+\sum_{k=0}^{\ell-s-2}{\mathbf m}_kp^{k+s+1}}{p^{\ell}}\right)-\textsc{1}_{\mathcal{D}}\left(\frac{\sum_{k=0}^{\ell-s-2}{\mathbf m}_kp^{k+s+1}}{p^{\ell}}\right)\right).
\end{align*}
Furthermore, for all $\ell\geq s+2$, we have
$$
{\mathbf 0}\leq \frac{\sum_{k=0}^{\ell-s-2}{\mathbf m}_kp^{k+s+1}}{p^{\ell}}\leq \frac{{\mathbf a}+\sum_{k=0}^{\ell-s-2}{\mathbf m}_kp^{k+s+1}}{p^{\ell}}\leq\frac{(p^{\ell}-1){\mathbf 1}}{p^{\ell}}\in[0,1[^d.
$$
Thus
\begin{align*}
\textsc{1}_{\mathcal{D}}\left(\frac{\sum_{k=0}^{\ell-s-2}{\mathbf m}_kp^{k+s+1}}{p^{\ell}}\right)=1\quad
&\Longrightarrow\quad\frac{\sum_{k=0}^{\ell-s-2}{\mathbf m}_kp^{k+s+1}}{p^{\ell}}\in\mathcal{D}\\
&\Longrightarrow\quad\frac{{\mathbf a}+\sum_{k=0}^{\ell-s-2}{\mathbf m}_kp^{k+s+1}}{p^{\ell}}\in\mathcal{D}\\
&\Longrightarrow\quad \Delta\left(\frac{{\mathbf a}+\sum_{k=0}^{\ell-s-2}{\mathbf m}_kp^{k+s+1}}{p^{\ell}}\right)\geq 1
\end{align*}
and so $\eta_{s+1}({\mathbf a},{\mathbf m})-\mu({\mathbf m})\geq 0$. This completes the proof of \eqref{eq du lemme 20}.

Let us now prove \eqref{cor valuation du quotient avec g}. We have 
\begin{align}
v_p\left(\frac{\mathcal{Q}({\mathbf a}+{\mathbf m}p^{s+1})}{g_p({\mathbf m})}\right)
&=\sum_{\ell=1}^{\infty}\Delta\left(\left\{\frac{{\mathbf a}+{\mathbf m}p^{s+1}}{p^{\ell}}\right\}\right)-\mu({\mathbf m})\notag\\
&=\sum_{\ell=1}^{s+1}\Delta\left(\left\{\frac{{\mathbf a}}{p^{\ell}}\right\}\right)+\sum_{\ell=s+2}^{\infty}\Delta\left(\left\{\frac{{\mathbf a}+{\mathbf m}p^{s+1}}{p^{\ell}}\right\}\right)-\mu({\mathbf m})\notag\\
&=\sum_{\ell=1}^{s+1}\Delta\left(\left\{\frac{{\mathbf a}}{p^{\ell}}\right\}\right)+\eta_{s+1}({\mathbf a},{\mathbf m})-\mu({\mathbf m}),\notag\\
&\geq \sum_{\ell=1}^{s+1}\Delta\left(\left\{\frac{{\mathbf a}}{p^{\ell}}\right\}\right)\label{expli expli 2}.
\end{align}
where we used inequality \eqref{eq du lemme 20} for \eqref{expli expli 2}.
\end{proof}

\begin{proof}[Proof of Lemma \ref{lemme Q(a) g(a)}]
We have $ v_p(Q({\mathbf a}))=\sum_{\ell=1}^{\infty}\Delta\left(\left\{\frac{{\mathbf a}}{p^{\ell}}\right\}\right)$. As ${\mathbf a}\in\Psi_s(\mathcal{N})$, we have $\{{\mathbf a}/p^s\}\notin\mathcal{D}$ and, for all $\ell\geq s+1$ and all ${\mathbf L}\in\{{\mathbf e}_1,\dots,{\mathbf e}_{q_1},{\mathbf f}_1,\dots,{\mathbf f}_{q_2}\}$, we get 
$$
{\mathbf L}\cdot\left\{\frac{{\mathbf a}}{p^{\ell}}\right\}={\mathbf L}\cdot\frac{{\mathbf a}}{p^{\ell}}\leq{\mathbf L}\cdot\frac{{\mathbf a}}{p^{s}}={\mathbf L}\cdot\left\{\frac{{\mathbf a}}{p^s}\right\}<1,
$$
\textit{i.e.} $\{{\mathbf a}/p^{\ell}\}\notin\mathcal{D}$. Thus, for all $\ell\geq s+1$, we have $\Delta\left(\left\{{\mathbf a}/p^{\ell}\right\}\right)=0$. This gives us the expected result.
\end{proof}

\section{Proof of assertions $(ii)$ of Theorems \ref{critère} and \ref{critère2}}\label{Proof(ii)768}

We assume the hypothesis of Theorems \ref{critère} and \ref{critère2}. Furthermore, we assume that ${\mathbf x}_0\in\mathcal{D}_{e,f}$ is a zero of $\Delta_{e,f}$. In Section \ref{Préliminaires}, we prove an elementary result of analysis which we will use for the proofs of assertions $(ii)$ of Theorems \ref{critère} and \ref{critère2}. We prove assertion $(ii)$ of Theorem \ref{critère} in Section \ref{reci1}. We will use certain results from Section \ref{reci1} for the proof of assertion $(ii)$ of Theorem \ref{critère2} which we present in Section \ref{reci2}.

\subsection{Preliminary}\label{Préliminaires}

The aim of this section is to prove that there exists a nonempty open subset $\mathcal{U}$ of $\mathcal{D}_{e,f}$ such that, for all ${\mathbf x}\in\mathcal{U}$, $i\in\{1,\dots,q_1\}$ and $j\in\{1,\dots,q_2\}$, we have $\lfloor{\mathbf e}_i\cdot{\mathbf x}\rfloor=\lfloor{\mathbf e}_i\cdot{\mathbf x}_0\rfloor$, ${\mathbf e}_i\cdot{\mathbf x}\neq 0$, $\lfloor{\mathbf f}_j\cdot{\mathbf x}\rfloor=\lfloor{\mathbf f}_j\cdot{\mathbf x}_0\rfloor$, ${\mathbf f}_j\cdot{\mathbf x}\neq 0$ and ${\mathbf e}_i\cdot{\mathbf x}\neq{\mathbf f}_j\cdot{\mathbf x}$. 

Particularly, for all ${\mathbf x}\in\mathcal{U}$, we would have $\Delta_{e,f}({\mathbf x})=\Delta_{e,f}({\mathbf x}_0)=0$. We will use this open set $\mathcal{U}$ throughout the rest of the proof.
\medskip

Applying Lemma \ref{localconst} with, instead of $u$, the sequence constituted by the elements of $e$ and $f$, we obtain that there exists $\mu>0$ such that, for all ${\mathbf x}\in[0,\mu]^d$ and all ${\mathbf L}\in\{{\mathbf e}_1,\dots,{\mathbf e}_{q_1},{\mathbf f}_1,\dots{\mathbf f}_{q_2}\}$, we have $\lfloor{\mathbf L}\cdot({\mathbf x}_0+{\mathbf x})\rfloor=\lfloor{\mathbf L}\cdot{\mathbf x}_0\rfloor$. As ${\mathbf x}_0\in[0,1[^d$, there exists $\mu_1>0$, $\mu_1\leq\mu$, such that, for all ${\mathbf x}\in[0,\mu_1]^d$, we have ${\mathbf x}_0+{\mathbf x}\in[0,1[^d$. Since ${\mathbf x}_0\in\mathcal{D}_{e,f}$, a ${\mathbf L}\in\{{\mathbf e}_1,\dots,{\mathbf e}_{q_1},{\mathbf f}_1,\dots,{\mathbf f}_{q_2}\}$ exists such that ${\mathbf L}\cdot{\mathbf x}_0\geq 1$, which gives us the result that, for all ${\mathbf x}\in[0,\mu_1]^d$, we have ${\mathbf L}\cdot({\mathbf x}_0+{\mathbf x})\geq{\mathbf L}\cdot{\mathbf x}_0\geq 1$ and thus, as ${\mathbf x}_0+{\mathbf x}\in[0,1[^d$, we get that ${\mathbf x}_0+{\mathbf x}\in\mathcal{D}_{e,f}$. Thereby, there exists a nonempty open subset $\mathcal{U}_1$ of $\mathcal{D}_{e,f}$ such that, for all ${\mathbf x}\in\mathcal{U}_1$ and ${\mathbf L}\in\{{\mathbf e}_1,\dots,{\mathbf e}_{q_1},{\mathbf f}_1,\dots,{\mathbf f}_{q_2}\}$, we have $\lfloor{\mathbf L}\cdot{\mathbf x}\rfloor=\lfloor{\mathbf L}\cdot{\mathbf x}_0\rfloor$.
\medskip

For all $i\in\{1,\dots,q_1\}$ and $j\in\{1,\dots,q_2\}$, we define the sets $\mathcal{H}_{{\mathbf e}_i}:=\{{\mathbf x}\in\mathbb{R}^d\,:\,{\mathbf e}_i\cdot{\mathbf x}=0\}$, $\mathcal{H}_{{\mathbf f}_j}:=\{{\mathbf x}\in\mathbb{R}^d\,:\,{\mathbf f}_j\cdot{\mathbf x}=0\}$ and $\mathcal{H}_{{\mathbf e}_i,{\mathbf f}_j}:=\{{\mathbf x}\in\mathbb{R}^d\,:\,{\mathbf e}_i\cdot{\mathbf x}={\mathbf f}_j\cdot{\mathbf x}\}$. Since $e$ and $f$ are two disjoint sequences constituted by nonzero vectors, we obtain that the $\mathcal{H}_{{\mathbf e}_i}$, $\mathcal{H}_{{\mathbf f}_j}$ and $\mathcal{H}_{{\mathbf e}_i,{\mathbf f}_j}$ are hyperplanes in $\mathbb{R}^d$ and are therefore closed subsets of $\mathbb{R}^d$ with empty interiors. Therefore, their complements are dense open subsets of $\mathbb{R}^d$ and the complement $\mathcal{U}_2$ of the union of $\mathcal{H}_{{\mathbf e}_i}$, $\mathcal{H}_{{\mathbf f}_j}$ and $\mathcal{H}_{{\mathbf e}_i,{\mathbf f}_j}$ is a dense open subset of $\mathbb{R}^d$. As a result, $\mathcal{U}:=\mathcal{U}_1\cap\mathcal{U}_2$ is a nonempty open subset of $\mathcal{D}_{e,f}$ and, for all ${\mathbf x}\in\mathcal{U}$, $i\in\{1,\dots,q_1\}$ and $j\in\{1,\dots,q_2\}$, we have ${\mathbf e}_i\cdot{\mathbf x}\neq 0$, ${\mathbf f}_j\cdot{\mathbf x}\neq 0$, ${\mathbf e}_i\cdot{\mathbf x}\neq{\mathbf f}_j\cdot{\mathbf x}$, $\lfloor{\mathbf e}_i\cdot{\mathbf x}\rfloor=\lfloor{\mathbf e}_i\cdot{\mathbf x}_0\rfloor$ and $\lfloor{\mathbf f}_j\cdot{\mathbf x}\rfloor=\lfloor{\mathbf f}_j\cdot{\mathbf x}_0\rfloor$.

\subsection{Proof of assertion $(ii)$ of Theorem \ref{critère}}\label{reci1}

The aim of this section is to prove that there exists $k\in\{1,\dots,d\}$ such that there are only finitely many prime numbers $p$ such that $q_{e,f,k}({\mathbf z})\in z_k\mathbb{Z}_p[[{\mathbf z}]]$. Following Section \ref{equiv crit}, we only have to prove that there exists $k\in\{1,\dots,d\}$ such that, for all large enough prime number $p$, there exists ${\mathbf a}\in\{0,\dots,p-1\}^d$ and ${\mathbf K}\in\mathbb{N}^d$ such that $\Phi_{p,k}({\mathbf a}+p{\mathbf K})\notin p\mathbb{Z}_p$.
We will actually prove that there exists $k\in\{1,\dots,d\}$ such that, for all large enough prime number $p$, there is an ${\mathbf a}\in\{0,\dots,p-1\}^d$ such that $\Phi_{p,k}({\mathbf a})\notin p\mathbb{Z}_p$. In this case, we have
\begin{equation}\label{Phi2a}
\Phi_{p,k}({\mathbf a})=-p\mathcal{Q}({\mathbf a})\left(\sum_{i=1}^{q_1}{\mathbf e}_i^{(k)}H_{{\mathbf a}\cdot{\mathbf e}_i}-\sum_{i=1}^{q_2}{\mathbf f}_i^{(k)}H_{{\mathbf a}\cdot{\mathbf f}_i}\right).
\end{equation}
For all ${\mathbf d}\in\mathbb{N}^d$, we have
\begin{align*}
pH_{{\mathbf d}\cdot{\mathbf a}}=p\sum_{i=1}^{{\mathbf d}\cdot{\mathbf a}}\frac{1}{i}
&\equiv p\sum_{j=1}^{\lfloor{\mathbf d}\cdot{\mathbf a}/p\rfloor}\frac{1}{jp}\mod p\mathbb{Z}_p\\ &\equiv\sum_{j=1}^{\lfloor{\mathbf d}\cdot{\mathbf a}/p\rfloor}\frac{1}{j}\mod p\mathbb{Z}_p.
\end{align*} 
For all $k\in\{1,\dots,d\}$ and ${\mathbf x}\in[0,1]^d$, we set
$$
\Psi_k({\mathbf x}):=\sum_{i=1}^{q_1}\sum_{j=1}^{\lfloor{\mathbf e}_i\cdot{\mathbf x}\rfloor}\frac{{\mathbf e}_i^{(k)}}{j}-\sum_{i=1}^{q_2}\sum_{j=1}^{\lfloor{\mathbf f}_i\cdot{\mathbf x}\rfloor}\frac{{\mathbf f}_i^{(k)}}{j}.
$$

Thus, for all $k\in\{1,\dots,d\}$ and ${\mathbf a}\in\{0,\dots,p-1\}^d$, we have $\Phi_{p,k}({\mathbf a})\equiv -\mathcal{Q}({\mathbf a})\Psi_k({\mathbf a}/p)\mod p\mathbb{Z}_p$. Therefore we now have to prove that there exists $k\in\{1,\dots,d\}$ such that, for all large enough prime number $p$, there exists ${\mathbf a}\in\{0,\dots,p-1\}^d$ such that $v_p(\mathcal{Q}({\mathbf a}))=v_p(\Psi_k({\mathbf a}/p))=0$. We set $\mathcal{M}:=\max\{|{\mathbf d}|\,:\,{\mathbf d}\in\{{\mathbf e}_1,\dots,{\mathbf e}_{q_1},{\mathbf f}_1,\dots,{\mathbf f}_{q_2}\}\}$.

A constant $\mathcal{P}_1\geq\mathcal{M}$ exists such that, for all prime number $p\geq\mathcal{P}_1$, there exists ${\mathbf a}_p\in\{0,\dots,p-1\}^d$ such that ${\mathbf a}_p/p\in\mathcal{U}$. For all $\ell\geq 2$, we have ${\mathbf a}_p/p^{\ell}\leq{\mathbf a}_p/p^2<{\mathbf 1}/p$ and thus, for all ${\mathbf L}\in\{{\mathbf e}_1,\dots,{\mathbf e}_{q_1},{\mathbf f}_1,\dots,{\mathbf f}_{q_2}\}$, we have ${\mathbf L}\cdot{\mathbf a}_p/p^{\ell}<{\mathbf L}\cdot{\mathbf 1}/p\leq\mathcal{M}/p\leq 1$. Hence, for all prime number $p\geq\mathcal{P}_1$ and all $\ell\geq 2$, we have ${\mathbf a}_p/p^{\ell}\notin\mathcal{D}_{e,f}$, which implies that
$$
v_p(\mathcal{Q}({\mathbf a}_p))=\sum_{\ell=1}^{\infty}\Delta_{e,f}\left(\frac{{\mathbf a}_p}{p^{\ell}}\right)=\Delta_{e,f}\left(\frac{{\mathbf a}_p}{p}\right)=0,
$$
because $\Delta_{e,f}$ vanishes on $\mathcal{U}$ and on $[0,1[^d\setminus\mathcal{D}_{e,f}$.

So we now have to prove that there exists $k\in\{1,\dots,d\}$ and a constant $\mathcal{P}\geq\mathcal{P}_1$ such that, for all prime number $p\geq\mathcal{P}$, we have $v_p(\Psi_k({\mathbf a}_p/p))=0$.

For all prime number $p\geq\mathcal{P}_1$, all $i\in\{1,\dots,q_1\}$ and $j\in\{1,\dots,q_2\}$, we write $\alpha_i:=\lfloor{\mathbf e}_i\cdot{\mathbf a}_p/p\rfloor$ and $\beta_j:=\lfloor{\mathbf f}_j\cdot{\mathbf a}_p/p\rfloor$. According to the construction of $\mathcal{U}$ and since ${\mathbf a}_p/p\in\mathcal{U}$, we have $\lfloor{\mathbf e}_i\cdot{\mathbf a}_p/p\rfloor=\lfloor{\mathbf e}_i\cdot{\mathbf x}_0\rfloor$ and $\lfloor{\mathbf f}_j\cdot{\mathbf a}_p/p\rfloor=\lfloor{\mathbf f}_j\cdot{\mathbf x}_0\rfloor$. Therefore, the $\alpha_i$ and $\beta_j$ do not depend on $p$. Thus there exists a constant $\mathcal{P}\geq\mathcal{P}_1$ such that, for all prime number $p\geq\mathcal{P}$ and all $k\in\{1,\dots,d\}$, we have 
$$
\Psi_{k}({\mathbf a}_p/p)=\sum_{i=1}^{q_1}\sum_{j=1}^{\alpha_i}\frac{{\mathbf e}_i^{(k)}}{j}-\sum_{i=1}^{q_2}\sum_{j=1}^{\beta_j}\frac{{\mathbf f}_i^{(k)}}{j}\in\mathbb{Z}_p^{\times}\cup\{0\}.
$$

Therefore we only have to prove that there exists $k\in\{1,\dots,d\}$ such that $\Psi_k({\mathbf a}_p/p)\neq 0$. For this purpose, we will use Lemma $16$ from \cite{Delaygue} which reads as follows.

\begin{lemme}\label{lemme 16 D}
Let ${\mathbf E}:=(E_1,\dots,E_{q_1})$ and ${\mathbf F}:=(F_1,\dots,F_{q_2})$ be two disjoint sequences of positive integers. We write $\mathcal{A}:=\{E_1,\dots,E_{q_1},F_1,\dots,F_{q_2}\}$ and $\gamma_1<\cdots<\gamma_t=1$ for the rational numbers which satisfy $\{\gamma_1,\dots,\gamma_t\}=\bigcup_{a\in\mathcal{A}}\{\frac{1}{a},\frac{2}{a},\dots,\frac{a-1}{a},1\}$ and $m_i\in\mathbb{Z}$ the amplitude of the jump of $\Delta_{{\mathbf E},{\mathbf F}}$ in $\gamma_i$. If there exists $i_0\in\{1,\dots,t\}$ such that $\Delta_{{\mathbf E},{\mathbf F}}\geq 0$ on $[\gamma_1,\gamma_{i_0}]$, then we have 
$$
\sum_{k=1}^{i_0}\frac{m_k}{\gamma_k}>0\quad\textup{and}\quad\prod_{k=1}^{i_0}\left(1+\frac{1}{\gamma_k}\right)^{m_k}>1.
$$
\end{lemme}

We will use Lemma \ref{lemme 16 D} with ${\mathbf E}_p:=({\mathbf e}_1\cdot{\mathbf a}_p,\dots,{\mathbf e}_{q_1}\cdot{\mathbf a}_p)$ instead of ${\mathbf E}$ and ${\mathbf F}_p:=({\mathbf f}_1\cdot{\mathbf a}_p,\dots,{\mathbf f}_{q_2}\cdot{\mathbf a}_p)$ instead of ${\mathbf F}$. 

First, we have to prove that ${\mathbf E}_p$ and ${\mathbf F}_p$ are two disjoint sequences of positive integers. Indeed, according to the construction of $\mathcal{U}$, for all $i\in\{1,\dots,q_1\}$ and all $j\in\{1,\dots,q_2\}$, we have ${\mathbf e}_i\cdot{\mathbf a}_p/p\neq 0$, ${\mathbf f}_j\cdot{\mathbf a}_p/p\neq 0$ and ${\mathbf e}_i\cdot{\mathbf a}_p/p\neq{\mathbf f}_j\cdot{\mathbf a}_p/p$, thus ${\mathbf e}_i\cdot{\mathbf a}_p\neq 0$, ${\mathbf f}_j\cdot{\mathbf a}_p\neq 0$ and ${\mathbf e}_i\cdot{\mathbf a}_p\neq{\mathbf f}_j\cdot{\mathbf a}_p$, which gives us that ${\mathbf E}_p$ and ${\mathbf F}_p$ are two disjoint sequences of positive integers. 

We write $\mathcal{A}:=\{{\mathbf e}_1\cdot{\mathbf a}_p,\dots,{\mathbf e}_{q_1}\cdot{\mathbf a}_p,{\mathbf f}_1\cdot{\mathbf a}_p,\dots,{\mathbf f}_{q_2}\cdot{\mathbf a}_p\}$ and $\gamma_1<\cdots<\gamma_t=1$ for the rational numbers which satisfy $\{\gamma_1,\dots,\gamma_t\}=\bigcup_{a\in\mathcal{A}}\{\frac{1}{a},\frac{2}{a},\dots,\frac{a-1}{a},1\}$ and $m_i\in\mathbb{Z}$ the amplitude of the jump of $\Delta_{{\mathbf E}_p,{\mathbf F}_p}$ in $\gamma_i$. As ${\mathbf a}_p/p\in\mathcal{D}_{{\mathbf e},{\mathbf f}}$, there exists $a\in\mathcal{A}$ such that $a\geq p$ and so $\max(\mathcal{A})\geq p$. Hence, we have $\gamma_1=1/\max(\mathcal{A})\leq 1/p$. Thus there exists $i_0\in\{1,\dots,t-1\}$ such that $\gamma_{i_0}\leq 1/p<\gamma_{i_0+1}$. Furthermore, for all $x\in[0,1]$, we have
$$
\Delta_{{\mathbf E}_p,{\mathbf F}_p}(x)=\sum_{i=1}^{q_1}\lfloor({\mathbf e}_i\cdot{\mathbf a}_p)x\rfloor-\sum_{j=1}^{q_2}\lfloor({\mathbf f}_j\cdot{\mathbf a}_p)x\rfloor=\Delta_{e,f}(x{\mathbf a}_p)\geq 0,
$$
because $\Delta_{{\mathbf e},{\mathbf f}}\geq 0$ on $[0,1]^d$. In particular, $\Delta_{{\mathbf E}_p,{\mathbf F}_p}\geq 0$ on $[\gamma_1,\gamma_{i_0}]$.

We can therefore apply Lemma \ref{lemme 16 D} which results in
\begin{align}
0<\sum_{i=1}^{i_0}\frac{m_i}{\gamma_i}
&=\sum_{c\in{\mathbf E}_p}\sum_{j=1}^{\lfloor c/p\rfloor}\frac{c}{j}-\sum_{d\in{\mathbf F}_p}\sum_{j=1}^{\lfloor d/p\rfloor}\frac{d}{j}\label{thuy}\\
&=\sum_{i=1}^{q_1}\sum_{j=1}^{\lfloor{\mathbf a}_p\cdot{\mathbf e}_i/p\rfloor}\frac{{\mathbf a}_p\cdot{\mathbf e}_i}{j}-\sum_{i=1}^{q_2}\sum_{j=1}^{\lfloor{\mathbf a}_p\cdot{\mathbf f}_i/p\rfloor}\frac{{\mathbf a}_p\cdot{\mathbf f}_i}{j}\notag\\
&=\sum_{k=1}^d{\mathbf a}_p^{(k)}\left(\sum_{i=1}^{q_1}\sum_{j=1}^{\alpha_i}\frac{{\mathbf e}_i^{(k)}}{j}-\sum_{i=1}^{q_2}\sum_{j=1}^{\beta_i}\frac{{\mathbf f}_i^{(k)}}{j}\right)=\sum_{k=1}^d\mathbf{a}_p^{(k)}\Psi_k({\mathbf a}_p/p),\notag
\end{align}
where \eqref{thuy} is valid because the abscissas of the jumps of $\Delta_{{\mathbf E}_p,{\mathbf F}_p}$ on $[0,1/p]$ are exactly the rational numbers $j/a$ with $a\in\mathcal{A}$ and $j\leq\lfloor a/p\rfloor$, and an abscissa $j/a$ corresponds to a jump with positive amplitude when $a\in{\mathbf E}_p$ and to a jump with negative amplitude when $a\in{\mathbf F}_p$.

Thus there exists $k\in\{1,\dots,d\}$ such that $\Psi_k({\mathbf a}_p/p)\neq 0$, which finishes the proof of assertion $(ii)$ of Theorem \ref{critère}.

\subsection{Proof of assertion $(ii)$ of Theorem \ref{critère2}}\label{reci2}

According to Section \ref{reci1}, there exists $k_0\in\{1,\dots,d\}$ such that there are only finitely many prime numbers $p$ such that $q_{e,f,k_0}(\mathbf{z})\in\mathbb{Z}_p[[\mathbf{z}]]$. In order to finish the proof of assertion $(ii)$ of Theorem~\ref{critère2}, we only have to prove that, for all $\mathbf{L}\in\mathcal{E}$ satisfying $\mathbf{L}^{(k_0)}\geq 1$, there are only finitely many prime numbers $p$ such that $q_{\mathbf{L},e,f}(\mathbf{z})\in\mathbb{Z}_p[[\mathbf{z}]]$. During the proof, we fix $\mathbf{L}\in\mathcal{E}_{e,f}$ satisfying $\mathbf{L}^{(k_0)}\geq 1$ (\footnote{Such a $\mathbf{L}$ exists because $q_{e,f,k_0}(\mathbf{z})\notin z_{k_0}\mathbb{Z}[[\mathbf{z}]]$.}). We will separate the proof into two cases depending on whether $\lfloor\mathbf{L}\cdot\mathbf{x}_0\rfloor=0$ or $\lfloor\mathbf{L}\cdot\mathbf{x}_0\rfloor\neq 0$.

According to Section \ref{reci1}, we know that there exists a constant $\mathcal{P}_1$ such that, for all prime number $p\geq\mathcal{P}_1$, there exists ${\mathbf a}_p\in\{0,\dots,p-1\}^d$ such that ${\mathbf a}_p/p\in\mathcal{U}$ and $v_p(\mathcal{Q}({\mathbf a}_p))=0$.

\subsubsection{When $\lfloor{\mathbf L}\cdot{\mathbf x}_0\rfloor\neq 0$}\label{premier temps}

The aim of this section is to prove that there exists a constant $\mathcal{P}\geq\mathcal{P}_1$ such that, for all prime number $p\geq\mathcal{P}$, we have $\Phi_{{\mathbf L},p}({\mathbf a}_p)\notin p\mathbb{Z}_p$, which, according to Section \ref{equiv crit}, will prove that there are only finitely many prime numbers $p$ such that $q_{{\mathbf L},e,f}({\mathbf z})\in\mathbb{Z}_p[[{\mathbf z}]]$.

We recall that, for all ${\mathbf a}\in\{0,\dots,p-1\}^d$, we have 
\begin{equation}\label{PhiaII}
\Phi_{{\mathbf L},p}({\mathbf a})=-p\mathcal{Q}({\mathbf a})H_{{\mathbf L}\cdot{\mathbf a}}\equiv -\mathcal{Q}({\mathbf a})H_{\lfloor{\mathbf L}\cdot{\mathbf a}/p\rfloor}\mod p\mathbb{Z}_p.
\end{equation}

For all prime number $p\geq\mathcal{P}_1$, we have $\lfloor{\mathbf L}\cdot{\mathbf a}_p/p\rfloor=\lfloor{\mathbf L}\cdot{\mathbf x}_0\rfloor\neq 0$ therefore $H_{\lfloor{\mathbf L}\cdot{\mathbf a}_p/p\rfloor}\in\{H_1,\dots,H_{|{\mathbf L}|}\}$. A constant $\mathcal{P}\geq\mathcal{P}_1$ exists such that, for all prime number $p\geq\mathcal{P}$, we have $\{H_1,\dots,H_{|{\mathbf L}|}\}\subset\mathbb{Z}_p^{\times}$. Thus, for all prime number $p\geq\mathcal{P}$, we have $\mathcal{Q}({\mathbf a}_p)H_{\lfloor{\mathbf L}\cdot{\mathbf a}_p/p\rfloor}\in\mathbb{Z}_p^{\times}$ and, following \eqref{PhiaII}, we obtain $\Phi_{{\mathbf L},p}({\mathbf a}_p)\notin p\mathbb{Z}_p$. 

We observe that in this case, we did not use the hypothesis $\mathbf{L}^{(k_0)}\geq 1$.

\subsubsection{When $\lfloor{\mathbf L}\cdot{\mathbf x}_0\rfloor=0$}\label{deuxième temps}

The aim of this section is to prove that there exists $r\in\mathbb{N}$, $r\geq 1$ and a constant $\mathcal{P}'\geq\mathcal{P}_1$ such that, for all prime number $p\geq\mathcal{P}'$, we have $\Phi_{{\mathbf L},p}({\mathbf a}_p+pr{\mathbf 1}_{k_0})\notin p\mathbb{Z}_p$. According to Section \ref{equiv crit}, this will prove that there are only finitely many prime numbers $p$ such that $q_{{\mathbf L},e,f}({\mathbf z})\in\mathbb{Z}_p[[{\mathbf z}]]$. 

In the sequel, for all $k\in\{1,\dots,d\}$, we write $R_k$ for the rational function defined by
\begin{equation}\label{DéfiR_k(X)}
R_k(X):=\frac{\prod_{i=1}^{q_1}\prod_{j=1}^{\alpha_i}\left(1+\frac{{\mathbf e}_i^{(k)}}{j}X\right)}{\prod_{i=1}^{q_2}\prod_{j=1}^{\beta_i}\left(1+\frac{{\mathbf f}_i^{(k)}}{j}X\right)}.
\end{equation}
We will use the following lemma, which we will prove at the end of this section.

\begin{lemme}\label{Re : Phi}
For all $r\in\mathbb{N}$, $r\geq 1$, there exists a constant $\mathcal{P}_r\geq\mathcal{P}_1$ such that, for all prime number $p\geq\mathcal{P}_r$ and all $k\in\{1,\dots,d\}$, we have
$$
\Phi_{{\mathbf L},p}({\mathbf a}_p+pr{\mathbf 1}_k)\equiv-\sum_{j=1}^rH_{j{\mathbf L}^{(k)}}\mathcal{Q}({\mathbf a}_p)\mathcal{Q}(j{\mathbf 1}_k)\mathcal{Q}((r-j){\mathbf 1}_k)\big(R_k(j)-R_k(r-j)\big)\mod p\mathbb{Z}_p.
$$
\end{lemme}

According to the end of Section \ref{reci1}, we know that 
\begin{equation}\label{chemical}
\sum_{i=1}^{q_1}\sum_{j=1}^{\alpha_i}\frac{{\mathbf e}_i^{(k_0)}}{j}-\sum_{i=1}^{q_2}\sum_{j=1}^{\beta_i}\frac{{\mathbf f}_i^{(k_0)}}{j}\neq 0.
\end{equation}

Inequality \eqref{chemical} proves that $R_{k_0}(X)$ is not a constant equal to $1$. Thus there exists $r\in\mathbb{N}$ such that $R_{k_0}(r)\neq 1$. Let $r_0$ be the smallest positive integer satisfying $R_{k_0}(r_0)\neq 1$. Applying Lemma \ref{Re : Phi} with $k_0$ instead of $k$ and $r_0$ instead of $r$, we obtain that a constant $\mathcal{P}_{r_0}\geq\mathcal{P}_1$ exists such that, for all prime number $p\geq\mathcal{P}_{r_0}$, we have
\begin{align}
\Phi_{{\mathbf L},p}({\mathbf a}_p+&pr_0{\mathbf 1}_{k_0})\notag\\
&\equiv-\sum_{j=1}^{r_0}H_{j{\mathbf L}^{(k_0)}}\mathcal{Q}({\mathbf a}_p)\mathcal{Q}(j{\mathbf 1}_{k_0})\mathcal{Q}((r_0-j){\mathbf 1}_{k_0})\big(R_{k_0}(j)-R_{k_0}(r_0-j)\big)\mod p\mathbb{Z}_p\notag\\
&\equiv-H_{r_0{\mathbf L}^{(k_0)}}\mathcal{Q}({\mathbf a}_p)\mathcal{Q}(r_0{\mathbf 1}_{k_0})\big(R_{k_0}(r_0)-1\big)\mod p\mathbb{Z}_p,\label{explikle}
\end{align}
where \eqref{explikle} is valid because, for all $j\in\{1,\dots,r_0-1\}$, we have $R_{k_0}(j)=R_{k_0}(r_0-j)=1$. Since $R_{k_0}(r_0)\neq 1$, we obtain that if ${\mathbf L}^{(k_0)}\geq 1$, then there exists a constant $\mathcal{P}\geq\mathcal{P}_{r_0}$ such that, for all prime number $p\geq\mathcal{P}$, we have 
$$
H_{r_0{\mathbf L}^{(k_0)}}\mathcal{Q}({\mathbf a}_p)\mathcal{Q}(r_0{\mathbf 1}_{k_0})\left(R_{k_0}(r_0)-1\right)\in\mathbb{Z}_p^{\times}
$$ 
and therefore $\Phi_{{\mathbf L},p}({\mathbf a}_p+pr_0{\mathbf 1}_{k_0})\notin p\mathbb{Z}_p$, which completes the proof of assertion $(ii)$ of Theorem~\ref{critère2} modulo the proof of Lemma \ref{Re : Phi}.

\begin{proof}[Proof of Lemma \ref{Re : Phi}]
According to Section \ref{equiv crit}, for all prime number $p\geq\mathcal{P}_1$ and all ${\mathbf K}\in\mathbb{N}^d$, we have 
\begin{equation}\label{PhiaK}
\Phi_{{\mathbf L},p}({\mathbf a}_p+p{\mathbf K})=\sum_{{\mathbf 0}\leq{\mathbf j}\leq{\mathbf K}}\mathcal{Q}({\mathbf K}-{\mathbf j})\mathcal{Q}({\mathbf a}_p+p{\mathbf j})\left(H_{{\mathbf L}\cdot({\mathbf K}-{\mathbf j})}-pH_{{\mathbf L}\cdot({\mathbf a}_p+p{\mathbf j})}\right).
\end{equation}
Furthermore, we have $pH_{{\mathbf L}\cdot({\mathbf a}_p+p{\mathbf j})}\equiv H_{\left\lfloor\frac{{\mathbf L}\cdot{\mathbf a}_p+p{\mathbf L}\cdot{\mathbf j}}{p}\right\rfloor}\mod p\mathbb{Z}_p$ with $\left\lfloor\frac{{\mathbf L}\cdot{\mathbf a}_p+p{\mathbf L}\cdot{\mathbf j}}{p}\right\rfloor=\lfloor{\mathbf L}\cdot{\mathbf a}_p/p\rfloor+{\mathbf L}\cdot{\mathbf j}={\mathbf L}\cdot{\mathbf j}$ because $\lfloor{\mathbf L}\cdot{\mathbf a}_p/p\rfloor=\lfloor{\mathbf L}\cdot{\mathbf x}_0\rfloor=0$. Thereby, for all ${\mathbf K},\mathbf{j}\in\mathbb{N}^d$, ${\mathbf j}\leq{\mathbf K}$, we obtain
\begin{equation}\label{transf765} \mathcal{Q}({\mathbf K}-{\mathbf j})\mathcal{Q}({\mathbf a}_p+p{\mathbf j})pH_{{\mathbf L}\cdot({\mathbf a}_p+p{\mathbf j})}\equiv\mathcal{Q}({\mathbf K}-{\mathbf j})\mathcal{Q}({\mathbf a}_p+p{\mathbf j})H_{{\mathbf L}\cdot{\mathbf j}}\mod p\mathbb{Z}_p.
\end{equation}
Applying \eqref{transf765} to \eqref{PhiaK}, we obtain that, for all ${\mathbf K}\in\mathbb{N}^d$, we have
\begin{align}
\Phi_{{\mathbf L},p}(&{\mathbf a}_p+p{\mathbf K})\notag\\
&\equiv\sum_{{\mathbf 0}\leq{\mathbf j}\leq{\mathbf K}}\mathcal{Q}({\mathbf K}-{\mathbf j})\mathcal{Q}({\mathbf a}_p+p{\mathbf j})\left(H_{{\mathbf L}\cdot({\mathbf K}-{\mathbf j})}-H_{{\mathbf L}\cdot{\mathbf j}}\right)\mod p\mathbb{Z}_p\notag\\
&\equiv-\sum_{{\mathbf 0}\leq{\mathbf j}\leq{\mathbf K}}H_{{\mathbf L}\cdot{\mathbf j}}\big(\mathcal{Q}({\mathbf a}_p+p{\mathbf j})\mathcal{Q}({\mathbf K}-{\mathbf j})-\mathcal{Q}({\mathbf j})\mathcal{Q}({\mathbf a}_p+p({\mathbf K}-{\mathbf j}))\big)\mod p\mathbb{Z}_p\notag\\
&\equiv-\sum_{{\mathbf 0}\leq{\mathbf j}\leq{\mathbf K}}H_{{\mathbf L}\cdot{\mathbf j}}\mathcal{Q}({\mathbf a}_p)\mathcal{Q}({\mathbf j})\mathcal{Q}({\mathbf K}-{\mathbf j})\left(\frac{\mathcal{Q}({\mathbf a}_p+p{\mathbf j})}{\mathcal{Q}({\mathbf a}_p)\mathcal{Q}({\mathbf j})}-\frac{\mathcal{Q}({\mathbf a}_p+p({\mathbf K}-{\mathbf j}))}{\mathcal{Q}({\mathbf a}_p)\mathcal{Q}({\mathbf K}-{\mathbf j})}\right)\mod p\mathbb{Z}_p\label{explik362}.
\end{align}
Applying \eqref{explik362} with $r{\mathbf 1}_k$ instead of ${\mathbf K}$, we finally obtain
\begin{multline}\label{expliqua}
\Phi_{{\mathbf L},p}({\mathbf a}_p+pr{\mathbf 1}_k)\equiv\\
-\sum_{j=0}^rH_{j{\mathbf L}^{(k)}}\mathcal{Q}({\mathbf a}_p)\mathcal{Q}(j{\mathbf 1}_k)\mathcal{Q}((r-j){\mathbf 1}_k)\left(\frac{\mathcal{Q}({\mathbf a}_p+pj{\mathbf 1}_k)}{\mathcal{Q}({\mathbf a}_p)\mathcal{Q}(j{\mathbf 1}_k)}-\frac{\mathcal{Q}({\mathbf a}_p+p(r-j){\mathbf 1}_k)}{\mathcal{Q}({\mathbf a}_p)\mathcal{Q}((r-j){\mathbf 1}_k)}\right)\mod p\mathbb{Z}_p.
\end{multline}
We will now prove that, for all $n\in\mathbb{N}$ and $k\in\{1,\dots,d\}$, we have
\begin{equation}\label{expliqua2}
\frac{\mathcal{Q}({\mathbf a}_p+pn{\mathbf 1}_k)}{\mathcal{Q}({\mathbf a}_p)\mathcal{Q}(n{\mathbf 1}_k)}=R_k(n)(1+O(p)),
\end{equation}
which will enable us to conclude. We have
\begin{align}
\frac{\mathcal{Q}({\mathbf a}_p+pn{\mathbf 1}_k)}{\mathcal{Q}({\mathbf a}_p)\mathcal{Q}(n{\mathbf 1}_k)}
&=\frac{\mathcal{Q}({\mathbf a}_p+pn{\mathbf 1}_k)}{\mathcal{Q}({\mathbf a}_p)\mathcal{Q}(pn{\mathbf 1}_k)}\frac{\mathcal{Q}(pn{\mathbf 1}_k)}{\mathcal{Q}(n{\mathbf 1}_k)}\notag\\
&=\frac{1}{\mathcal{Q}({\mathbf a}_p)}\frac{\prod_{i=1}^{q_1}\prod_{j=1}^{{\mathbf e}_i\cdot{\mathbf a}_p}(pn{\mathbf e}_i^{(k)}+j)}{\prod_{i=1}^{q_2}\prod_{j=1}^{{\mathbf f}_i\cdot{\mathbf a}_p}(pn{\mathbf f}_i^{(k)}+j)}\big(1+O(p)\big)\label{explikk1}\\
&=\frac{\prod_{i=1}^{q_1}\prod_{j=1}^{{\mathbf e}_i\cdot{\mathbf a}_p}\left(1+\frac{pn{\mathbf e}_i^{(k)}}{j}\right)}{\prod_{i=1}^{q_2}\prod_{j=1}^{{\mathbf f}_i\cdot{\mathbf a}_p}\left(1+\frac{pn{\mathbf f}_i^{(k)}}{j}\right)}\big(1+O(p)\big)\notag\\
&=\frac{\prod_{i=1}^{q_1}\prod_{j=1}^{\lfloor{\mathbf e}_i\cdot{\mathbf a}_p/p\rfloor}\left(1+\frac{{\mathbf e}_i^{(k)}}{j}n\right)}{\prod_{i=1}^{q_2}\prod_{j=1}^{\lfloor{\mathbf f}_i\cdot{\mathbf a}_p/p\rfloor}\left(1+\frac{{\mathbf f}_i^{(k)}}{j}n\right)}\big(1+O(p)\big)\label{explikk2}\\
&=R_k(n)\big(1+O(p)\big),\notag
\end{align} 
where we obtain \eqref{explikk1} by applying Lemma \ref{lemme 16} with $s=0$, ${\mathbf c}={\mathbf 0}$ and $n{\mathbf 1}_k$ instead of ${\mathbf m}$, which leads to $\mathcal{Q}(pn{\mathbf 1}_k)/\mathcal{Q}(n{\mathbf 1}_k)=1+O(p)$. Equation \eqref{explikk2} is valid because, for all ${\mathbf d}\in\{{\mathbf e}_1,\dots,{\mathbf e}_{q_1},{\mathbf f}_1,\dots,{\mathbf f}_{q_2}\}$ and $j\in\{1,\dots,{\mathbf d}\cdot{\mathbf a}_p\}$, if $j$ is not divisible by  $p$ then we have $1+\frac{pn{\mathbf d}^{(k)}}{j}=1+O(p)$. 

There exists a constant $\mathcal{P}_r\geq\mathcal{P}_1$ such that, for all prime number $p\geq\mathcal{P}_r$ and all $n\in\{0,\dots,r\}$, we have $R_k(n)\in\mathbb{Z}_p^{\times}$ and $H_{n{\mathbf L}^{(k)}}\in\mathbb{Z}_p$. Therefore, applying \eqref{expliqua2} to \eqref{expliqua}, we obtain that, for all prime number $p\geq\mathcal{P}_r$, we have
$$
\Phi_{{\mathbf L},p}({\mathbf a}_p+pr{\mathbf 1}_k)\equiv-\sum_{j=1}^rH_{j{\mathbf L}^{(k)}}\mathcal{Q}({\mathbf a}_p)\mathcal{Q}(j{\mathbf 1}_k)\mathcal{Q}((r-j){\mathbf 1}_k)\left(R_k(j)-R_k(r-j)\right)\mod p\mathbb{Z}_p,
$$
which finishes the proof of Lemma \ref{Re : Phi}.
\end{proof}

\section{Proof of Theorem \ref{cas e>f}}\label{démo e>f}

We assume the hypothesis of Theorem \ref{cas e>f}. The aim of this section is to prove that there are only finitely many prime numbers $p$ such that $q_{e,f,k}({\mathbf z})\in z_k\mathbb{Z}_p[[{\mathbf z}]]$ and that, for all ${\mathbf L}\in\mathcal{E}_{e,f}$ satisfying ${\mathbf L}^{(k)}\geq 1$, there are only finitely many prime numbers $p$ such that $q_{{\mathbf L},e,f}({\mathbf z})\in\mathbb{Z}_p[[{\mathbf z}]]$. We fix a ${\mathbf L}\in\mathcal{E}_{e,f}$ satisfying ${\mathbf L}^{(k)}\geq 1$ throughout this section. 
 
According to Section \ref{equiv crit}, we only have to prove that, for all large enough prime number $p$, there exists ${\mathbf a}\in\{0,\dots,p-1\}^d$ and ${\mathbf K}\in\mathbb{N}^d$ such that $\Phi_{p,k}({\mathbf a}+{\mathbf K}p)\notin p\mathbb{Z}_p$ and $\Phi_{{\mathbf L},p}({\mathbf a}+{\mathbf K}p)\notin p\mathbb{Z}_p$. In fact, we will prove that, for all large enough prime number $p$, we have $\Phi_{p,k}(p{\mathbf 1}_k)\notin p\mathbb{Z}_p$ and $\Phi_{{\mathbf L},p}(p{\mathbf 1}_k)\notin p\mathbb{Z}_p$. We have  
\begin{align}
&\Phi_{p,k}(p{\mathbf 1}_k)\notag\\
&=\sum_{j=0}^1\mathcal{Q}((1-j){\mathbf 1}_k)\mathcal{Q}(jp{\mathbf 1}_k)\left(\sum_{i=1}^{q_1}{\mathbf e}_i^{(k)}(H_{{\mathbf e}_i^{(k)}(1-j)}-pH_{{\mathbf e}_i^{(k)}jp})-\sum_{i=1}^{q_2}{\mathbf f}_i^{(k)}(H_{{\mathbf f}_i^{(k)}(1-j)}-pH_{{\mathbf f}_i^{(k)}jp})\right)\notag\\
&=\mathcal{Q}({\mathbf 1}_k)\left(\sum_{i=1}^{q_1}{\mathbf e}_i^{(k)}H_{{\mathbf e}_i^{(k)}}-\sum_{i=1}^{q_2}{\mathbf f}_i^{(k)}H_{{\mathbf f}_i^{(k)}}\right)-p\mathcal{Q}(p{\mathbf 1}_k)\left(\sum_{i=1}^{q_1}{\mathbf e}_i^{(k)}H_{{\mathbf e}_i^{(k)}p}-\sum_{i=1}^{q_2}{\mathbf f}_i^{(k)}H_{{\mathbf f}_i^{(k)}p}\right)\label{nkj1}
\end{align}
and
\begin{align}
\Phi_{{\mathbf L},p}(p{\mathbf 1}_k)&=\sum_{j=0}^1\mathcal{Q}((1-j){\mathbf 1}_k)\mathcal{Q}(jp{\mathbf 1}_k)(H_{{\mathbf L}^{(k)}(1-j)}-pH_{{\mathbf L}^{(k)}jp})\notag\\
&=\mathcal{Q}({\mathbf 1}_k)H_{{\mathbf L}^{(k)}}-p\mathcal{Q}(p{\mathbf 1}_k)H_{{\mathbf L}^{(k)}p}\label{nkj}.
\end{align}

There exists a constant $\mathcal{P}_1$ such that, for all prime number $p\geq\mathcal{P}_1$, we have 
$$
\sum_{i=1}^{q_1}{\mathbf e}_i^{(k)}H_{{\mathbf e}_i^{(k)}}-\sum_{i=1}^{q_2}{\mathbf f}_i^{(k)}H_{{\mathbf f}_i^{(k)}}\in\mathbb{Z}_p^{\times}\cup\{0\}
$$ 
and $H_{{\mathbf L}^{(k)}}\in\mathbb{Z}_p^{\times}$ because ${\mathbf L}^{(k)}\geq 1$. In the sequel, we write $\Delta_k$ for the Landau's function associated with sequences $e^{(k)}:=({\mathbf e}_1^{(k)},\dots,{\mathbf e}_{q_1}^{(k)})$ and $f^{(k)}:=({\mathbf f}_1^{(k)},\dots,{\mathbf f}_{q_2}^{(k)})$. We also write $M$ for the largest element of sequences $e^{(k)}$ and $f^{(k)}$. We note that $M$ is nonzero because $|e|^{(k)}>|f|^{(k)}$, and that $\Delta_k$ vanishes on $[0,1/M[$. If $p>M$, then, for all $\ell\geq 1$, we have $1/p^{\ell}<1/M$ and thus $v_p(\mathcal{Q}({\mathbf 1}_k))=\sum_{\ell=1}^{\infty}\Delta_{e,f}({\mathbf 1}_k/p^{\ell})=\sum_{\ell=1}^{\infty}\Delta_k(1/p^{\ell})=0$. Hence, for all prime number $p>\max(\mathcal{P}_1,M)=:\mathcal{P}_2$, we have 
\begin{equation}\label{valu1}
\mathcal{Q}({\mathbf 1}_k)\left(\sum_{i=1}^{q_1}{\mathbf e}_i^{(k)}H_{{\mathbf e}_i^{(k)}}-\sum_{i=1}^{q_2}{\mathbf f}_i^{(k)}H_{{\mathbf f}_i^{(k)}}\right)\in\mathbb{Z}_p^{\times}\cup\{0\}\quad\textup{and}\quad\mathcal{Q}({\mathbf 1}_k)H_{{\mathbf L}^{(k)}}\in\mathbb{Z}_p^{\times}.
\end{equation}
Furthermore, we have 
$$
pH_{{\mathbf L}^{(k)}p}=p\left(\sum_{i=1}^{{\mathbf L}^{(k)}}\frac{1}{ip}+\sum_{\underset{p\nmid j}{j=1}}^{{\mathbf L}^{(k)}p}\frac{1}{j}\right)\equiv H_{{\mathbf L}^{(k)}}\mod p\mathbb{Z}_p,
$$ 
which gives us that, for all prime number $p>\mathcal{P}_2$, we have $pH_{{\mathbf L}^{(k)}p}\in\mathbb{Z}_p^{\times}$. Similarly, we get 
$$
p\left(\sum_{i=1}^{q_1}{\mathbf e}_i^{(k)}H_{{\mathbf e}_i^{(k)}p}-\sum_{i=1}^{q_2}{\mathbf f}_i^{(k)}H_{{\mathbf f}_i^{(k)}p}\right)\in\mathbb{Z}_p.
$$ 
Finally, for all prime number $p>\mathcal{P}_2$, we have 
$$
v_p(\mathcal{Q}(p{\mathbf 1}_k))=\sum_{\ell=1}^{\infty}\Delta_{e,f}\left(\frac{p{\mathbf 1}_k}{p^{\ell}}\right)=\sum_{\ell=1}^{\infty}\Delta_k\left(\frac{p}{p^{\ell}}\right)=\Delta_k(1)+\sum_{\ell=1}^{\infty}\Delta_k\left(\frac{1}{p^{\ell}}\right)=|{\mathbf e}|^{(k)}-|{\mathbf f}\,|^{(k)}\geq 1,
$$
from which we obtain that, for all prime number $p>\mathcal{P}_2$, we have 
\begin{equation}\label{valu2}
p\mathcal{Q}(p{\mathbf 1}_k)\left(\sum_{i=1}^{q_1}{\mathbf e}_i^{(k)}H_{{\mathbf e}_i^{(k)}p}-\sum_{i=1}^{q_2}{\mathbf f}_i^{(k)}H_{{\mathbf f}_i^{(k)}p}\right)\in p\mathbb{Z}_p\quad\textup{and}\quad p\mathcal{Q}(p{\mathbf 1}_k)H_{{\mathbf L}^{(k)}p}\in p\mathbb{Z}_p.
\end{equation}

Applying \eqref{valu1} and \eqref{valu2} to \eqref{nkj}, we obtain that, for all prime number $p>\mathcal{P}_2$, we have $\Phi_{{\mathbf L},p}(p{\mathbf 1}_k)\notin p\mathbb{Z}_p$. 

Congruences \eqref{valu1} and \eqref{valu2} associated with \eqref{nkj1} prove that it suffices to prove that $\sum_{i=1}^{q_1}{\mathbf e}_i^{(k)}H_{{\mathbf e}_i^{(k)}}-\sum_{i=1}^{q_2}{\mathbf f}_i^{(k)}H_{{\mathbf f}_i^{(k)}}\neq 0$ to conclude that, for all prime number $p>\mathcal{P}_2$, we have $\Phi_{p,k}(p{\mathbf 1}_k)\notin p\mathbb{Z}_p$. 

For this purpose, we write ${\mathbf E}$ and ${\mathbf F}$ the respective subsequences of $e^{(k)}$ and $f^{(k)}$ obtained as follows. We remove the zero elements of $e^{(k)}$ and $f^{(k)}$ and, if $e^{(k)}$ and $f^{(k)}$ have an element in common, then we remove it from $e^{(k)}$ and $f^{(k)}$ once only. This latest step is repeated until the obtained sequences are disjoint. The sequence ${\mathbf F}$ can be empty but the hypothesis $|e|^{(k)}>|f|^{(k)}$ ensures that the sequence ${\mathbf E}$ is nonempty. Thus we have
\begin{equation}\label{dapres2}
\sum_{i=1}^{q_1}{\mathbf e}_i^{(k)}H_{{\mathbf e}_i^{(k)}}-\sum_{i=1}^{q_2}{\mathbf f}_i^{(k)}H_{{\mathbf f}_i^{(k)}}=\sum_{c\in{\mathbf E}}cH_c-\sum_{d\in{\mathbf F}}dH_d\quad\textup{and}\quad\Delta_k=\Delta_{{\mathbf E},{\mathbf F}}.
\end{equation}
Particularly, if ${\mathbf F}$ is empty then we have 
$$
\sum_{i=1}^{q_1}{\mathbf e}_i^{(k)}H_{{\mathbf e}_i^{(k)}}-\sum_{i=1}^{q_2}{\mathbf f}_i^{(k)}H_{{\mathbf f}_i^{(k)}}=\sum_{c\in{\mathbf E}}cH_c>0. 
$$
In the sequel of the proof, we assume that ${\mathbf F}$ is nonempty. 

Since ${\mathbf E}$ and ${\mathbf F}$ are two disjoint sequences of positive integers, we can apply Lemma \ref{lemme 16 D} to the sequences ${\mathbf E}$ and ${\mathbf F}$. Using the notations of Lemma \ref{lemme 16 D}, we obtain
\begin{equation}\label{dapres3}
\sum_{c\in{\mathbf E}}cH_c-\sum_{d\in{\mathbf F}}dH_d=\sum_{c\in{\mathbf E}}\sum_{j=1}^c\frac{c}{j}-\sum_{d\in{\mathbf F}}\sum_{j=1}^d\frac{d}{j}=\sum_{i=1}^t\frac{m_i}{\gamma_i}.
\end{equation}
Furthermore, for all $x\in\mathbb{R}$, we have $\Delta_{{\mathbf E},{\mathbf F}}(x)=\Delta_k(x)=\Delta_{e,f}(x{\mathbf 1}_k)\geq 0$ so $\Delta_{{\mathbf E},{\mathbf F}}\geq 0$ on $[\gamma_1,\gamma_t]$ and Lemma \ref{lemme 16 D} leads to $\sum_{i=1}^t\frac{m_i}{\gamma_i}>0$. This inequality associated with \eqref{dapres2} and \eqref{dapres3} proves that $\sum_{i=1}^{q_1}{\mathbf e}_i^{(k)}H_{{\mathbf e}_i^{(k)}}-\sum_{i=1}^{q_2}{\mathbf f}_i^{(k)}H_{{\mathbf f}_i^{(k)}}\neq 0$ and completes the proof of Theorem~\ref{cas e>f}.

\section{A consequence of Theorems \ref{critère} and \ref{critère2}}\label{consequenceZudi}

Almkvist, van Enckevort, van Straten and Zudilin present in \cite{Tables} a list of more than $400$ fourth order differential equations that they call of Calabi--Yau type. In most of the considered equations, they give an explicit formula for the analytic solution $F(z)$ normalized by the condition $F(0)=1$. One of the required conditions so that an equation to be of Calabi--Yau type is that its indicial equation at $z=0$ should have $0$ as its only solution (see \cite{Tables}). In particular, according to Section $4.3$ of \cite{Batyrev}, there is a unique power series without constant term $G(z)\in\mathbb{C}[[z]]$ such that $G(z)+\log(z)F(z)$ is a solution of the differential equation linearly independent of $F(z)$. We can then define the $q$-parameter for the equation (following \cite{Batyrev}) $q(z):=z\exp(G(z)/F(z))$.

In \cite{Tanguy 2}, Krattenthaler and Rivoal observe that $43$ equations from the list \cite{Tables} have for solution $F$ a specialization of a series $F_{e,f}(\mathbf{z})$, where the sequences $e$ and $f$ verify the conditions of Theorem $2$ from \cite{Tanguy 2}. We understand by specialization of $F_{e,f}(\mathbf{z})$ any series obtained by replacing each $z_i$, $1\leq i\leq d$, by $z_i=M_iz^{N_i}$, where $M_i\in\mathbb{Z}\setminus\{0\}$ and $N_i\in\mathbb{N}$, $N_i\geq 1$. According to Theorem $2$ from \cite{Tanguy 2}, we see that the Taylor coefficients of canonical coordinates and mirror-type maps associated with $e$ and $f$ are all integers. It is the same for their specializations, which ensures the integrality of the Taylor coefficients of numerous new univariate mirror-type maps. In particular, we can obtain the integrality of the $q$-parameter for the differential equation. 

We found $100$ additional equations from the list \cite{Tables} which have as solution $F(z)$ a specialization of a series $F_{e,f}(\mathbf{z})$. Among these new cases, $97$ correspond to sequences $e$ and $f$ such that $\Delta_{e,f}\geq 1$ on $\mathcal{D}_{e,f}$ and thus, according to Theorems \ref{critère} and \ref{critère2}, such that the specializations of canonical coordinates and mirror-type maps lie in $z\mathbb{Z}[[z]]$. On the other hand, Cases $84$, $284$ and $338$ correspond to sequences $e$ and $f$ such that there exists $\mathbf{x}\in\mathcal{D}_{e,f}$ such that $\Delta_{e,f}(\mathbf{x})=0$. Thus we know that at least one of the canonical coordinates does not lie in $z\mathbb{Z}[[\mathbf{z}]]$.

All in all, we obtain $143$ equations which are the cases: $1$--$25$, $29$, $3^*$, $4^{**}$, $10^{**}$, $13^{**}$, $\hat{1}$--$\hat{14}$, $30$, $34$--$40$, $43$--$53$, $55$, $56$, $58$--$60$, $62$--$91$, $93$--$99$, $110$--$112$, $116$, $119$, $125$--$128$, $130$, $149$, $180$, $185$, $188$, $190$--$192$, $208$, $209$, $212$, $229$, $232$, $233$, $237$--$241$, $278$, $284$, $288$, $292$, $307$, $330$, $337$, $338$, $340$ and $367$.  

Let us, for example, give the details for Case $30$. The differential operator is
\begin{equation}\label{opérateur}
\mathcal{L}:=\theta^4-2^4z(4\theta+1)(4\theta+3)(8\theta^2+8\theta+3)+2^{12}z^2(4\theta+1)(4\theta+3)(4\theta+5)(4\theta+7),
\end{equation}
where $\theta=z\frac{d}{dz}$. The function $F$ canceled by this operator is
$$
F(z)=\sum_{n=0}^{\infty}z^n\frac{(4n)!}{(n!)^2(2n)!}\sum_{k=0}^n2^{2k}\binom{2(n-k)}{n-k}^2\binom{2k}{k}.
$$
Given $e=((4,4),(2,0),(2,0),(0,2))$ and 
$$
f=((2,2),(1,1),(1,1),(1,0),(1,0),(1,0),(1,0),(0,1),(0,1)),
$$ 
we obtain $|e|=|f|$ and
\begin{align*}
F_{e,f}(z,2^2z)
&=\sum_{k,m\geq 0}\frac{(4k+4m)!}{(2k+2m)!((k+m)!)^2}\cdot\frac{((2k)!)^2}{(k!)^4}\cdot\frac{((2m)!)}{(m!)^2}2^{2m}z^{k+m}\\
&=\sum_{n=0}^{\infty}z^n\sum_{k+m=n}\frac{(4n)!}{(2n)!(n!)^2}2^{2m}\binom{2k}{k}^2\binom{2m}{m}=F(z).
\end{align*}

The solution $F(z)$ is therefore a specialization of $F_{e,f}(z,w)$. We will now prove that $\Delta_{e,f}\geq 1$ on $\mathcal{D}_{e,f}$.

For all $(x,y)\in\mathcal{D}_{e,f}$, we have 
\begin{equation}\label{deltapart}
\Delta_{e,f}(x,y)=\lfloor 4x+4y\rfloor+2\lfloor 2x\rfloor+\lfloor 2y\rfloor-\lfloor 2x+2y\rfloor-2\lfloor x+y\rfloor, 
\end{equation}
because $x$ and $y$ lie in $[0,1[$. According to the definition of $\mathcal{D}_{e,f}$, at least one of the floor functions in \eqref{deltapart} must be greater than or equal to $1$. If $2x+2y<1$, then we have $\Delta_{e,f}(x,y)\geq 1$. Let us assume that $2x+2y\geq 1$. Thus we have 
$$
\lfloor 4x+4y\rfloor\geq 2\lfloor 2x+2y\rfloor\geq 1+\lfloor 2x+2y\rfloor,
$$ 
so that if $x+y<1$, then $\Delta_{e,f}(x,y)\geq 1$. On the other hand, if $x+y\geq 1$, then $\lfloor 2x\rfloor\geq 1$ or $\lfloor 2y\rfloor\geq 1$, and since  
$$
\lfloor 4x+4y\rfloor\geq \lfloor 2x+2y\rfloor+2\lfloor x+y\rfloor, 
$$
we obtain $\Delta_{e,f}(x,y)\geq 1$. Thus, according to Theorem \ref{critère}, we have $q_{e,f,1}(z,4z)\in z\mathbb{Z}[[z]]$ and $q_{e,f,2}(z,4z)\in 4z\mathbb{Z}[[z]]$. We will now prove that the $q$-parameter associated with operator \eqref{opérateur} is equal to $q_{e,f,1}(z,4z)$.
\medskip

Let us write $G(z)$ as the power series without constant term such that $G(z)+\log(z)F(z)$ is canceled by operator \eqref{opérateur}. In order to determine the power series $G(z)$ we use the Frobenius method presented in \cite{Schwarz}. 

For all $r\in\mathbb{C}$, $|r|\leq 1/4$, and all $n,k\in\mathbb{N}$, we set
$$
h_{n,k}(r):=2^{2k}\left(\frac{\Gamma(1+2(n+r-k))}{\Gamma(1+n+r-k)^2}\right)^2\binom{2k}{k}\quad\textup{if}\;r\neq 0,
$$
$$
h_{n,k}(0):=2^{2k}\left(\frac{\Gamma(1+2(n-k))}{\Gamma(1+n-k)^2}\right)^2\binom{2k}{k}\quad\textup{if}\;k\leq n
$$ 
and $h_{n,k}(0)=0$ if $k\geq n+1$.

The function $\Gamma$ is meromorphic on $\mathbb{C}\setminus\mathbb{Z}_{\leq 0}$ and has a simple pole in all nonpositive integer. Hence, the functions $h_{n,k}$ are analytic on $\{r\in\mathbb{C}\,:\,|r|\leq 1/4\}$ and, when $k\geq n+1$, the functions $h_{n,k}$ have a zero of order $2$ in $r=0$. 

For all $n\in\mathbb{N}$ and $r\in\mathbb{C}$, $|r|\leq 1/4$, we set
$$
c_n(r):=\frac{\Gamma(1+4n+4r)}{\Gamma(1+n+r)^2\Gamma(1+2n+2r)}\sum_{k=0}^{\infty}h_{n,k}(r).
$$
\medskip

Let us prove that the series $c_n(r)$ is well defined. We recall Euler's reflection formula: for all $z\in\mathbb{C}\setminus\mathbb{Z}$, we have
\begin{equation}\label{formule des compléments}
\Gamma(1-z)\Gamma(z)=\frac{\pi}{\sin(\pi z)}.
\end{equation}
We will also use the property that, for all $\alpha\in]0,\pi[$, when $|\arg(z)|<\pi-\alpha$ and $z\rightarrow\infty$, we have $\Gamma(z)\sim e^{-z}z^{z-1/2}\sqrt{2\pi}$. Particularly, there exists a constant $\mathcal{K}>0$ such that, if $\Re(z)>0$ and $|z|>\mathcal{K}$, then
\begin{equation}\label{Asympt}
\frac{1}{2}\left|e^{-z}z^{z-1/2}\sqrt{2\pi}\right|\leq\left|\Gamma(z)\right|\leq 2 \left|e^{-z}z^{z-1/2}\sqrt{2\pi}\right|.
\end{equation}
Let us fix $n\in\mathbb{N}$. There exists a constant $\mathcal{K}'> n+1$ such that, for all $k\geq\mathcal{K}'$ and $r\in\mathbb{C}$, $|r|\leq 1/4$, we have $|k-n-r|\geq \mathcal{K}$. Thereby, for all $k\in\mathbb{N}$, $k\geq\mathcal{K}'$ and all $r\in\mathbb{C}\setminus\{0\}$, $r\leq 1/4$, following \eqref{formule des compléments}, we get 
\begin{equation}\label{grim1}
\Gamma(1+2(n+r-k))=\frac{\pi}{\sin(2\pi(k-n-r))\Gamma(2(k-n-r))}
\end{equation}
and
\begin{equation}\label{grim2}
\Gamma(1+n+r-k)=\frac{\pi}{\sin(\pi(k-n-r))\Gamma((k-n-r))}.
\end{equation}
Applying \eqref{Asympt} to \eqref{grim1} and \eqref{grim2}, we respectively obtain
$$
|\Gamma(1+2(n+r-k))|
\leq \left|\frac{\sqrt{2\pi} e^{2(k-n-r)}}{\sin(2\pi r)(2(k-n-r))^{2(k-n-r)-1/2}}\right|
$$
and
$$
|\Gamma(1+n+r-k)|\geq\left|\frac{\sqrt{\pi} e^{k-n-r}}{2\sqrt{2}\sin(\pi r)(k-n-r)^{k-n-r-1/2}}\right|.
$$
Thus we have
\begin{equation}\label{eqGamma}
\left|\frac{\Gamma(1+2(n+r-k))}{\Gamma(1+n+r-k)^2}\right|^2\leq \left|\frac{2^5\sin^2(\pi r)}{\pi\cos^2(\pi r)2^{4(k-n-r)}(k-n-r)}\right|.
\end{equation}
Furthermore, according to \eqref{Asympt}, for all $k\in\mathbb{N}$, $k\geq\mathcal{K}'$, we have
\begin{equation}\label{eqGamma2}
\binom{2k}{k}=\frac{2\Gamma(2k)}{k\Gamma(k)^2}\leq 8\frac{2^{2k}}{\sqrt{\pi k}}.
\end{equation}
Therefore, following \eqref{eqGamma} and \eqref{eqGamma2}, for all $r\in\mathbb{C}$, $|r|\leq 1/4$, and all $k\in\mathbb{N}$, $k\geq \mathcal{K}'$, we have
\begin{equation}\label{equivalentsim}
|h_{n,k}(r)|\leq 2^8\left|\frac{16^{n+r}\sin^2(\pi r)}{\cos^2(\pi r)\pi^{\frac{3}{2}}(k-n-r)\sqrt{k}}\right|.
\end{equation}
There exists a constant $C>0$ such that, for all $r\in\mathbb{C}$, $|r|\leq 1/4$, we have 
$$
|\sin^2(\pi r)/\cos^2(\pi r)|\leq C,\quad \left|16^{n+r}\right|\leq 2\cdot16^n\quad\textup{and}\quad |k-n-r|\geq k-n-1.
$$ 
Thus, for all $r\in\mathbb{C}$, $|r|\leq 1/4$, and all $k\in\mathbb{N}$, $k\geq\mathcal{K}'$, we have
$$
|h_{n,k}(r)|\leq 2^9\frac{C16^{n}}{\pi^{\frac{3}{2}}(k-n-1)\sqrt{k}}.
$$
Thus the series $c_n(r)$ is uniformly convergent on $\{r\in\mathbb{C}\,:\,|r|\leq 1/4\}$ and defines an analytic function. 
\medskip

Let us write $P_2(X):=X^4$, $P_1(X):=-2^4(4X+1)(4X+3)(8X^2+8X+3)$ and
$$
P_0(X):=2^{12}(4X+1)(4X+3)(4X+5)(4X+7),
$$
so that operator \eqref{opérateur} is written like $\mathcal{L}=P_2(\theta)+zP_1(\theta)+z^2P_0(\theta)$. 
\medskip

First, we prove that, for all $r\in\mathbb{C}$, $|r|\leq 1/4$, and all $n\in\mathbb{N}$, we have 
\begin{equation}\label{larec2965}
P_2(n+r+2)c_{n+2}(r)+P_1(n+r+1)c_{n+1}(r)+P_0(n+r)c_n(r)=0.
\end{equation}

If $r=0$, then, according to \cite{Tables}, $F(z)=\sum_{n=0}^{\infty}c_n(0)z^n$ is canceled by $\mathcal{L}$ thus we have \eqref{larec2965}. Let us assume that $r\neq 0$. In order to prove \eqref{larec2965}, we apply the \textit{Zeilberger} procedure from Maple $12$ to the sequence 
$$
(b_k(n,r))_{k\geq 0}:=\left(\frac{\Gamma(1+4n+4r)}{\Gamma(1+n+r)^2\Gamma(1+2n+2r)}2^{2k}\left(\frac{\Gamma(1+2(n+r-k))}{\Gamma(1+n+r-k)^2}\right)^2\binom{2k}{k}\right)_{k\geq 0}.
$$ 
For a hypergeometric term $T(n,k)$ this procedure constructs a sequence $(d_k)_{k\geq 0}$ and an operator $\mathcal{R}=P_v(n)\delta^v+\cdots+P_1(n)\delta+P_0(n)$ such that $\mathcal{R}T(n,k)=d_{k+1}-d_k$, where $P_i(X)\in\mathbb{C}[[X]]$ and $\delta$ is the shift operator $\delta T(n,k)=T(n+1,k)$. In our case, we obtain an explicit sequence $(d_k)_{k\geq 0}$ such that, for all $n,k\in\mathbb{N}$, we have
\begin{equation}\label{transitZeilberger}
P_2(n+r+2)b_k(n+2,r)+P_1(n+r+1)b_k(n+1,r)+P_0(n+r)b_k(n,r)=d(k+1)-d(k),
\end{equation}
with $d(0)=0$ and (quick calculation)
$$
d(k)=O\left(\frac{k4^k\Gamma(2(n+r-k)+1)^2}{\Gamma(n+r-k+1)^4}\binom{2k}{k}\right)=O\left(\frac{1}{\sqrt{k}}\right),
$$
when $k\rightarrow+\infty$. Summing identity \eqref{transitZeilberger} for $k$ from $0$ to $+\infty$ and using that $d(0)=0$ and $d(k)\underset{k\rightarrow+\infty}{\rightarrow}0$, we get \eqref{larec2965}.
\medskip

Therefore, writing $\tilde{F}(z,r):=\sum_{n=0}^{\infty}c_n(r)z^{n+r}$, we have
\begin{equation}\label{eqdiffinter}
\mathcal{L}\tilde{F}(z,r)=P_2(r)c_0(r)z^r+\big(P_2(r+1)c_1(r)+P_1(r)c_0(r)\big)z^{1+r}.
\end{equation}
\medskip

Let us prove that $\left.\frac{\partial}{\partial r}\mathcal{L}\tilde{F}(z,r)\right|_{r=0}=0$. The series $c_n(r)$ is analytic on $\{r\in\mathbb{C}\,:\,|r|\leq 1/4\}$ and its derivative is obtained by differentiating term by term. When $k\geq n+1$, the functions $h_{n,k}(r)$ are analytic in a neighborhood of $0$ and have a zero of order $2$ in $0$. For all $k\geq n+1$, we obtain 
$$
\left.\frac{\partial}{\partial r}\left(\frac{\Gamma(1+4n+4r)}{\Gamma(1+n+r)^2\Gamma(1+2n+2r)}\left(\frac{\Gamma(1+2(n+r-k))}{\Gamma(1+n+r-k)^2}\right)^2\right)\right|_{r=0}=0.
$$

On the other hand, if $m\in\mathbb{N}$, $m\geq 1$, then we have $\Gamma'(m)=\Gamma(m)(H_{m-1}-\gamma)$, where $\gamma$ is Euler's constant. Hence, for all $k\in\{0,\dots,n\}$, we obtain
\begin{multline*}
\left.\frac{\partial}{\partial r}\left(\frac{\Gamma(1+4n+4r)}{\Gamma(1+n+r)^2\Gamma(1+2n+2r)}\left(\frac{\Gamma(1+2(n+r-k))}{\Gamma(1+n+r-k)^2}\right)^2\right)\right|_{r=0}\\
=\frac{\Gamma(1+4n)}{\Gamma(1+n)^2\Gamma(1+2n)}\left(\frac{\Gamma(1+2(n-k))}{\Gamma(1+n-k)^2}\right)^2\\
\times\left(4H_{4n}-2H_n-2H_{2n}+4H_{2(n-k)}-4H_{n-k}\right).
\end{multline*}

Thus, for all $n\in\mathbb{N}$, we have
\begin{multline}\label{formule c_n'(0)}
c_n'(0)=\frac{\Gamma(1+4n)}{\Gamma(1+n)^2\Gamma(1+2n)}\sum_{k=0}^n\left(\frac{\Gamma(1+2(n-k))}{\Gamma(1+n-k)^2}\right)^2\\
\times\left(4H_{4n}-2H_n-2H_{2n}+4H_{2(n-k)}-4H_{n-k}\right)
\end{multline}
Particularly, we obtain 
$$
\left.\frac{d}{dr}\left(P_2(r)c_0(r)z^r\right)\right|_{r=0}=\left.\frac{d}{dr}\left(r^4c_0(r)z^r\right)\right|_{r=0}=0
$$
and, a simple calculation \textit{via} Maple $12$ leads to
$$
\left.\frac{d}{dr}\left((P_2(r+1)c_1(r)+P_1(r)c_0(r))z^{1+r}\right)\right|_{r=0}=0.
$$
Therefore, we have $\left.\frac{\partial}{\partial r}\left(\mathcal{L}\tilde{F}(z,r)\right)\right|_{r=0}=0$.
\medskip

Since the sequence $(c_n(r))_{n\geq 0}$ satisfies recurrence relation \eqref{larec2965}, we can follow Section $16.2$ from \cite{Schwarz} and we obtain that there exists $R>0$ such that, for all $r\in\mathbb{C}$, $|r|\leq 1/4$, the power series $\tilde{F}(z,r)$ in $z$ has a radius of convergence at least equal to $R$. Furthermore, if $|z|<R$, then $F(z,r)$ is derivable with respect to $r$ and 
\begin{multline*}
\left.\frac{\partial}{\partial r}\tilde{F}(z,r)\right|_{r=0}=\log(z)F(z)\\
+\sum_{n=0}^{\infty}z^n\frac{(4n)!}{(n!)^2(2n)!}\sum_{k=0}^n2^{2k}\binom{2(n-k)}{n-k}^2\binom{2k}{k}(4H_{4n}-2H_n-2H_{2n}+4H_{2(n-k)}-4H_{n-k}).
\end{multline*}

As the operator $\frac{\partial}{\partial r}$ commutes with $\mathcal{L}$, we obtain 
$$
\left.\mathcal{L}\frac{\partial}{\partial r}\tilde{F}(z,r)\right|_{r=0}=\left.\frac{\partial}{\partial r}(\mathcal{L}\tilde{F}(z,r))\right|_{r=0}=0.
$$
Thereby, $\left.\frac{\partial}{\partial r}\tilde{F}(z,r)\right|_{r=0}=G_{e,f,1}(z,4z)+\log(z)F(z)$ is canceled by $\mathcal{L}$ and, according to the uniqueness of $G(z)$, we have $G(z)=G_{e,f,1}(z,4z)$. The $q$-parameter associated with operator \eqref{opérateur} is thus 
$$
q(z)=z\exp\big(G_{e,f,1}(z,4z)/F_{e,f}(z,4z)\big)=q_{e,f,1}(z,4z)\in z\mathbb{Z}[[z]].
$$
So in Case $30$, the $q$-parameter is a specialization of a canonical coordinate. We did not verify in detail the $143$ cases cited above but it seems that the presented method for Case $30$ can prove in many cases that the $q$-parameter associated with the operator is a specialization of a canonical coordinate and thus, when $\Delta_{e,f}\geq 1$ on $\mathcal{D}_{e,f}$, all its Taylor coefficients are integers. It would be interesting to have a more general method to prove this.

\address{E. Delaygue, Institut Fourier, CNRS et Université Grenoble 1, 100 rue des Maths, BP 74, 38402 Saint-Martin-d'Hères cedex, France. Email : Eric.Delaygue@ujf-grenoble.fr}


\begin{thebibliography}{1}\label{sec:biblio} 
\addcontentsline{toc}{section}{Bibliographie}
\bibitem{Tables} G. Almkvist, C. van Enckevort, D. van Straten, W. Zudilin, \textit{Tables of Calabi--Yau equations}, preprint (2010), arXiv:math/0507430v2 [math.AG].
\bibitem{Batyrev} V. V. Batyrev, D. van Straten, \textit{Generalized Hypergeometric Functions and Rational Curves on Calabi-Yau Complete Intersections in Toric Varieties}, Commun Math. Phys. \textbf{168} (1995), 493--533.
\bibitem{Delaygue} E. Delaygue, \textit{Critère pour l'intégralité des coefficients de Taylor des applications miroir}, J. Reine Angew. Math. (to appear), published online at http://dx.doi.org/10.1515/CRELLE.2011.094 
\bibitem{Dwork} B. Dwork, \textit{On $p$-adic differential equations IV : generalized hypergeometric functions as $p$-adic functions in one variable}, Annales scientifiques de l'E. N. S. $4^e$ série, tome \textbf{6}, numéro 3 (1973), p. 295-316.
\bibitem{Hosono} S. Hosono, A. Klemm, S. Theisen and S.-T. Yau, Mirror symmetry, \textit{Mirror map and applications to
complete intersection Calabi--Yau spaces}, Nuclear Phys. B \textbf{433}, no. 3 (1995), 501--552.
\bibitem{Kobliz} N. Kobliz, \textit{$p$-Adic Numbers, $p$-Adic Analysis, and Zeta-functions}, Springer-Verlag, Heidelberg, 1977.
\bibitem{Tanguy} C. Krattenthaler, T. Rivoal, \textit{On the integrality of the Taylor coefficients of mirror maps}, Duke Math. J., \textbf{151}.2 (2010), 175--218.
\bibitem{Tanguy 2} C. Krattenthaler, T. Rivoal, \textit{Multivariate $p$-adic formal congruences and integrality of Taylor coefficients of mirror maps}, preprint (2008), arXiv:0804.3049v3 [math.NT], to appear in series \textit{Séminaire et Congrès} of the SMF. Proceedings of the conference \textit{Théorie galoisiennes et arithmétiques des équations différentielles} (CIRM, september 2009).
\bibitem{Lang} S. Lang, \textit{Cyclotomic Fields, I, II}, Combined 2nd edition, vol. \textbf{121}, Graduate Texts in Math., Springer-Verlag, New York, 1990.
\bibitem{Landau} E. Landau, \textit{Sur les conditions de divisibilité d'un produit de factorielles par un autre}, collected works, \textbf{I}, page 116. Thales-Verlag (1985).
\bibitem{Lian 2} B. H. Lian, S. T. Yau, \textit{Mirror Maps, Modular Relations and Hypergeometric Series I}. arXiv:hep-th/9507151v1. Paru sous le titre : \textit{Integrality of certain exponential series. Algebra and geometry} (Taipei, 1995), 215--227, Lect. Algebra Geom., \textbf{2}, Int. Press, Cambridge, MA (1998). (Reviewer : Nobuo Tsuzuki).
\bibitem{Lian} B. H. Lian, S. T. Yau, \textit{Arithmetic properties of mirror map and quantum coupling}, Comm. Math. Phys. 176, \textbf{1} (1996), 163--191.
\bibitem{Stienstra} J. Stienstra, \textit{GKZ Hypergeometric Structures}, Arithmetic and geometry around hypergeometric functions, R.-P. Holzapfel, A. Muhammed Uluda¢g and M. Yoshida (eds.), Progr. Math., \textbf{260},
Birkh¨auser, Basel (2007), 313--371.
\bibitem{Schwarz} M. Yoshida, \textit{Fuchsian differential equations}, Aspects of Mathematics \textbf{11}, Vieweg  (1987)
\bibitem{Zudilin} V. V. Zudilin, \textit{Integrality of power expansions related to hypergeometric series}, Mathematical Notes, vol. \textbf{71}, no. 5 (2002), 604--616.
\end{thebibliography}
\end{document}